\documentclass[12pt,twoside]{book}
\usepackage{amsmath}
\usepackage{a4wide,amsthm, amsfonts, amssymb, latexsym, epsfig,cite, fancyhdr, mathrsfs }
\usepackage{graphicx}
\usepackage{float}
\usepackage[numbers]{natbib}

\usepackage{amsxtra}
\usepackage{mathrsfs}
\usepackage[mathscr]{eucal}
\usepackage{graphicx}
\usepackage[titles]{tocloft}
\usepackage{calligra}
\usepackage[T1]{fontenc}
\usepackage{tikz}
\usepackage{ifthen}
\usepackage{pgfplots}
\newtheorem{thm}{Theorem}[chapter]

\newtheorem{cor}[thm]{Corollary}
\newtheorem{lem}[thm]{Lemma}
\newtheorem{exam}[thm]{Example}
\newtheorem{prop}[thm]{Proposition}
\theoremstyle{definition}

\theoremstyle{remark}
\newtheorem{rem}[thm]{Remark}

\numberwithin{equation}{chapter}
\newcommand{\si}{\sigma}
\newcommand{\de}{\delta}
\newcommand{\om}{\omega}
\newcommand{\ga}{\gamma}
\newcommand{\mbb}{\mathbb}
\newcommand{\ra}{\rightarrow}

\newcommand{\pa}{\partial}
\newcommand{\ov}{\overline}
\newcommand{\sm}{\setminus}
\newcommand{\ep}{\epsilon}
\newcommand{\no}{\noindent}

\newcommand{\Om}{\Omega}

\newcommand{\ti}{\tilde}
\newcommand{\la}{\lambda}

\numberwithin{equation}{chapter}
\RequirePackage{filecontents}        

\begin{document}
\thispagestyle{empty}
\begin{center}
{\LARGE {\bf Dynamical properties of families of holomorphic mappings }}\\
\end{center}
\vspace{0.8in}
\begin{center}
{\large  A Dissertation \\
submitted in partial fulfilment \\
\vspace{0.04in}
of the requirements for the award of the  }\\
\vspace{0.04in}
{\large{ degree of}} \\
\vspace{0.06in}
{\calligra \Large Doctor of Philosophy} \\
\vspace{0.75in}
{\em by}\\
\vspace{0.05in}
{\calligra\large Ratna Pal}\\
\vspace{1.5in}
\includegraphics[scale=0.3]{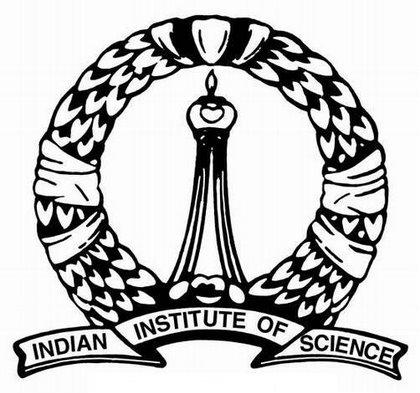} \\
\vspace{0.2in}
{\large{Department of Mathematics}}\\
{\large{Indian Institute of Science}} \\
{\large{Bangalore \,-\, 560012}}\\
{July 2015}
\end{center}
\newpage

\thispagestyle{empty}
\cleardoublepage
\pagenumbering{roman}

\addcontentsline{toc}{chapter}{Declaration}
\chapter*{Declaration}

\vspace*{.80in}

I hereby declare that the work reported in this thesis is entirely
original and has been carried out by me under the supervision of
Dr. Kaushal Verma at the Department of Mathematics, Indian
Institute of Science, Bangalore. I further declare that this work
has not been the basis for the award of any degree, diploma,
fellowship, associateship or similar title of any University or
Institution.
\vspace*{.5in}

\begin{flushright}
Ratna Pal\\
S. R. No. 6910-110-101-07964
\end{flushright}
Indian Institute of Science,\\
Bangalore,\\
July, 2015.
\smallskip

\begin{flushright}
 Dr. Kaushal Verma \\
 (Research advisor)
\end{flushright}

\newpage

\thispagestyle{empty}
\cleardoublepage
\chapter*{}
\begin{center}
{\bf{
To my Father}}
\end{center}

\thispagestyle{empty}
\cleardoublepage
\chapter*{Acknowledgement}

Completing a Ph.D. is a marathon task and writing a thesis is the end of it. I take this opportunity to express my deeply-felt thanks to all of those who helped me moving.

\medskip

First of all, I would like to give my sincere thanks to my advisor Prof. Kaushal Verma, not only for his warm encouragement and thoughtful guidance, but also for listening to me with immense patience whenever I had any ideas. I could not have imagined a better advisor for my Ph.D. 

\medskip

My heart-felt thanks go to my student colleagues for providing me a stimulating and fun-filled environment. My special thanks go to Hari for his distinguished helping nature and encouragement. Sayani deserves a special mention, numerous discussions with her really helped me to understand the subject better. Soma, Atreyee, Jyoti, Lakshmi, Soumitra, Amar, Monojit, Somnath, Sanjay, Sourav, Samarpita, Papri -- you people are amazing, you made my work place as cozy as home. 

\medskip

Words cannot express how deeply grateful I am to my father whose firm and kind-hearted personality always taught me to be steadfast, and my dear mother who was my friend, teacher and dedicated a lot of her time and energy to shape my life -- I wish you were around. A huge thanks to my elder sister whose motherly love and care always gave me a shelter.

\thispagestyle{empty}
\cleardoublepage
\addcontentsline{toc}{chapter}{Abstract}
\chapter*{Abstract}

In the first part of the thesis,  we study some dynamical properties of skew products of H\'{e}non maps of $\mbb C^2$ that are fibered over a compact metric space $M$. The problem reduces to understanding the dynamical behavior of the composition of a pseudo-random sequence of H\'{e}non mappings. In analogy with the dynamics of the iterates of a single H\'{e}non map, it is possible to construct fibered Green functions that satisfy suitable invariance properties and the corresponding stable and unstable currents. Further, it is shown that the successive pullbacks of a suitable current by the skew H\'{e}non maps converge to a multiple of the fibered stable current.

\medskip

Second part of the thesis generalizes most of the above-mentioned results for a completely random sequence of H\'{e}non maps. In addition, for this random system of H\'{e}non maps, we introduce the notion of average Green functions and average Green currents  which carry many typical features of the classical Green functions and Green currents. 

\medskip

Third part consists of some results about the global dynamics of a special class of skew maps. To prove these results, we use the knowledge of dynamical behavior of pseudo-random sequence of H\'{e}non maps widely. We show that the global skew map is strongly mixing for a class of invariant measures and also provide a lower bound  on the topological entropy of the skew product.   

\medskip

We conclude the thesis by studying another class of maps which are skew products of holomorphic endomorphisms of $\mbb P^k$ fibered over a compact base. We define the fibered Fatou components and show that they are pseudoconvex and Kobayashi hyperbolic.

\thispagestyle{empty}

\cleardoublepage \addcontentsline{toc}{chapter}{Contents}
\tableofcontents \cleardoublepage

\pagestyle{fancy}
\renewcommand{\chaptermark}[1]{\markboth{\textsl{\thechapter.\ #1}}{}}
\renewcommand{\sectionmark}[1]{\markright{\textsl{\thesection.\ #1}}}
\fancyhf{}
\renewcommand{\headrulewidth}{.05pt}
\fancyhead[LE]{\thepage} \fancyhead[RO]{\thepage}
\fancyhead[RE]{\leftmark} \fancyhead[LO]{\rightmark}

\fancypagestyle{plain}{
\fancyhead{} 
\renewcommand{\headrulewidth}{0pt} 
 }

%
\pagenumbering{arabic} \setcounter{chapter}{0}
\chapter{Introduction}

A holomorphic automorphism $F:(x,y)\mapsto (f(x,y),g(x,y))$ of $\mathbb{C}^2$ is called a polynomial automorphism, if both $f$ and $g$ are polynomials in the variables $x, \ y$. In \cite{FM} Friedland and Milnor showed that any polynomial automorphism of $\mathbb{C}^2$, upto conjugation by a polynomial automorphism, is one of the followings:
\begin{itemize}
\item
an affine map
\[
(x,y)\mapsto (ax+by+c,a'x+b'y+c') 
\]
with $ab'-a'b\neq 0$;
\item
an elementary map
\[
(x,y)\mapsto (ax+b,sy+p(x))
\]
with $as\neq 0$ and $p(x)$ a polynomial in the single variable $x$;
\item
a  composition of generalized H\'{e}non maps of the form
\[
(x,y)\mapsto(y,p(y)-\delta x)
\]
with $\delta\neq 0$ and $p(y)$ a polynomial in the single variable $y$ of degree greater than or equal to $2$.
\end{itemize}

In contrast to the first two classes of automorphisms which exhibit simple dynamical behavior, the class of generalized H\'{e}non maps possesses very rich dynamics. So to understand dynamics of polynomial automorphisms of $\mathbb{C}^2$, we can restrict ourselves to the class of generalized H\'{e}non maps. In a series of papers (\cite{BS1}, \cite{BS2}, \cite{BS3}), Bedford and Smillie investigated various dynamical properties of these maps.

\medskip

In the first part of this thesis, we study some dynamical properties of skew products of H\'{e}non maps of $\mathbb{C}^2$ that are fibered over a compact metric space $M$. The problem reduces to understanding the dynamical behavior of the composition of a pseudo-random sequence of H\'{e}non maps. In the subsequent chapters, we shall see that it is possible to get many results in the pseudo-random case in analogy with the dynamics of the iterates of a single H\'{e}non map.  

\medskip

\no  We will consider families of the form $H : M \times
\mbb C^2 \ra M \times \mbb C^2$, defined by
\begin{equation}
H(\la, x, y) = (\sigma(\la), H_{\la}(x, y)) \label{1}
\end{equation}
where $M$ is an appropriate parameter space, $\sigma$ is a self map of $M$ and for each $\la \in M$, the map
\begin{equation}
H_{\la}(x, y) = H_{\la}^{(m)} \circ H_{\la}^{(m-1)} \circ \cdots \circ H_{\la}^{(1)}(x, y) \label{1.5}
\end{equation}
where for every $1 \le j \le m$,
\[
H_{\la}^{(j)}(x, y) = (y, p_{j, \la}(y) - a_{j}(\la) x)
\]
is a generalized H\'{e}non map with $p_{j, \la}(y)$ a monic polynomial of degree $d_j \ge 2$ whose coefficients and $a_{j}(\la)$ are functions on $M$. The degree of $H_{\la}$ is $d =
d_1d_2 \cdots d_m$ which does not vary with $\la$. Here $M$ is a compact metric space
and $\sigma$, $a_j$ and the coefficients of $p_{j, \la}$ are continuous functions on $M$.
Further, $a_j$ is assumed to be a non-vanishing function on $M$. We are interested
in studying the ergodic properties of such a family of mappings. Part of the reason for this choice stems from the
Forn{\ae}ss-Wu classification (\cite{FW}) of polynomial automorphisms of $\mbb C^3$ of degree at most $2$, according to which any such map is affinely conjugate to one of the followings:

\begin{enumerate}
\item[(a)] an affine automorphism;
\item[(b)] an elementary polynomial automorphism of the form
\[
E(x, y, z) = (P(y, z) + ax, Q(z) + by, cz + d)
\]
where $P, Q$ are polynomials with $\max \{\deg (P), \deg (Q) \} = 2$ and $abc \not= 0$; or
\item[(c)] to one of the followings:
\begin{itemize}
\item $H_1(x, y, z) = (P(x, z) + ay, Q(z) + x, cz + d)$
\item $H_2(x, y, z) = (P(y, z) + ax, Q(y) + bz, y)$
\item $H_3(x, y, z) = (P(x, z) + ay, Q(x) + z, x)$
\item $H_4(x, y, z) = (P(x, y) + az, Q(y) + x, y)$
\item $H_5(x, y, z) = (P(x, y) + az, Q(x) + by, x)$
\end{itemize}
\no
where $P, Q$ are polynomials with $\max \{ \deg(P), \deg(Q) \} = 2$ and $abc \not= 0$.
\end{enumerate}

\no The six classes in (b) and (c) put together were studied in \cite{CF} and \cite{CG} where suitable Green functions and associated invariant measures were constructed for them. As
observed in \cite{FW}, several maps in (c) are in fact families of H\'{e}non maps for special values of the parameters $a, b, c$ and for judicious choices of the
polynomials $P, Q$. For instance, if $Q(z) = 0$ and $P(x, z) = x^2 + \cdots$, then $H_1(x, y, z) = (P(x, z) + ay, x, z)$ which is conjugate to
\[
(x, y, z) \mapsto (y, P(y, z) + ax, cz + d) = (y, y^2 + \cdots + ax, cz + d)
\]
by the inversion $\tau_1(x, y, z) = (y, x, z)$. Here $\si(z) = cz + d$. Similarly, if $a = 1, P(y, z) = 0$ and $Q$ is a quadratic polynomial, then $H_2(x, y, z) = (x, Q(y) + bz, y)$
which is conjugate to
\[
(x, y, z) \mapsto (x, z, Q(z) + by) = (x, z, z^2 + \cdots + by)
\]
by the inversion $\tau_3(x, y, z) = (x, z, y)$. Here $\si(x) = x$ and finally, if $b = 1, Q(x) = 0$ and $P(x, y) = x^2 + \cdots$, then $H_5(x, y, z) = (P(x, y) + az, y, x)$ which is
conjugate to
\[
(x, y, z) \mapsto (z, y, P(z, y) + ax) = (z, y, z^2 + \cdots + ax)
\]
by the inversion $\tau_5(x, y, z) = (z, y, x)$ where again $\si(y) = y$. All of these are examples of the kind described in (\ref{1})
 with $M = \mbb C$. In the first example, if
$c \not= 1$ then an affine change of coordinates involving only the $z$-variable can make $d = 0$ and if further $\vert c \vert \le 1$, then we may take a closed disc around the origin
in $\mbb C$ which
will be preserved by $\sigma(z) = cz$. This provides an example of a H\'{e}non family that is fibered over a compact base $M$. Further, since the parameter mapping $\sigma$ in the last
two examples is just the identity, we may restrict it to a closed ball to obtain more examples of the case when $M$ is compact.

\medskip

The maps considered in (\ref{1}) are in general $q$-regular, for some $q \ge 1$, in the sense of Guedj--Sibony (\cite{GS}) as the following example shows. Let $\mathcal H : \mbb C^3 \ra
\mbb C^3$ be given by
\[
\mathcal H(\la, x, y) = (\la, y, y^2 - ax), a \not= 0
\]
which in homogeneous coordinates becomes
\[
\mathcal H([\la : x : y : t]) = [\la t : yt : y^2 - axt : t^2].
\]
The indeterminacy set of this map is $I^+ = [\la : x : 0 : 0]$ while that for $\mathcal H^{-1}$ is $I^{-1} = [\la : 0 : y : 0]$. Thus $I^+ \cap I^- = [1 : 0 : 0 : 0]$ and it can be
checked that $X^+ = \ov{\mathcal H \big( (t = 0) \sm I^+ \big)} = [0: 0: 1: 0]$ which is disjoint from $I^+$. Also, $X^- = \ov {\mathcal H^- \big( (t = 0) \sm I^- \big)} =
[0:1:0:0]$ which is disjoint from $I^-$. All these observations imply that $\mathcal H$ is $1$-regular in the sense of \cite{GS}. Further, $\deg(\mathcal H)
= \deg(\mathcal H^{-1}) = 2$. This global view point does have several advantages as the results in \cite{GS}, \cite{G} show. However, thinking of (\ref{1}) as a family of maps was
seconded by the hope that the methods of Bedford--Smillie (\cite{BS1}, \cite{BS2} and
\cite{BS3}) and Forn{\ae}ss--Sibony \cite{FS} that were developed to handle the case of a single generalized H\'{e}non map would be amenable to this situation -- in fact, they are to a
large extent. Finally, in view of the systematic treatment of families of rational maps of the sphere by Jonsson (see \cite{JM}, \cite{J}), considering families of 
H\'{e}non maps appeared to be a natural next choice. Several pertinent remarks about the family $H$ in (1.1) with $\sigma(\lambda)=\lambda$ can be found in \cite{DS}.

\medskip

For $n \ge 0$, let
\[
H_{\la}^{\pm n} = H_{\si^{n-1}(\la)}^{\pm 1} \circ \cdots \circ H_{\si(\la)}^{\pm 1} \circ H_{\la}^{\pm 1}.
\]
Note that $H_{\la}^{+n}$ is the second coordinate of the $n$-fold iterate of $H(\la, x, y)$. Furthermore,
\[
(H_{\la}^{+n})^{-1} = H_{\la}^{-1} \circ H_{\si(\la)}^{-1} \circ \cdots \circ H_{\si^{n-1}(\la)}^{-1} \not= H_{\la}^{-n}
\]
and
\[
(H_{\la}^{-n})^{-1} = H_{\la} \circ H_{\si(\la)} \circ \cdots \circ H_{\si^{n-1}(\la)} \not= H_{\la}^{+n}
\]
for $n \ge 2$. The presence of $\si$ creates an asymmetry which is absent in the case of a single H\'{e}non map and which requires the consideration of these maps as will be seen
later. In what follows, no conditions on $\si$ except continuity are assumed unless stated otherwise.

\medskip

The first thing to do is to construct invariant measures for the family $H(\la, x, y)$ that respect the action of
$\sigma$. The essential step toward this is to construct a uniform filtration $V_R$, $V_R^{\pm}$ for the maps $H_\lambda$ where $R>0$ is sufficiently large.

\medskip

For each $\lambda \in M$, the sets $I_\lambda^{\pm}$ of escaping points and the sets $K_\lambda^{\pm}$ of non-escaping points under suitable random iteration determined by $\sigma$ on $M$ are defined as follows:
\[
I_\lambda^{\pm}=\{z\in \mathbb{C}^2: \Vert H_{\la}^{\pm n}(x, y) \Vert \rightarrow \infty \; \text{as} \;
n\rightarrow \infty \},
\]
\[
K_\lambda^{\pm}=\{z\in \mathbb{C}^2: \; \text{the sequence}\;  \{ H_{\la}^{\pm n} (x, y)\}_n \; \text{is bounded}\}.
\]
Clearly, $H_\lambda^{\pm 1}(K_\lambda^{\pm})= K_{\sigma(\lambda)}^{\pm}$ and
$H_\lambda^{\pm 1}(I_\lambda^{\pm})= I_{\sigma(\lambda)}^{\pm}$. Define $K_{\la} = K_{\la}^+ \cap
K_{\la}^-, J_{\la}^{\pm} = \pa K_{\la}^{\pm}$ and $J_{\la} = J_{\la}^+ \cap J_{\la}^-$. For each $\la \in M$ and $n \ge 1$, let
\[
G_{n, \la}^{\pm}(x, y) = \frac{1}{d^n} \log^+ \Vert H_{\la}^{\pm n}(x, y) \Vert
\]
where $\log^+ t=\max \{\log t,0\}$.

\begin{prop}\label{pr1}
The sequence $G_{n, \la}^{\pm}$ converges uniformly on compact subsets of $\mbb C^2$ to the continuous function $G_{\la}^{\pm}$ as $n \ra \infty$ that satisfies
\[
d G_{\la}^{\pm} = G_{\si(\la)}^{\pm} \circ H_{\la}^{\pm 1}
\]
on $\mbb C^2$. Moreover the convergence is independent of $\lambda$. The functions $G_{\la}^{\pm}$ are positive pluriharmonic on $\mbb C^2 \setminus K_{\la}^{\pm}$, plurisubharmonic on $\mbb C^2$ and vanish precisely on $K_{\la}^{\pm}$.
The correspondence $\lambda \mapsto G_\lambda^{\pm}$ is continuous. In case $\si$ is surjective, $G_{\la}^+$ is locally uniformly H\"{o}lder continuous, i.e., for each compact $S
\subset
\mbb C^2$, there exist constants $\tau, C > 0$ such that
\[
\big\vert G_{\la}^+(x, y) - G_{\la}^+(x', y') \big\vert \le C \Vert (x, y) - (x', y') \Vert^{\tau}
\]
for all $(x, y), (x', y') \in S$. The constants $\tau, C$ depend on $S$ and the map $H$ only.
\end{prop}

\no As a result, $\mu_{\la}^{\pm} = dd^c G_{\la}^{\pm}$ are well-defined positive closed $(1, 1)$-currents on $\mbb C^2$ of mass $2\pi$ and hence $\mu_{\la} = \mu_{\la}^+ \wedge \mu_{\la}^-$
defines a measure of finite mass on $\mbb C^2$ whose support is contained in $V_R$ for every $\la \in M$. Moreover, the correspondence $\lambda\mapsto \mu_\lambda$ is continuous. That these
objects are well behaved under the pull-back and push-forward operations by $H_{\la}$ and at the same time respect the action of $\si$ is recorded in the following:

\begin{prop}\label{pr2}
With $\mu_{\la}^{\pm}, \mu_{\la}$ as above, we have
\[
{(H_{\la}^{\pm 1})}^{\ast} \mu_{\si(\la)}^{\pm} = d \mu_{\la}^{\pm} \text{ and }  (H_{\la}^{\pm 1})_{\ast} \mu_{\la}^{\pm} = d^{-1} \mu_{\si(\la)}^{\mp}.
\]
The support of $\mu^{\pm}_{\la}$ equals $J_{\la}^{\pm}$ and the correspondence $\la \mapsto J_{\la}^{\pm}$ is lower semi-continuous. Furthermore, for each $\lambda\in
M$, the pluricomplex Green function of $K_\lambda$ is $\max\{G_\lambda^+, G_\lambda^-\}$, $\mu_\lambda$ is the complex
equilibrium measure of $K_\lambda$ and ${\rm supp}(\mu_\lambda)\subseteq J_\lambda$.

\medskip

In particular, if $\si$ is the identity on $M$, then $(H_{\la}^{\pm 1})^{\ast} \mu_{\la} = \mu_{\la}$.
\end{prop}

\no Let $T$ be a positive closed $(1, 1)$-current in a domain $\Om \subset \mbb C^2$ and let $\psi \in C^{\infty}_0(\Om)$ with $\psi \ge 0$ be such that $\text{supp}(\psi) \cap
\text{supp}(dT) = \phi$. Theorem 1.6 in \cite{BS3} shows that for a single H\'{e}non map $H$ of degree $d$, the sequence $d^{-n} H^{n \ast}(\psi T)$ always converges to $c \mu^+$
where $c = \frac{1}{4\pi^2}\int \psi T \wedge \mu^- > 0$. In the same vein, for each $\la \in M$, let $S_{\la}(\psi, T)$ be the set of all possible limit points of the sequence
$d^{-n}\big( H_{\la}^{+n}\big)^{\ast}(\psi T)$.

\begin{thm}\label{thm1}
Let $T$ and  $\psi$ be as above. Then for all $\lambda\in M$, $S_{\la}(\psi, T)$ is nonempty and each $\ga_{\la} \in S_{\la}(\psi, T)$ is a positive multiple of  $\mu_{\la}^+$. Moreover, 
\begin{itemize}
\item[$(i)$]
$\lim_{n\ra \infty} (d^{-n})\big( H_{\la}^{+n}\big)^{\ast}(T\wedge d\psi)=\lim_{n\ra \infty} (d^{-n})\big( H_{\la}^{+n}\big)^{\ast}(T\wedge d^c\psi)=0$;
\item[$(ii)$]
$\lim_{n\ra \infty} (d^{-n})\big( H_{\la}^{+n}\big)^{\ast}(T\wedge dd^c\psi)=0$.
\end{itemize}
\end{thm}

In general, $S_\lambda(\psi, T)$ may be a large set. However, there are two cases for which it is possible to
determine the cardinality of $S_{\la}(\psi, T)$ and both are illustrated by the examples mentioned earlier. Please look at Remark \ref{des} for a description of $S_\lambda(\psi, T)$ when $\sigma$ is a homemorphism on $M$.

\begin{prop}\label{pr3}
If $\sigma$ is the identity on $M$ or when $\sigma : M \ra M$ is a contraction, i.e., there exists 
 $\la_0 \in M$ such that $\si^n(\la) \ra \la_0$ for all $\la \in M$, the set
$S_{\la}(\psi, T)$ consists of precisely one element. Consequently, in each of these cases there exists a constant $c_{\la}(\psi, T) > 0$ such that
\[
\lim_{n \ra  \infty} d^{-n} \big( H_{\la}^{+n}\big)^{\ast}(\psi T) = c_{\la}(\psi, T) \mu_{\la}^+.
\]
\end{prop}

\medskip

The techniques used to understand the dynamical behavior of a pseudo-random system of H\'{e}non maps, determined by a continuous map $\sigma$, can also be implemented to describe various dynamical aspects of a completely random system of H\'{e}non maps. 

\medskip

Another motivation for studying the dynamics of a completely random system of H\'{e}non maps is \cite{FWe}. In this paper, Forn{\ae}ss and Weickert developed a pluripotential theory for random iterations of holomorphic endomorphisms on $k$-dimensional complex  projective space $\mathbb{P}^k$. The success of this theory comes from the fact that the notion of an exceptional set vanishes as soon as we allow a slight randomness in the system. In this connection, they introduced an average Green function and the corresponding average Green current which interact nicely with each holomorphic endomorphism in the system and also carry many typical features of classical Green functions and Green currents.  

\medskip

It is natural to be inquisitive about the validity of this phenomenon if we start with a random system of regular maps on $\mathbb{P}^k$, which is definitely the next obvious choice in this direction. In this case, the indeterminacy sets of regular maps are the additional troublesome factor. So one can start with the basic examples -- H\'{e}non maps on $\mathbb{C}^2$ which can be viewed as birational maps on $\mathbb{P}^2$. For a given H\'{e}non map $H:\mathbb{C}^2 \ra \mathbb{C}^2$ of degree $d$ defined by 
$$
H(x,y)=(y,p(y)-\delta x)
$$
where $p(y)$ is a monic polynomial of degree $d$ and $\delta\neq 0$, one can consider the following extension to $\mathbb{P}^2$:
\[
[t:x:y] \ra \big[ t^d: yt^{d-1}: t^d \big(p(\frac{y}{t})-\delta\frac{x}{t} \big)\big],
\] 
which for simplicity we still denote by $H$. Note that each of the extended H\'{e}non maps $H$ has a common attracting fixed point $I^-=[0:0:1]$ which is the unique indeterminacy point of $H^{-1}$ and a common repelling fixed point $I^+=[0:1:0]$ which is the unique indeterminacy point of $H$ in $\mathbb{P}^2$. Further, note that the orbit of backward indeterminacy set of $I^+$ and forward indeterminacy set of $I^-$ under random iteration of any sequence of H\'{e}non maps ${\{H_{n}\}}_{n\geq 1}$, are uniformly separated, i.e., 
$I_{\infty}^+ \cap I_{\infty}^-=\phi$
where $$
I_{\infty}^+=\bigcup_{n=0}^\infty \big (H_{n}^{-1}\circ \cdots \circ H_{1}^{-1}\big)(I^+)=I^+
$$
and 
$$
I_{\infty}^-=\bigcup_{n=0}^\infty \big(H_{n}\circ \cdots \circ H_{1} \big )(I^-)=I^-.
$$
So the indeterminacy sets of H\'{e}non maps are extremely well-behaved under random iterations which allows us to develop several analogous results for random system of H\'{e}non maps as in \cite{FWe}. 

\medskip

Let $M$ be a compact set in $\mathbb{C}^k$ and $\nu$ be the normalized Lebesgue measure on $M$, then 
\[
\bar{\nu}:=\nu \times \nu \times \cdots \times \nu\times \cdots
\]
is a probability measure on $X:=\Pi_{i=1}^\infty M$ where $X$ is endowed with the product topology. Let ${(H_\lambda)}_{\lambda\in M}$ be a family of generalized H\'{e}non maps as in (\ref{1.5}) of fixed degree $d$. 
For each $\Lambda=(\lambda_1, \lambda_2, ...)\in X$ and $\lambda\in M$, let $\lambda \Lambda=(\lambda,\lambda_1,\lambda_2,\ldots)$ and $\Lambda_n=(\lambda_n, \lambda_{n+1},\ldots)$ for all $n\geq 1$. Further, we define 
\[
H_{n,\Lambda}^{\pm}=H_{\lambda_n}^{\pm 1}\circ \cdots \circ H_{\lambda_1}^{\pm 1}
\]
for each $\Lambda\in X$ and for each $n\geq 1$.

\medskip

To gather dynamical information about the above-mentioned system of random H\'{e}non maps, we first associate Green functions $G_\Lambda^{\pm}$ to any random sequence of H\'{e}non maps ${\{H_{\lambda_n}\}}_{n\geq 1}$ where $\Lambda=(\lambda_1,\lambda_2,...)\in X$. For each $\Lambda\in X$ and for all $n\geq 1$, let
\[
G_{n,\Lambda}^{\pm}(z)=\frac{1}{d^n}\log^+ \big \lVert H_{n,\Lambda}^{\pm}(z)\big\rVert
\]
where $\log^+t=\max \{\log t,0\}$.

For each $\Lambda\in X$, we define the sets $I_\Lambda^{\pm}$ of escaping points and the sets $K_\Lambda^{\pm}$ of non-escaping points as follows:
$$
I_\Lambda^{\pm}=\Big \{z\in \mathbb{C}^2: \lVert H_{\lambda_n}^{\pm 1}\circ \cdots \circ H_{\lambda_1}^{\pm 1}(z) \rVert\ra \infty \text{ as } n\ra \infty \Big \}
$$ 
and
$$
K_\Lambda^{\pm}=\Big \{z\in \mathbb{C}^2: {\big(H_{\lambda_n}^{\pm 1}\circ \cdots \circ H_{\lambda_1}^{\pm 1}(z)  \big )}_{n\geq 0} \text{ is bounded in } \mathbb{C}^2\Big \}.
$$

\begin{thm} \label{R thm1}
For each $\Lambda\in X$, the sequence ${\big\{G_{n,\Lambda}^{\pm}\big \}}_{n\geq 1}$ converges locally uniformly on compact subset of $\mathbb{C}^2$ to the locally H\"{o}lder continuous plurisubharmonic function $G_\Lambda^{\pm}$ as $n\ra \infty$ which satisfies 
\[
G_\Lambda^{\pm}\circ H_\lambda^{\pm 1}=d G_{\Lambda'}^{\pm}
\]
where $\Lambda'=\lambda \Lambda$. Moreover, $G_\Lambda^{\pm}$ are strictly positive pluriharmonic function on $\mathbb{C}^2\setminus K_\Lambda^{\pm}$ and vanish precisely on $K_\Lambda^{\pm}$. Moreover, the correspondence $\Lambda\ra G_\Lambda^{\pm}$ is continuous. 
\end{thm}

As in \cite{FWe}, we define average Green functions $EG^{\pm}$ on $\mathbb{C}^2$ over the set $X$ as follows:
\begin{equation}
EG^\pm(z)=\int_X G_\Lambda^\pm(z)d\bar{\nu}(\Lambda)\label{R1}
\end{equation}
for each $z\in \mathbb{C}^2$. That for a fixed $z\in \mathbb{C}^2$, $G_\Lambda^{\pm}$ depend continuously on $\Lambda$ implies that the integral in (\ref{R1}) makes sense. Further, note that $0\leq EG^\pm(z)< \infty$ since $$
0\leq G_\Lambda^\pm(z)< \log^+\lvert z \rvert+ L
$$
for some $L>0$, for all $\Lambda\in X$ and for all $z\in \mathbb{C}^2$.

\begin{prop} \label{R pr1}
$EG^{\pm}$ are  plurisubharmonic locally H\"{o}lder continuous functions on $\mathbb{C}^2$. Moreover, $EG^{\pm}$ are strictly positive pluriharmonic functions outside $\overline{\bigcup_{\Lambda} K_\Lambda^{\pm}}$ and vanish precisely on $\bigcap_\Lambda K_\Lambda^\pm$. 
\end{prop}  

In view of previous results, we can define $\mu_\Lambda^\pm=\frac{1}{2\pi}dd^c G_\Lambda^{\pm}$ for all $\Lambda\in X$  and $\mu^{\pm}=\frac{1}{2\pi} dd^c(EG^{\pm})$. These are well-defined positive closed $(1,1)$-currents of mass $1$ on $\mathbb{C}^2$. 

\medskip

Given a family of $(1,1)$-currents ${(S_\Lambda)}_{\Lambda\in X}$ defined on $\mathbb{C}^2$ varying continuously in $\Lambda$, we define a positive $(1,1)$-current $\int_X S_\Lambda$ on $\mathbb{C}^2$ as follows:
\begin{equation}
\Big \langle \int_X S_\Lambda,\varphi \Big \rangle=\int_X \big \langle S_\Lambda,\varphi \big \rangle d\bar{\nu} \label{R6}
\end{equation}
for a test form $\varphi$ in $\mathbb{C}^2$. Since the correspondence $\Lambda\mapsto \mu_\Lambda^{\pm}$ is continuous and $\mu_\Lambda^{\pm}$ have mass $1$ for each $\Lambda\in X$, the currents $\int \mu_\Lambda^\pm$ both are well-defined objects in $\mathbb{C}^2$.

\medskip

The next proposition records the action of pull back and push forward operations by a random H\'{e}non map on Green currents and shows that the currents associated with the average Green functions agree with the corresponding average Green currents.

\begin{prop}\label{R pr2}
With $\mu_\Lambda^\pm$ and $\mu^\pm$ as above, we have
\[
\mu^{\pm}= \int_X \mu_\Lambda^{\pm}.
\]
Furthermore,
\begin{enumerate}
\item[$(i)$] ${(H_\lambda^{\pm 1})}^*(\mu_\Lambda^{\pm})=d^{\pm 1}(\mu_{\lambda\Lambda}^\pm) \text{ for any } \lambda\in M \text{ and } \Lambda\in X$;
\item[$(ii)$] $\int_{M} {EG}^{\pm}\circ H_\lambda^{\pm 1}(z)d\nu(\lambda)= d {(EG)}^{\pm}(z)$;
\item[$(iii)$] $\int_{M} {(H_\lambda^{\pm 1})}^* \mu^{\pm}=d \mu^{\pm}$.
\end{enumerate}
The support of $\mu_\Lambda^\pm$ is $J_\Lambda^\pm$ and the correspondences $\Lambda\mapsto J_\Lambda^\pm$ are lower semi-continuous. In addition, we have ${ \rm supp} (\mu^{\pm})=\overline{\bigcup_{\Lambda\in X} J_\Lambda^\pm}$.
\end{prop}

For each $\Lambda\in X$, the positive closed $(1,1)$-current $\mu_\Lambda^+$ can be extended by $0$ as a positive closed $(1,1)$-current of mass $1$ on $\mathbb{P}^2$ and this, by a little abuse of notation, will still be denoted by $\mu_\Lambda^+$. Likewise, the average Green current $\mu^+$ can also be viewed as a positive closed $(1,1)$-current of mass $1$ on $\mathbb{P}^2$ and we continue to call it as $\mu^+$. 

\medskip

Let ${(S_n)}_{n\geq 1}$ be a sequence of positive closed $(1,1)$-currents of mass $1$ on $\mathbb{P}^2$  such that ${(S_n)}_{n\geq 1}$ do have uniformly bounded quasi-potential near $I^-$. Then Theorem $6.6$ in \cite{DS} shows that for a single H\'{e}non map $H$ of degree $d$ the sequence $d^{-n}{(H^n)}^*(S_n)$ converges exponentially fast to the corresponding stable Green current. In lieu of taking a single H\'{e}non map, if we start with a random sequence of H\'{e}non maps and pull the currents back appropriately by the composition of these random H\'{e}non maps, we show that the resulting sequence of currents converges to the corresponding random stable Green current exponentially. 

\begin{thm}\label{R thm2}
Let $\{S_n\}_{n\geq 1}$ be a sequence of positive closed $(1,1)$-currents of mass $1$ on $\mathbb{P}^2$. Moreover, we assume that there exists $C>0$ and some neighborhood $U$ of $I^{-}$ in $\mathbb{P}^2$ such that each ${S_n}$ admits a quasi-potential $u_n$ satisfying $\lvert u_n \rvert \leq C$ on $U$. Then there exists $A>0$ such that
$$
\big\lvert \big\langle  d^{-n} {(H_{\lambda_n}\circ \cdots \circ H_{\lambda_1} )}^* (S_n)-\mu_\Lambda^+, \varphi \big\rangle \big\rvert \leq And^{-n} {\lVert \varphi \rVert}_{C^1}
$$
for all $\Lambda\in X$ and for every test form $\varphi$ of class $C^2$ on $\mathbb{P}^2$. 
\end{thm}

For a positive closed $(1,1)$-current $S$ on $\mathbb{P}^2$, we consider an average pull-back of $S$ as follows:
\begin{equation}
 \Theta(S)=\frac{1}{d} \int_{M} H_\lambda^* S.
\end{equation}
Clearly $\Theta(S)$ is also a  positive closed $(1,1)$-current on $\mathbb{P}^2$.

\begin{thm} \label{R thm3}
Let ${(S_n)}_{n\geq 1}$ be a sequence of  positive closed $(1,1)$-currents on $\mathbb{P}^2$ of mass $1$ such that each ${S_n}$ admits a quasi-potential $u_n$ satisfying $\lvert u_n \rvert \leq C$ on a fixed neighborhood $U$ of $I^-$ in $\mathbb{P}^2$ for some $C>0$. Then $\Theta^n(S_n)\ra \mu^+$ as $n\ra \infty$ in the sense of currents. 
\end{thm}

Let $A$ be a closed subset of a complex manifold $X$. We say that $A$ is {\it{rigid}} in $X$ if $A$ supports at most one non-zero positive closed $(1,1)$-current upto a multiplicative constant.

\medskip

For each $\Lambda\in X$, let $\overline{K_\Lambda^+}$ is the closure of $K_\Lambda$ in $\mathbb{P}^2$, i.e., $\overline{K_\Lambda^+}=K_\Lambda \cup I^+$.  We prove the following.

\begin{thm}\label{R thm4}
For each $\Lambda\in X$, the Green current $\mu_\Lambda^+$ is the unique closed positive $(1,1)$-current of mass $1$ supported on the set $\overline{K_\Lambda^+}$. Hence $\overline{K_\Lambda^+}$ is rigid in $\mathbb{P}^2$.
\end{thm}


So far, we have considered two different families of H\'{e}non maps, namely the pseudo-random case when $\sigma:M\ra M$ is assumed to be merely continuous and secondly, the random case in which $\sigma$ corresponds to a right shift on a set of sequences. In case $\sigma:M\ra M$ is a homeomorphism where $M$ is compact subset of $\mathbb{C}^k$ for $k\geq 1$, it is possible to make some further remarks on the global dynamics of the skew map $H: M\times \mathbb{C}^2 \ra M\times \mathbb{C}^2$, defined by
\begin{equation}
H(\lambda,x,y)=(\sigma(\lambda), H_\lambda(x,y)) \label{homeo}
\end{equation}
where for each $\lambda\in M$, $H_\lambda$ as in (\ref{1.5}). 
Define 
\begin{equation*}
\mathcal{G}^{\pm}_n(\lambda,x,y)= \frac{1}{d^n} \log^+ \lVert H^{\pm n}(\lambda,x,y)\rVert
\end{equation*}
for $n\geq 0$ and for $(\lambda,x,y)\in M\times \mathbb{C}^2$. 

\medskip
 
One can easily check that 
\begin{equation}
H^n(\lambda,x,y)=\big (\sigma^n(\lambda), H_{\sigma^{n-1}(\lambda)}\circ H_{\sigma^{n-2}(\lambda)}\circ\cdots\circ H_\lambda(x,y)\big )
\end{equation}
and
\begin{equation}
H^{-n}(\lambda,x,y)=\big (\sigma^{-n}(\lambda), H_{\sigma^{-n}(\lambda)}^{-1}\circ H_{\sigma^{-{(n-1)}}(\lambda)}^{-1}\circ\cdots\circ H_{\sigma^{-1}(\lambda)}^{-1}(x,y)\big ). 
\end{equation}
For each $n\geq 0$ and $\lambda\in M$, define
$$
\mathcal{H}_\lambda^n= H_{\sigma^{n-1}(\lambda)}\circ H_{\sigma^{n-2}(\lambda)}\circ\cdots\circ H_\lambda
$$
and
$$
\mathcal{H}_\lambda^{-n}= H_{\sigma^{-n}(\lambda)}^{-1}\circ H_{\sigma^{-{(n-1)}}(\lambda)}^{-1}\circ\cdots\circ H_{\sigma^{-1}(\lambda)}^{-1}.
$$
Note that $\mathcal{H}_\lambda^{\pm n}$ is the second coordinate of the $n$-fold iterate of $H^{\pm 1}$. For simplicity $\mathcal{H}_\lambda^n$ will be denoted by $H_\lambda^n$ since they both are the same map. 

\medskip

Now we give definitions of another set of fibered Green functions which behave well with respect to the map $\sigma$ and also harmonize satisfactorily with the global Green functions on $M\times \mathbb{C}^2$. For each $n\geq 0$ and $\lambda\in M$, define
\begin{equation}
\mathcal{G}^{\pm}_{n,\lambda}(x,y)=\frac{1}{d^n} \log^+ \lVert \mathcal{H}_\lambda^{\pm n}(x,y)\rVert \label{G0}
\end{equation}  
for $(x,y)\in \mathbb{C}^2$ where $\log^+(t)=\max \{\log t,0\}$. Note that in general, for each $n\geq 1$ and for each $\lambda\in M$,  $G_{n,\lambda}^+=\mathcal{G}_{n,\lambda}^+$ and $G_{n,\lambda}^- \neq \mathcal{G}_{n,\lambda}^-$.

\medskip

Now for each $\lambda\in M$, we define the sets $\mathcal{I}_\lambda^{\pm}$ of escaping points and the sets $\mathcal{K}_\lambda^{\pm}$ of non-escaping points under suitable random iteration determined by $\sigma$. These definitions appear as follows:
\[
\mathcal{I}_\lambda^{\pm}=\Big \{z\in \mathbb{C}^2: \lVert (\mathcal{H}_\lambda^{\pm n}(z))\rVert\ra \infty \text{ as } n\ra \infty \Big \}
\] 
and
\[
\mathcal{K}_\lambda^{\pm}=\Big \{z\in \mathbb{C}^2: {\big (\mathcal{H}_\lambda^{\pm n}(z)\big )}_{n\geq 0} \text{ is bounded in } \mathbb{C}^2\Big \}.
\]
For each $\lambda\in M$, we have $H_\lambda(\mathcal{K}_\lambda^{\pm})=\mathcal{K}_{\sigma(\lambda)}^{\pm}$ and $H_\lambda(\mathcal{I}_\lambda^{\pm})=\mathcal{I}_{\sigma(\lambda)}^{\pm}$. Define $\mathcal{K}_\lambda=\mathcal{K}_\lambda^+ \cap \mathcal{K}_\lambda^{-}$, $\mathcal{J}_\lambda^{\pm}=\partial \mathcal{K}_\lambda^{\pm}$ and $\mathcal{J}_\lambda=\mathcal{J}_\lambda^+\cap \mathcal{J}_\lambda^{-}$. Thus $H_\lambda(\mathcal{K}_\lambda)= \mathcal{K}_{\sigma(\lambda)}$ and $H_\lambda(\mathcal{J}_\lambda)=\mathcal{J}_{\sigma(\lambda)}$.  

\medskip

Define the sets of escaping points $\mathcal{I}^{\pm}$ and the sets of non-escaping points $\mathcal{K}^{\pm}$ under iteration of $H$ as follows:
\[
\mathcal{I}^{\pm}=\Big \{(\lambda,z)\in M\times \mathbb{C}^2: \lVert (H^{\pm n}(\lambda,x,y))\rVert\ra \infty \text{ as } n\ra \infty \Big \}
\]
and
\[
\mathcal{K}^{\pm}=\Big \{(\lambda,z)\in M\times \mathbb{C}^2: {\big (H^{\pm n}(\lambda,x,y)\big )}_{n\geq 0} \text{ is bounded in } M\times \mathbb{C}^2\Big \}.
\]

\medskip

It is easy to check that
\begin{equation}\label{Esc}
\mathcal{I}^{\pm}= \bigcup_{\lambda\in M} \{\lambda\}\times \mathcal{I}_\lambda^{\pm}
\text{ and } 
\mathcal{K}^{\pm}= \bigcup_{\lambda\in M} \{\lambda\}\times \mathcal{K}_\lambda^{\pm}
\end{equation}
which are by definition both forward and backward invariant under the map $H$.

\medskip

As Proposition \ref{pr1}, it can be shown that the sequence of functions $\mathcal{G}_{n,\lambda}^{\pm}$ converges uniformly on compact subset of 
$\mathbb{C}^2$ to a continuous function $\mathcal{G}_\lambda^{\pm}$ as $n\ra \infty$ which satisfies
\begin{equation}
\mathcal{G}_{\sigma(\lambda)}^{\pm}\circ H_\lambda=d^{\pm 1}\mathcal{G}_\lambda^{\pm} \label{G1}
\end{equation}
on $\mathbb{C}^2$. The functions $\mathcal{G}_\lambda^{\pm}$ are positive pluriharmonic on $\mathbb{C}^2\setminus \mathcal{K}_\lambda^{\pm}$, plurisubharmonic on $\mathbb{C}^2$ and vanishes precisely on $\mathcal{K}_\lambda^{\pm}$. The correspondences $\lambda \mapsto \mathcal{G}_\lambda^{\pm}$ are continuous and $\mathcal{G}_\lambda^{\pm}$ are locally uniformly H\"{o}lder continuous on $\mathbb{C}^2$.

\medskip

Thus for each $\lambda\in M$, $\vartheta_\lambda^{\pm}= \frac{1}{2\pi}dd^c \mathcal{G}_\lambda^{\pm}$ are positive closed $(1,1)$-currents on $\mathbb{C}^2$ of mass $1$ and hence $\vartheta_\lambda=\vartheta_\lambda^+ \wedge \vartheta_\lambda^-$ gives a family of probability measures on $\mathbb{C}^2$. 

\medskip

Using similar techniques as Proposition \ref{pr2}, one can prove that for $\vartheta_{\la}^{\pm} \text{ and } \vartheta_{\la}$, we have
\[
(H_{\la})^{\ast} \vartheta_{\si(\la)}^{\pm} = d^{\pm 1} \vartheta_{\la}^{\pm}, \; (H_{\la})_{\ast} \vartheta_{\la}^{\pm} = d^{\mp 1} \vartheta_{\si(\la)}^{\pm}, \; (H_{\la})^{\ast} \vartheta_{\sigma(\la)}=\vartheta_\lambda.
\]
The support of $\vartheta^{\pm}_{\la}$ equals $\mathcal{J}_{\la}^{\pm}$ and the correspondence $\la \mapsto \mathcal{J}_{\la}^{\pm}$ is lower-semicontinuous. Furthermore, for each $\lambda\in 
M$, the pluricomplex Green function of $\mathcal{K}_\lambda$ is $\max\{\mathcal{G}_\lambda^+, \mathcal{G}_\lambda^-\}$, $\vartheta_\lambda$ is the complex
equilibrium measure of $\mathcal{K}_\lambda$ and ${\rm supp}(\vartheta_\lambda)\subseteq \mathcal{J}_\lambda$.

\begin{prop}\label{G pr1}
The sequences $\mathcal{G}_{n}^{\pm}$ converge uniformly on compact subsets of $M \times \mbb C^2$ to a continuous plurisubharmonic function $\mathcal{G}^{\pm}$ as $n \ra \infty$ that satisfies
\begin{equation}
d^{\pm} \mathcal{G}^{\pm} = \mathcal{G}^{\pm} \circ H^{\pm 1}.\label{G2}
\end{equation}
The functions $\mathcal{G}^{\pm}$ are positive on $(M\times \mbb C^2) \setminus \mathcal{K}^{\pm}$ and vanish precisely on $\mathcal{K}^{\pm}$.
\end{prop}

Therefore, $\mathcal{T}^{\pm}=dd^c \mathcal{G}^{\pm}$ are positive closed $(1,1)$-currents on $M\times \mathbb{C}^2$. Further, by (\ref{G2}), we have the following:
\[
H^*(\mathcal{T}^{\pm})=d^{\pm 1}\mathcal{T}^{\pm}.
\]

\medskip

For every probability measure $\vartheta'$ on $M$,
\begin{equation}
\langle \vartheta, \varphi \rangle = \int_{M} \bigg( \int_{\{\la\} \times \mbb C^2} \varphi \; \vartheta_{\la} \bigg) \vartheta'(\la)\label{mu}
\end{equation}
defines a measure on $M \times \mbb C^2$ by describing its action on continuous functions $\varphi$ on $M \times \mbb C^2$. This is not a dynamically natural measure as $\vartheta'$ is arbitrary. Note that,
since for each $\lambda\in M$, we have $H_\lambda^*(\vartheta_{\sigma(\lambda)})=\vartheta_\lambda$, it follows that $H^*(\vartheta)=\vartheta$ on $M\times \mathbb{C}^2$ when $\vartheta'$ is an invariant measure for $\sigma$ on $M$.

\medskip

Let $T$ be a measure preserving transformation on a probability space $(X,\vartheta)$. Then $T$ is called {\it{mixing}} for the measure $\vartheta$ if 
\[
\lim_{n\ra \infty} \int (T^{n*}\varphi)\psi d\vartheta = \int \varphi d\vartheta \int \psi d\vartheta
\]
for all continuous functions $\varphi$ and $\psi$ with compact support in $X$. We have the following theorem.

\begin{thm}\label{G thm2}
If $\sigma$ is mixing for $\vartheta'$, then $H$ is mixing for $\vartheta$.
\end{thm}

One can also generalize the notion of mixing in random set up. In \cite{P}, the concept of random mixing is discussed for a sequence of measures $\{\vartheta_n\}_{n\geq 0}$ which are preserved by a sequence of transformations $\{T_n\}_{n\geq 0}$ on a space $X$, i.e., $T_{n}^*(\vartheta_{n+1})= \vartheta_n$ on $X$ for each $n\geq 0$. 

\medskip

A measure preserving sequence ${\{T_n\}}_{n\geq 0}$ is randomly mixing for the sequence of measures ${\{\vartheta_n\}}_{n\geq 0}$ if for all continuous functions $\varphi$ and $\psi$ on $X$, the following holds:
\[
\int (T_n^* \varphi)\circ \psi d \vartheta_0- \int \varphi d \vartheta_n \int \psi \vartheta_0 \ra 0
\]   
as $n\ra \infty$.

\medskip

Note that for each $\lambda\in M$, the sequence of H\'{e}non maps $\{H_{\sigma^n(\lambda)}\}$ is measure preserving for the sequence of probability measures $\{\vartheta_{\sigma^n(\lambda)}\}$ on $\mathbb{C}^2$. The following proposition shows that in the above-mentioned system, mixing happens at the fiber level also.

\begin{prop}\label{G pr2}
For each $\lambda\in M$, the sequence of H\'{e}non maps ${\{H_{\sigma^n(\lambda)}\}}_{n\geq 0}$ is randomly mixing for the sequence of probability measures ${\{\vartheta_{\sigma^n(\lambda)}\}}_{n\geq 0}$.
\end{prop}

It turns out that the support of $\vartheta$ is contained in
\[
\mathcal {J} =\overline{ \bigcup_{\la \in M} \left( \{ \la \} \times \mathcal{J}_{\la} \right)} \subset M \times V_R.
\]

So one can consider $H$ as a self map of ${\rm supp}(\vartheta$) with invariant measure $\vartheta$. The following theorem holds.

\begin{thm}\label{thm2} 
The measure theoretic entropy of $H$ with respect to $\vartheta$ is at least $\log d$. In particular, the topological
entropy of $H : \cal J \ra \cal J$ is at least $\log d$.
\end{thm}

It would be both interesting and desirable to obtain the above results for an arbitrary continuous function $\sigma$ in (\ref{1}).

\medskip

We will now consider continuous families of holomorphic endomorphisms of $\mbb P^k$. For a compact metric space $M$, $\si$ a continuous self map of $M$, define $F : M \times \mbb P^k
\ra M \times \mbb P^k$ as
\begin{equation}
F(\la, z) = (\si(\la), f_{\la}(z)) \label{2}
\end{equation}
where $f_{\la}$ is a holomorphic endomorphism of $\mbb P^k$ that depends continuously on $\la$. Each $f_{\la}$ is assumed to have a fixed degree $d \ge 2$. Corresponding to each
$f_{\la}$  there exists a non-degenerate homogeneous holomorphic mapping $F_{\la} : \mbb C^{k+1} \ra \mbb C^{k+1}$ such that $\pi \circ F_{\la} = f_{\la} \circ \pi$ where
$\pi : \mbb C^{k+1} \setminus \{0\} \ra \mbb P^k$ is the canonical projection. Here, non-degeneracy means that $F_{\la}^{-1}(0) = 0$ which in turn implies that
there are uniform constants $l, L >0$ with
\begin{eqnarray}
l \Vert x \Vert^d \le \Vert F_{\la}(x) \Vert \le L \Vert x \Vert^d \label{ineq}
\end{eqnarray}
for all $\la \in M$ and $x \in \mbb C^{k+1}$. Therefore, for $0 < r \leq (2L)^{-1/(d-1)}$
\[
\Vert F_\lambda(x) \Vert \leq (1/2) \Vert x \Vert
\]
for all $\lambda\in M$ and $\Vert x \Vert \leq r$. Likewise, for $R\geq (2l)^{-1/(d-1)}$
\[
\Vert F_\lambda(x) \Vert \geq 2 \Vert x \Vert
\]
for all $\lambda\in M$ and  $\Vert x \Vert \geq R$.

\medskip

While the ergodic properties of such a family have been considered in \cite{T1}, \cite{T2} for instance, we are interested in looking at
the basins of attraction which may be defined for each $\la \in M$ as
\[
\mathcal A_{\la} =  \big\{ x \in \mbb C^{k+1} : F_{\si^{n-1}(\la)} \circ \cdots \circ F_{\si(\la)} \circ F_{\la}(x) \ra 0 \; \text{as} \; n \ra \infty \big\}
\]
and for each $\lambda\in M$, the region of normality $\Om'_{\la} \subset \mbb P^k$ which consists of all points $z \in \mbb P^k$ for which there is a neighborhood $V_z$ on which the sequence
$\big \{f_{\si^{n-1}(\la)} \circ \cdots \circ f_{\si(\la)} \circ f_{\la} \big\}_{n \ge 1}$ is normal. Analogs of $\mathcal A_{\la}$ arising from composing a given sequence of
automorphisms of $\mbb C^n$ have been considered in \cite{PW} where an example can be found for which these are not open in $\mbb C^n$. However, since each $F_{\la}$ is homogeneous, it
is straightforward to verify that each $\mathcal A_{\la}$ is a nonempty, pseudoconvex complete circular domain. As in the
case of a single holomorphic endomorphism of $\mbb P^k$ (see \cite{HP}, \cite{U}), the link between these two domains is provided by the Green function.

\medskip

For each $\la \in M$ and $n \ge 1$, let
\[
G_{n, \la}(x) = \frac{1}{d^n} \log \Vert F_{\si^{n-1}(\la)} \circ \cdots \circ F_{\si(\la)} \circ F_{\la}(x) \Vert.
\]

\begin{prop}\label{pr4}
For each $\la \in M$, the sequence $G_{n, \la}$ converges uniformly on $\mbb C^{k+1}$ to a continuous plurisubharmonic function $G_{\la}$ which satisfies
\[
G_{\la}(c x) = \log \vert c \vert + G_{\la}(x)
\]
for $c \in \mbb C^{\ast}$. Further, $d G_{\la} = G_{\si(\la)} \circ F_{\la}$, and $G_{\la_n} \ra G_{\la}$ locally uniformly on $\mbb C^{k+1} \setminus \{0\}$ as $\la_n \ra \la$ in
$M$. Finally,
\[
\mathcal A_{\la} = \{x \in \mbb C^{k+1} : G_{\la}(x) < 0\}
\]
for each $\la \in M$.
\end{prop}

For each $\la \in M$, let $\mathcal Q_{\la} \subset \mbb C^{k+1}$ be the collection of those points in a neighborhood of which $G_{\la}$ is pluriharmonic and define $\Om_{\la} =
\pi(\mathcal Q_{\la}) \subset \mbb P^k$.

\begin{prop}\label{pr5}
For each $\la \in M$, $\Om_{\la} = \Om'_{\la}$. Further, each $\Om_{\la}$ is pseudoconvex and Kobayashi hyperbolic.
\end{prop}

\medskip

This thesis is adorned as follows. Chapter $2$ contains some preliminaries. In Chapter $3$, we study dynamics of skew products of H\'{e}non maps. Chapter $4$ generalizes some of the results obtained in Chapter $3$ and some other relevant results for a random system of H\'{e}non maps.
In Chapter $5$, we mainly study some important ergodic theoretic aspect of the global map $H$ defined in (\ref{1}). Chapter $6$ intends to make some remark about the dynamics of skew products of holomorphic endomorphisms of $\mathbb{P}^k$ that are fibered over a compact base.

\thispagestyle{empty}

\thispagestyle{empty}

\chapter{Some Preliminaries}
The purpose of this chapter is to gather some standard facts about plurisubharmonic functions and positive currents which will be used later.
\section{Positive currents and pluripotential theory on complex manifolds}
Here we gather some standard facts about the positive currents which we will be using later. All of the results discussed are classical and their proofs are omitted here. These results can be found in \cite{DEM} and \cite{HDS}. 

\medskip

Let $X$ be a complex manifold of dimension $k$ and $\omega$ a Hermitian $(1,1)$-form on $X$ which is positive definite at every point. The space of currents of degree $l$ ( or dimension $(2k-l)$) on $X$ is the space $\mathcal{D}_{2k-l}'(X)$ of continuous linear form $S$ on the space $\mathcal{D}_{2k-l}(X)$ of smooth $(2k-l)$-forms with compact support in $X$. It turns out that in any local coordinate chart a current can be thought of as a differential form with distribution coefficients. In particular, currents of order zero on $X$ can be considered as differential forms with measure coefficients. In the sequel, we let $\langle S, \varphi \rangle$ be the pairing between a current $S$ and a test form $\varphi$. The support of $S$, denoted by ${\rm{supp}}(S)$, is the smallest closed subset of $X$ such that $\langle S,\varphi \rangle=0$ when $\varphi$ is supported on $X\setminus {\rm{supp}}(S)$.

\begin{exam}
Let $f$ be a differential form of degree $q$ on $X$ with $L_{\rm{loc}}^1$ coefficients. We can associate to $f$ a current of dimension $(2k-q)$ as follows:
\[
\langle T_f, \varphi\rangle:=\int_X f\wedge \varphi
\]
for $\varphi\in \mathcal{D}_{2k-q}(X)$. 
\end{exam}

Many of the operations available for differential forms can be extended to currents by simple duality arguments. In particular, one can define the operators $d,\partial, \bar{\partial}$ on currents using duality. For example, one can define a $(l+1)$-current $dS$ on $X$ as follows: 
\[
\langle dS,\varphi \rangle:={(-1)}^{l+1}\langle S, d\varphi\rangle\] 
for  $\varphi\in \mathcal{D}_{2k-l-1}(X)$. It can be checked that when $S$ is a smooth form, the above identity is a consequence of Stokes' formula. That the linear form $dS$ is continuous follows from the continuity of the map $d$. 

\medskip

Let $Y$ be a complex manifold of dimension $k'$ with $2k'\geq 2k-l$ and $F:X \ra X'$ a holomorphic map such that the restriction of $F$ to ${\rm{supp}}(S)$ is proper, i.e., if ${\rm{supp}}(S)\cap F^{-1}(K)$ is compact for every compact subset $K\subset X'$. Then the linear form $\varphi\mapsto \langle S, F^* \varphi\rangle$ is well-defined and continuous on $\mathcal{D}_{2k-l}(Y)$. Hence there exists a unique current $F_*S$, called the push-forward current of $S$, such that 
\[
\langle F_*S, \varphi\rangle:=\langle S,F^*(\varphi)\rangle
\]
for $\varphi\in \mathcal{D}_{2k-l}(Y)$. Note that the currents $S$ and $F_* S$ are of the same dimension. One can easily verify that support of $F_* S$ lies in the image of ${\rm{supp}}(S)$ under $F$. Further, the operators $d$ and $F_*$ commute with each other, i.e., $d(F_* S)=F_*(dS)$. If $Y$ is a complex manifold of dimension $k'\geq k$ and if $F: Y\ra X$ is a submersion, we can define the pull-back $F^*S$ of $S$ by $F$ as follows:
\[
\langle F^*S, \varphi\rangle:=\langle S,F_*(\varphi)\rangle
\]
for $\varphi\in \mathcal{D}_{2k'-l}(Y)$. The current $F^*S$ is an $l$-current supported on $F^{-1}({\rm{supp}}(S))$. As before, in this case also the operators $d$ and $F^*$ commute with each other, i.e., $d(F^* S)=F^*(dS)$.

\medskip

A current $S$ is of bidegree $(p,q)$ and of bidimension $(k-p,k-q)$ if it vanishes on the forms of bidegree $(r,s)\neq (k-p,k-q)$. The push-forward and the pull-back operations can be defined as above. And here also the push-forward preserves the bidimension and the pull-back preserves the bidegree.

\medskip

Let $V$ be a complex vector space of dimension $k$ with coordinates $(z_1,...,z_k)$. We denote by $({\partial}/{\partial z_1},...,{\partial}/{\partial z_n})$ the corresponding basis of the tangent space of $TV$ of $V$ and by $(d z_1,...,d z_n)$ its dual basis in ${TV}^*$. Consider the exterior algebra $\bigwedge^{p,q} {TV}^*=\bigwedge^p {TV}^* \otimes \bigwedge^q\overline{{TV}^*}$ where $\bigwedge^p {TV}^*$ and $\bigwedge^p \overline{{TV}^*}$ are the collection of $p$-forms on ${TV}^*$ and $\overline{{TV}^*}$ respectively. A $(p,p)$-form $\varphi \in \bigwedge^{p,p}{TV}^*$ is positive if for all $\alpha_j\in {TV}^*$ with $1\leq j \leq q=k-p$, 
\[
\varphi\wedge i\alpha_1\wedge\bar{\alpha}_1\wedge \cdots \wedge i\alpha_q\wedge\bar{\alpha}_q
\] 
is a positive $(k,k)$-form. A $(q,q)$-form $\psi\in \bigwedge^{q,q}{TV}^*$ is said to be strongly positive if $\psi$ is a convex combination as follows:
\[
\psi=\sum \gamma_t  \big( i\alpha_{t,1}\wedge\bar{\alpha}_{t,1}\wedge \cdots \wedge i\alpha_{t,q}\wedge\bar{\alpha}_{t,q}\big)
\] 
where $\alpha_{t,j}\in {TV}^*$ and $\gamma_t\geq 0$. Equivalently, a form $\varphi\in \bigwedge^{p,p}{TV}^* $ is positive if and only if its restriction to any $p$-dimensional subspace of  $S\subset TV$ is a positive volume form.  Moreover, one can prove that all positive forms $\varphi$ are real, i.e., $\varphi=\bar{\varphi}$. The notion of positivity and strong positivity differ in all bidegrees $(p,p)$ with $2\leq p \leq (k-2)$ but  they coincide for $p=0,1,k-1,k$. 

\medskip

The sets of positive and strongly positive forms are closed convex cones, i.e., closed and stable under convex combination and by definition the positive cone is dual to strongly positive cone. The duality between these two cones allows us to define corresponding positivity notions for currents. A $(p,p)$-current $S$ is said to be positive (resp. strongly positive) if $\langle S,\varphi \rangle \geq 0$ for all $(k-p,k-p)$-forms  $\varphi$ which are strongly positive (resp. positive) at each point of $X$. It is clear that (strong) positivity is a local property and any strongly positive current is always positive. In any local coordinate chart a positive current corresponds to a differential form whose coefficients are complex measures. Positivity property of currents remain preserved under direct or inverse image by holomorphic maps.  

\begin{exam}
Let $Z$ be an analytic subset of pure codimension $p$ of $X$. It can be shown that the $2(k-p)$-dimensional volume of $Z$ is locally finite in $X$. Hence we can define the current of integration over $Z$, denoted by $[Z]$, as follows:
\[
\langle [Z],\varphi \rangle :=\int_{reg(Z)} \varphi
\] 
for $\varphi \in \mathcal{D}_{k-p,k-p}$, the space of smooth $(k-p,k-p)$-forms with compact support in $X$ and $reg(Z)$ is the regular part of $Z$ in $X$. Lelong proved that this current is positive and closed. 
\end{exam}

\medskip

The theory of closed positive $(1,1)$-currents is well developed thanks to the use of plurisubharmonic functions. They are the counterparts of subharmonic functions in higher dimension. An upper semi-continuous function $u: X\ra \mathbb{R}\cup \{-\infty\}$, not identically $-\infty$ on any component of $X$, is said to be plurisubharmonic (abbreviated as p.s.h.) if it is subharmonic or identically $-\infty$ on any holomorophic disc in $X$. As in the case of subharmonic function, here also we have the submean inequality: in local coordinates, the value at $a$ of a p.s.h. function is smaller or equal to the average of the function on a sphere centered at $a$. Indeed, this average increases with the radius of the sphere. Further, the submean inequality implies the maximum principle which says that if a p.s.h. function on a connected manifold $X$ has a maximum, it is constant. The upper semi-continuity of p.s.h. functions implies that these functions are bounded above. For a complex manifold $Y$, if $F:X\ra Y$ is a holomorphic map, then it is easy to check that $u\circ F$ is p.s.h. in $X$. A function $u$ is pluriharmonic if $u$ and $-u$ both are plurisubharmonic. Pluriharmonic functions are locally the real part of a holomorphic function. In particular, they are real analytic. 

\begin{exam}
The function $\log \lvert z \rvert$ is subharmonic on $\mathbb{C}$, thus $\log\lvert F \rvert$ is plurisubharmonic in $X$ for every complex valued holomorophic function $F$ on $X$. More generally 
\[
\log ({\lvert f_1 \rvert}^{\alpha_1}+ \cdots +{\lvert f_q \rvert}^{\alpha_q})
\]
is a plurisubharmonic function in $X$ for holomorphic functions $f_j$ on $X$ and $\alpha_j \geq 0$.
\end{exam} 

The following results are already known for subharmonic functions and it is straightforward to extend them to the case of p.s.h. functions.

\begin{prop}
Let $\Omega \subset \mathbb{C}^k$ and ${(u_n)}_{n\geq 1}$ be a sequence of p.s.h. functions on $\Omega$, then the limit $u=\lim u_n$ is p.s.h. on $\Omega$.
\end{prop}

\begin{prop}
Let $u_1,\ldots ,u_p$ are p.s.h. on $\Omega$ and $\chi: \mathbb{R}^k \ra \mathbb{R}$ be a convex function such that $\chi(t_1,\ldots ,t_k)$ is non-decreasing in each $t_j$. Then $\chi(t_1,\ldots,t_k)$ is p.s.h. on $\Omega$. In particular, $u_1+\cdots +u_k$, $\max\{u_1,\ldots ,u_k\}$, $\log (e^{u_1}+ \cdots +e^{u_k})$ are p.s.h. on $\Omega$.
\end{prop}

\begin{prop}
Let $u$ be a p.s.h. function on an open set $\Omega\subset\mathbb{C}^k$ and $\Omega'\subset \subset \Omega$ be an open set. Then there is a sequence of smooth p.s.h. functions $u_n$ on $\Omega'$ which decreases to $u$. 
\end{prop}

The following proposition gives an important characterization of p.s.h. functions.

\begin{prop}
A function $u: X\ra \mathbb{R}\cup \{-\infty\}$ is p.s.h. if and only if the following conditions hold:

\begin{itemize}
\item
u is strongly upper semi-continuous, i.e., for any subset $A$ of full Lebesgue measure in $X$ and for any point $a\in X$, we have $u(a)=\limsup_{z\ra a} u(z)$ for $z\in A$.
\item
u is locally integrable with respect to the Lebesgue measure on $X$ and $dd^c u$ is a positive closed $(1,1)$-current. 
\end{itemize}
\end{prop}

Conversely, if $S$ is a positive closed current of bidegree $(1,1)$, then for every point $p\in X$, there exists a neighborhood $\Omega_p$ of $p$ and a p.s.h. function $u$ on $\Omega_p$ such that $S=dd^c u$. This function is called a local potential of the current. Two local potentials are same modulo a pluriharmonic function. So there is well-defined correspondence between positive closed $(1,1)$-currents and p.s.h. functions. In particular, for a positive closed $(1,1)$-current $S$ in $\mathbb{C}^k$, there is a plurisubharmonic function $u\in \mathbb{C}^k$ such that $S=dd^c u$.

\begin{thm}
Let $u$ be a locally integrable real function satisfying $dd^c u=0$ on an open set $\Omega\subset \mathbb{C}^2$ in the sense of currents, then there is a pluriharmonic function $v$ on $\Omega$ such that $u=v$ in the sense of distribution.
\end{thm}

A subset $E\subset X$ is said to be complete pluripolar in $X$ if for every point $p\in X$ there exists a neighborhood $\Omega_p$ of $p$ and a locally integrable  p.s.h. function $u$ on $\Omega_p$ such that $E\cap \Omega_p=\{z\in \Omega_p: u(z)=-\infty\}$. Note that any closed analytic subset $A\subset X$ is complete pluripolar. Pluripolar sets are of Hausdorff dimension $\leq 2k-2$, in particular, they have zero Lebesgue measure. Finite and countable union of pluripolar sets are pluripolar. Thus finite and countable union of analytic subsets are pluripolar.

\begin{prop}
Let $E$ be a closed pluripolar set in $X$ and $u$ a p.s.h. function on $X\setminus E$, locally bounded above near $E$. Then the extension of $u$ to $X$ given by
\[
u(z):=\limsup_{w\ra z, w\in X\setminus E} u(z) \text{ for } z\in E
\]
is a p.s.h. function.
\end{prop}

Skoda proved that a closed positive current $S$ defined in the complement of an analytic set $A$ can be extended through $A$ as a closed positive current $\tilde{S}$ if the mass of the current is locally finite near $A$. Later El. Mir showed that it is enough to assume that $A$ is complete pluripolar. This extension can be obtained by extending the measure coefficients of $S$ by $0$ on $A$. This is the trivial extension of $S$. We have the following theorem:

\begin{thm}
Let $A\subset X$ be a closed complete pluripolar set and $S$ be a positive closed $(p,p)$-current on $X\setminus A$. Assume that $S$ has finite mass in a neighborhood of every point of $A$. Then its trivial extension through $A$ is a positive closed $(p,p)$-current on $X$.  
\end{thm}

We have seen that positive closed currents generalize differential forms and analytic sets. But it is not always possible to extend the underlying calculus of forms and analytic sets to positive closed currents. The following discussion will show that how positive closed currents are flexible and how they are very rigid in some sense. In general, wedging of two positive closed $(1,1)$-currents does not make sense simply because both of the currents have measure coefficients and measures can not be multiplied. But one can define the wedge product $dd^c u \wedge S $ where $u$ is locally bounded p.s.h. function on $X$ and $S$ is a closed positive $(1,1)$-current on $X$. In this case, because of the locally boundedness of $u$, the current $uS$ is well-defined. According to Bedford-Taylor,
\[
dd^c u \wedge S= dd^c (uS)
\] 
where $dd^c$ is taken in the sense of distributions. This current is again a closed positive current. Given locally bounded p.s.h. functions $u_1,\ldots,u_q$, we define inductively
\[
dd^c u_1 \wedge dd^c u_2 \wedge \cdots \wedge dd^c u_q \wedge S=dd^c(u_1 dd^c u_2 \wedge \cdots \wedge dd^c u_q\wedge S) 
\]
which turns out to be a closed positive current.

\medskip

If $\omega$ is a fixed Hermitian form as above and $S$ is a $(p,p)$-current, then $S\wedge {\omega}^{k-p}$ is a positive measure which is called the trace measure of $S$ for $p\leq k-1$. In local coordinates, the coefficients of $S$ are measures, bounded by a constant times the trace measure. The following version of the Chern-Levine-Nirenberg inequality is useful.
\begin{thm}
Let $S$ be a positive closed $(p,p)$-current on $X$. Let $u_1,\ldots,u_q$, $q\leq k-p$, be locally bounded p.s.h. functions on $X$ and $K$ a compact subset of $X$. Then then there exists a constant $c>0$ depending only on $K$ and $X$ such that if $u$ is p.s.h. on $X$, $u$ satisfies
\[
{\lVert u\wedge dd^c u_1\wedge \cdots\wedge dd^c u_q \wedge S\rVert}_K \leq c {\lVert u \rVert}_{L^1(\sigma_S)}{\lVert u_1 \rVert}_{L^\infty(X)}\cdots {\lVert u_q \rVert}_{L^\infty(X)}
\] 
where $\sigma_S$ denotes the trace measure of $S$.
\end{thm} 
The above theorem shows that p.s.h. functions are locally integrable with respect to the current $dd^c u_1\wedge\cdots \wedge dd^c u_q$. We deduce the following corollary.
\begin{cor}
Let $u_1,\ldots ,u_p$, $p\leq k$, be locally bounded p.s.h. functions on $X$. Then the current $dd^c u_1\wedge\cdots \wedge dd^c u_q $ has no mass on locally pluripolar sets, in particular, on proper analytic sets of $X$.
\end{cor}
Here is a monotone continuity theorem due to Bedford-Taylor.

\begin{thm}
Let $u_1,\ldots ,u_q$ be locally bounded p.s.h. functions with $q\leq k-p$ and let $u_1^n,\ldots,u_q^n$ be decreasing sequence of p.s.h. functions converging pointwise to $u_1,\ldots,u_q$. Then 
\[
u_1^n dd^c u_2^n \wedge\cdots\wedge dd^c u_q^n\wedge S\ra u_1 dd^c u_2 \wedge\cdots\wedge dd^c u_q\wedge S
\]
and
\[
dd^c u_1^n \wedge\cdots \wedge dd^c u_q^n\wedge S \ra dd^c u_1^n \wedge \cdots \wedge dd^c u_q^n \wedge S
\]
weakly for any positive closed $(p,p)$-current $S$.
\end{thm}

\section{Currents on projective spaces}

Complex projective $k$-space, denoted by $\mathbb{P}^k$, is defined to be the set of $1$-dimensional complex linear subspaces of $\mathbb{C}^{k+1}$ with the quotient topology inherited from the natural projection
\[
\pi: \mathbb{C}^{k+1}\setminus \{0\}\ra \mathbb{C}\mathbb{P}^k.
\]
Thus $\mathbb{P}^k$ is the quotient of $\mathbb{C}^{k+1}\setminus \{0 \}$ by the equivalence relation $z\sim w $ if and only if there exists a $\lambda\in \mathbb{C}\setminus \{0 \}$ such that $w=\lambda z$. We can cover $\mathbb{P}^k$ by open sets $U_i$ associated with the sets $\{w_i\neq 0\}$ in $\mathbb{C}^{k+1}\setminus \{0 \}$. Each $U_i$ is biholomorphic to $\mathbb{C}^k$ and $({w_0}/{w_i},\ldots,{w_{i-1}}/{w_i},{w_{i+1}}/{w_i},\ldots ,{w_k}/{w_i})$ is a coordinate system on it. The complement of $U_i$ is the hyperplane defined by $\{w_i=0\}$. Thus $\mathbb{P}^k$ can be recovered from  $\mathbb{C}^k$ by adding a hyperplane at infinity. Moreover, $\mathbb{P}^k$ can be viewed as a natural compactification of $\mathbb{C}^k$. We denote by $[w_0:\cdots:w_k]$ the point of $\mathbb{P}^k$ associated with the point $(w_0,\ldots,w_k)$ in $\mathbb{C}^{k+1}$. Further, the vector space structure of $\mathbb{C}^{k+1}$ induces an analogous structure on $\mathbb{P}^k$ by homogenization. In particular, each linear inclusion $\mathbb{C}^{r+1}\subset \mathbb{C}^{k+1}$ induces an inclusion $\mathbb{P}^r \subset \mathbb{P}^k$ and the image of a hyperplane in $\mathbb{C}^{k+1}$ is again a hyperplane $\mathbb{P}^k$.

\medskip

The projective space $\mathbb{P}^k$ is endowed with a K\"{a}hler form $\omega_{FS}$, called the Fubini-Study form. This is defined on the chart $U_i$ by
\[
\omega_{FS}:=\frac{1}{2\pi} dd^c \log \Big( \sum_{j=0}^k {\Big \lvert \frac{w_j}{w_i} \Big\rvert}^2\Big).
\] 
For a closed positive $(p,p)$-current $S$, the mass of $S$ is given by 
$$
\lVert S \rVert:=\langle S, \omega_{FS}^{k-p}\rangle.
$$ 
Since $H^{p,p}(\mathbb{P}^k, \mathbb{R})$ is generated by $\omega_{FS}^p$, a $(p,p)$-current $S$ is cohomologous to $\lVert S \rVert \omega_{FS}^p$. So $S-\lVert S \rVert \omega_{FS}^p$ is exact and by the $dd^c$-lemma there exists a $(p-1,p-1)$-current $U$ such that $S=\lVert S \rVert \omega_{FS}^p+dd^c U$. We call $U$ a quasi-potential of $S$.

\medskip

S. T. Yau introduced in \cite{Ya} the notion of quasi-p.s.h. functions. A quasi-p.s.h. function is locally the difference of a p.s.h. function and a smooth one. Various useful properties of quasi-p.s.h. functions can be deduced from those of p.s.h. functions. 

\medskip

In a local coordinates of $\mathbb{P}^k$, we set 
\[
\omega=\sqrt{-1}\sum_{i=1}^k dz_i \wedge d \bar{z_i}
\]
which is  the standard Hermitian form in $\mathbb{P}^k$. We have the following results.

\begin{prop}\label{pre 1}
If $U$ is quasi-p.s.h. on $\mathbb{P}^k$, then it belongs to $L^p(\mathbb{P}^k)$ for every $1\leq p <\infty$ and $dd^c U \geq -c \omega$ for some constant $c\geq 0$. If ${(U_n)}_{n\geq 1}$ is a decreasing sequence of quasi-p.s.h. functions on $\mathbb{P}^k$ satisfying $dd^c U_n \geq -c \omega$ with $c$ independent of $n$, then the limit is also quasi-p.s.h. function. If $S$ is  a positive closed $(1,1)$-current and $\alpha$ a smooth real form in the cohomology class of $S$, then there exists a quasi-p.s.h. function $U$, unique up to an additive constant, such that $dd^c U=S-\alpha$. 
\end{prop}

\begin{prop}\label{pre 2}
Let $U$ be a quai-p.s.h. function on $\mathbb{P}^k$ such tha $dd^c U \geq -\omega$. Then there is a sequence ${(U_n)}_{n\geq 1}$ of smooth quasi-p.s.h functions decreasing to $U$ such that $dd^c U_n \geq -\omega$. In particular, if $S$ is a positive closed $(1,1)$-current on $\mathbb{P}^k$, then there are smooth positive closed $(1,1)$-forms $S_n$ converging to $S$. 
\end{prop}

A function on $\mathbb{P}^k$ is called d.s.h. (difference of quasi-p.s.h. functions) if it is the difference of two quasi-p.s.h. functions outside a pluripolar set. The notion of d.s.h. functions was introduced by Dinh and Sibony (see \cite{DS1}). Two d.s.h. functions are said to be equal if they agree outside a pluripolar set. We denote the space of d.s.h. functions by DSH$(\mathbb{P}^k)$. The following proposition can be easily deduced from the properties of p.s.h. functions.

\begin{prop}\label{pre 3}
The space DSH$(\mathbb{P}^k)$ is contained in $L^p(\mathbb{P}^k)$ for $1\leq p <\infty$. If $U$ is a d.s.h. function, then we can write $dd^c U=S^+ - S^-$ where $S^{\pm}$ are positive closed $(1,1)$-currents in the same cohomology class. Conversely if $S^{\pm}$ are positive closed $(1,1)$-currents of the same mass, then there is a d.s.h. function $U$, unique upto a constant, such that $dd^c U=S^+-S^-$. 
\end{prop}

For a $(1,1)$-current $S$ which is the difference of two positive closed $(1,1)$-currents, define
\[
{\lVert S \rVert}_*:=\inf (\lVert S^+ \rVert+\lVert S^- \rVert)
\]
where the infimum is taken over all positive closed $(1,1)$-currents $S^\pm$ such that $S=S^+- S^-$. Now we define a norm on DSH$(\mathbb{P}^k)$. For $U\in \text{DSH}(\mathbb{P}^k)$, define
\[
{\lVert U \rVert}_{\text{DSH}}:=\lvert \langle \omega,U\rangle\rvert + {\lVert dd^c U \rVert}_*.
\]
The term $\lvert \langle \omega^k,U\rangle\rvert$ may be replaced with ${\lVert U \rVert}_{L^p}$, $1\leq p <\infty$. This replacement will give an equivalent norm on DSH$(\mathbb{P}^k)$. The space of d.s.h. functions endowed with the above norm is a Banach space. In fact, replacing $\omega$ with any positive measure, for which quasi-p.s.h. functions are integrable, will give an equivalent norm. Here is a useful result.

\begin{lem}\label{pre 4}
Let $\mathcal{F}$ be a family of d.s.h. functions on $\mathbb{P}^k$ such that ${\lVert dd^c U \rVert}_*$ is bounded by a fixed constant for every $U\in \mathcal{F}$. Let $\Omega$ be a non-empty open subset of $\mathbb{P}^k$. Then $\mathcal{F}$ is bounded in DSH$(\mathbb{P}^k)$ if and only if $\mathcal{F}$ is bounded in $L^1(\Omega)$.
\end{lem}

The following result can be inferred from the properties of p.s.h. functions.

\begin{thm}\label{pre 5}
The canonical embedding of DSH$(\mathbb{P}^k)$ in $L^p(\mathbb{P}^k)$ is compact. Further, for a bounded subset $\mathcal{F}\subset \text{DSH}(\mathbb{P}^k)$, there are positive constants $\alpha$ and $A$ such that 
\[
\int_{\mathbb{P}^k} e^{\alpha \lvert U \rvert} \omega \leq A
\]
for all $U\in \mathcal{F}$.
\end{thm}

We deduce the following useful corollary.

\begin{cor}\label{pre 6}
Let $\mathcal{F}$ be a bounded subset in DSH$(\mathbb{P}^k)$. If $\mu$ is a probability measure associated with a bounded form of maximal degree on $\mathbb{P}^k$, then there is a positive constant $A$ such that 
\[
\langle \mu, \lvert U \rvert\rangle \leq A(1+ \log^+ {\lVert \mu \rVert}_\infty)
\]
for all $U\in \mathcal{F}$ where $\log^+:=\max \{\log,0\}$.
\end{cor}

\thispagestyle{empty}

\thispagestyle{empty}
\chapter{Dynamics of skew products of H\'{e}non maps}
In this chapter, we study the dynamics of skew products of H\'{e}non maps. This work can also be found in \cite{KP}.
\section{Existence of filtrations}
\no The existence of a filtration $V^{\pm}_R, V_R$ for a H\'{e}non map is useful in localizing its dynamical behavior. To study a family of such maps, it is therefore essential to first establish the existence of a uniform filtration that works for all of them. Let
\begin{align*}
V_R^+ &= \big\{ (x, y) \in \mbb C^2 : \vert y \vert > \vert x \vert, \vert y \vert > R \big\},\\
V_R^- &= \big\{ (x, y) \in \mbb C^2 : \vert y \vert < \vert x \vert, \vert x \vert > R \big\} \text{ and }\\
V_R   &= \big\{ (x, y) \in \mbb C^2 : \vert x \vert, \vert y \vert \le R\}
\end{align*}
be a filtration of $\mathbb{C}^2$ where $R$ is large enough so that
\[
H_{\la}(V_R^+) \subset V_R^+
\]
for each $\la \in M$. The existence of such a $R$ is shown in the following lemma.
\subsection{A lemma on uniform filtration}
\begin{lem} \label{le1}
There exists $R>0$ such that
$$
H_\lambda(V_R^+)\subset V_R^+, \ \ H_\lambda(V_R^+\cup V_R)\subset V_R^+\cup V_R
$$
and
$$
H_\lambda^{-1}(V_R^-)\subset V_R^-, \ \ H_\lambda^{-1}(V_R^-\cup V_R)\subset V_R^-\cup V_R
$$
for all $\lambda \in M$. Furthermore,
$$
I_\lambda^{\pm}=\mathbb{C}^2\setminus K_\lambda^{\pm}=\bigcup_{n=0}^\infty (H_{\la}^{\pm n})^{-1}(V_R^{\pm}).
$$
\end{lem}
\begin{proof}
Let
\[
p_{j,\lambda}(y)=y^{d_j} + c_{\lambda(d_j-1)}y^{d_j-1} + \cdots + c_{\lambda 1}y + c_{\lambda 0}
\]
be the polynomial that occurs in the definition of $H_\lambda^{(j)}$. Then

\begin{equation}
\vert y^{-d_j} p_{j, \la}(y) - 1 \vert \le \vert c_{\la(d_j - 1)} y^{-1} \vert + \cdots + \vert c_{\la 1} y^{-d_j + 1} \vert + \vert c_{\la 0} y^{-d_j} \vert. \label{3}
\end{equation}

Let $a=\sup_{\lambda,j}|a_{\lambda,j}|$. Since the coefficients of $p_{j,\lambda}$ are continuous on $M$, which is assumed to be compact, and $d_j \ge 2$ it follows that there exists
$R>0$ such that
\[
\vert p_{j,\lambda}(y)  \vert \geq (2 + a) \vert y \vert
\]
for $\vert y \vert>R$, $\lambda\in M$ and $1\leq j \leq m$. To see that $H_\lambda(V_R^+)\subset V_R^+$ for this $R$, pick $(x,y)\in V_R^+$. Then
\begin{equation}
\lvert p_{j,\lambda}(y)-a_j(\lambda)x\rvert \geq \lvert p_{j,\lambda}(y)\rvert -\lvert a_j(\lambda)x\rvert
\geq \lvert y \rvert \label{4}
\end{equation}
for all $1\leq j \leq m$. It follows that the second coordinate of each $H_\lambda^{(j)}$ dominates the first one. This implies that
\[
H_\lambda(V_R^+)\subset V_R^+
\]
for all $\lambda\in M$. The other invariance properties follow by using similar arguments.

\medskip

Let $\rho>1$ be such that
\[
\lvert p_{j,\lambda}(y)-a_j(\lambda)x \rvert > \rho \lvert y \rvert
\]
for $(x,y)\in \overline{V_R^+}$, $\lambda\in M$ and $1\leq j \leq m$. That 
 a $\rho$ exists follows from  (\ref{4}). By letting $\pi_1$ and $\pi_2$ be the projections on the first
and second coordinate respectively, one can conclude inductively that
\begin{equation}
H_\lambda(x,y)\in V_R^+ \text{ and } \vert \pi_2(H_\lambda(x,y)) \vert >\rho^m \vert y \vert. \label{5}
\end{equation}
Analogously, for all $(x,y)\in \overline{V_R^{-}}$ and for all $\lambda\in M$, there exists a  $\rho>1$ satisfying
\begin{equation}
H_\lambda^{-1}(x,y)\in V_R^- \text{ and }|\pi_1(H_\lambda^{-1}(x,y))|>\rho^m|x|.\label{6}
\end{equation}
These two facts imply that
\begin{equation}
\overline{V_R^+} \subset H_{\la}^{-1}(\overline{V_R^+})\subset H_{\la}^{-1} \circ H_{\si(\la)}^{-1} (\ov{V_R^+}) \subset  \cdots \subset (H_{\la}^{+n})^{-1}(\overline{V_R^+})\subset
\cdots
\end{equation}
and
\begin{equation}
\overline{V_R^{-}} \supset H_{\la}^{-1}(\overline{V_R^{-}})\supset H_{\la}^{-1} \circ H_{\si(\la)}^{-1}(\ov{V_R^-}) \supset \cdots \supset (H_{\la}^{+n})^{-1}(\overline{V_R^-})\supset
\cdots .
\label{7}
\end{equation}

\medskip

At this point one can observe that if we start with a point in $\overline{V_R^+}$, it eventually escapes toward the point at infinity under forward iteration determined by the
continuous function $\sigma$, i.e., $\vert H_{\la}^{+n}(x, y) \vert \rightarrow \infty$ as $n\rightarrow \infty$. This can be
justified by using (\ref{5}) and observing that
\begin{equation*}
\lvert y_\lambda^n \rvert > \rho^m \lvert y_\lambda^{n-1} \rvert> \rho^{2m}\lvert y_\lambda^{n-2} \rvert> \cdots >\rho^{nm}\lvert y \rvert>\rho^{nm}R
\end{equation*}
where $H_{\la}^{+n}(x, y) =(x_\lambda^n,y_\lambda^n)$. A similar argument shows that if we start with any point in $(x,y)\in
\bigcup_{n=0}^{\infty} (H_{\la}^{+n})^{-1}(V_R^+)$, the orbit of the point never remains bounded. Therefore
\begin{equation}
\bigcup_{n=0}^{\infty} (H_{\la}^{+n})^{-1}(V_R^+)\subseteq I_\lambda^+.
\end{equation}
 Moreover using (\ref{5}) and (\ref{6}), we get
\[
(H_{\la}^{-n})^{-1}(V_R^+)\subseteq \big\{(x,y):\lvert y\rvert > \rho^{nm}R \big\}
\]
and
\[
(H_{\la}^{+n})^{-1}(V_R^-)\subseteq \big\{(x,y):\lvert x \rvert > \rho^{nm}R \big\}
\]
which give
\begin{equation}
\bigcap_{n=0}^{\infty} (H_{\la}^{-n})^{-1}(V_R^+) = \bigcap_{n=0}^{\infty} (H_{\la}^{-n})^{-1}(\overline{V_R^+})=\phi \label{8}
\end{equation}
and
\begin{equation}
\bigcap_{n=0}^{\infty} (H_{\la}^{+n})^{-1}(V_R^-)= \bigcap_{n=0}^{\infty} (H_{\la}^{+n})^{-1}(\overline{V_R^{-}})=\phi \label{9}
\end{equation}
respectively. Set
\[
W_R^+=\mathbb{C}^2\setminus \overline{V_R^{-}} \text{ and }W_R^-=\mathbb{C}^2\setminus \overline{V_R^+}.
\]
Note that (\ref{7}) and (\ref{9}) are equivalent to
\begin{equation}
W_R^+\subset H_{\lambda}^{-1}(W_R^+) \subset \cdots \subset (H_{\la}^{+n})^{-1}(W_R^+)\subset \cdots
\end{equation}
and
\begin{equation}
\bigcup_{n=0}^{\infty} (H_{\la}^{+n})^{-1}(W_R^+)= \mathbb{C}^2 \label{10}
\end{equation}
respectively. Now (\ref{10}) implies that for any point $(x,y)\in \mathbb{C}^2$, there exists $n_0>0$ such that $H_{\la}^{+n}(x,y)\in
W_R^+\subset V_R\cup \overline{V_R^+}$ for all $n\geq n_0$. So either
\[
H_{\la}^{+n}(x,y)\in V_R
\]
for all $n \ge n_0$ or there exists $n_1 \geq n_0$ such that $H_{\la}^{+n_1}(x,y)\in \overline{V_R^+}$. In the latter case,
$H_{\la}^{+(n_1+1)}(x,y)\in V_R^+$ by (\ref{5}). This implies that
\begin{equation*}
I_\lambda^{+}=\mathbb{C}^2\setminus K_\lambda^{+}=\bigcup_{n=0}^\infty (H_{\la}^{+n})^{-1}(V_R^{+}).\label{11}
\end{equation*}
A set of similar arguments yield
\begin{equation*}
I_\lambda^{-}=\mathbb{C}^2\setminus K_\lambda^{-}=\bigcup_{n=0}^\infty (H_{\la}^{-n})^{-1}(V_R^{-}).
\end{equation*}
\end{proof}

\begin{rem}\label{re1}
It follows from Lemma \ref{le1} that for any compact $A_\lambda \subset \mathbb{C}^2$ satisfying $A_\la \cap K_\lambda^+=\phi$, there exists $N_\lambda>0$
such that $H_{\la}^{+n_{\la}}(A_\lambda)\subseteq V_R^+$. More generally, for any compact $A \subset \mathbb{C}^2$
that satisfies $A\cap K_\lambda^+=\phi$ for each $\lambda\in M$,
there exists $N>0$ so that $H_{\la}^{+N}(A)\subseteq V_R^+$ for all $\lambda\in M$. The proof again relies on the fact that
the coefficients of $p_{j,\lambda}$ and $a_j(\lambda)$
vary continuously in $\lambda$ on the compact set $M$ for all $1 \le j \le m$.
\end{rem}

\begin{rem}\label{re2}
By applying the same kind of techniques as in the case of a single H\'{e}non map, it is possible to show that $I_\lambda^{\pm}$ are nonempty, pseudoconvex domains and $K_\lambda^{\pm}$
are closed sets satisfying $K_\lambda^{\pm}\cap V_R^{\pm}=\phi$ and having nonempty intersection with the $y$-axis and $x$-axis respectively. In particular, $K_\lambda^{\pm}$ are
nonempty and unbounded.
\end{rem}

\section{Fibered Green functions and Green currents}
This section proves Proposition \ref{pr1} and Proposition \ref{pr2} stated in the Introduction. 
\subsection{Proof of Proposition \ref{pr1}}

\begin{proof}
Since the polynomials $p_{j, \la}$ are all monic, it follows that for every small $\ep_1 > 0$ there is a large enough $R > 1$ so that for all
$(x,y)\in \overline{V_R^+}$, $1\leq j \leq m$ and for all $\lambda\in M$, we have $H_\lambda^{(j)}(x,y)\in V_R^+$
and
\begin{equation}
(1-\ep_1)\lvert y \rvert^{d_j}<\lvert \pi_2\circ H_\lambda^{(j)}(x,y)\rvert < (1+\ep_1)\lvert y \rvert^{d_j}. \label{12}
\end{equation}
For a given $\ep > 0$, choose an $\epsilon_1>0$ small enough so that the constants
\[
A_1=\prod_{j=1}^m (1-\epsilon_1)^{d_{j+1} \cdots d_m} \text{ and }
A_2=\prod_{j=1}^m (1+\epsilon_1)^{d_{j+1} \cdots d_m}
\]
(where $d_{j+1} \cdots d_m=1$ by definition when $j=m$) satisfy $1-\epsilon \leq A_1$ and $A_2 \leq 1+\epsilon$. Therefore by applying (\ref{12}) inductively, we get
\begin{equation}
(1-\epsilon)\lvert y \rvert^{d} \leq A_1\lvert y \rvert^{d}<\lvert \pi_2\circ H_\lambda(x,y)\rvert<A_2\lvert y \rvert^{d}\leq (1+\epsilon)\lvert y \rvert^{d}  \label{13} 
\end{equation} 
for all $\lambda\in M$ and for all $(x,y)\in \overline{V_R^+}$. Let $(x,y)\in \overline{V_R^+}$. In view of (\ref{5}) there exists a large $R>1$ so that
$H_\lambda^{+n}(x,y)=(x_\lambda^n,y_\lambda^n)\in V_R^+$ for all $n\geq 1$ and for all $\lambda\in M$. Therefore
$$
G_{n,\lambda}^+(x,y)=\frac{1}{d^n}\log\lvert \pi_2\circ H_\lambda^{+n}(x,y)\rvert
$$
and by applying (\ref{13}) inductively we obtain
\begin{equation*}
(1-\epsilon)^{1+d+ \cdots +d^{n-1}} \lvert y \rvert^{d^n}<\lvert y_\lambda^n \rvert<
(1+\epsilon)^{1+d+ \cdots +d^{n-1}}\lvert y \rvert^{d^n}.
\end{equation*}
Hence
\begin{equation}
0<\log\lvert y \rvert+K_1<G_{n,\lambda}^+(x,y)=\frac{1}{d^n}\log\lvert \pi_2\circ H_\lambda^{+n}(x,y)\rvert<\log\lvert y \rvert+K_2,\label{14}
\end{equation}
with $K_1= \frac{d^n-1}{d^n(d-1)} \log(1-\epsilon)$ and $K_2= \frac{d^n-1}{d^n(d-1)} \log(1+\epsilon)$.

\medskip

By (\ref{14}) it follows that
$$
\lvert G_{n+1,\lambda}^+(x,y)-G_{n,\lambda}^+(x,y)\rvert=\left | d^{-n -1} \log\lvert {y_\lambda^{n+1}}/{(y_\lambda^n)^d}\rvert\right | \lesssim d^{-n-1}
$$
which proves that $\{G_{n,\lambda}^+\}$ converges uniformly on $\overline{V_R^+}$. As a limit of a sequence of uniformly convergent pluriharmonic functions $\{G_{n,\lambda}^+\}$,
$G_\lambda^+$ is also pluriharmonic for each $\lambda\in M$ on $V_R^+$. Again by (\ref{14}), for each $\lambda\in M$,
\[
G_\lambda^+-\log\lvert y \rvert
\]
is a bounded pluriharmonic function in $\overline{V_R^+}$. Therefore its restriction to vertical lines of the form $x = c$ can be continued across the point $(c, \infty)$
as a pluriharmonic function. Since
\[
\lim_{\lvert y \rvert\rightarrow \infty}(G_\lambda^+(x,y)-\log\lvert y \rvert)
\]
is bounded in $x\in \mathbb{C}$, by (\ref{14}) it follows that $\lim_{\lvert y \rvert\rightarrow \infty}(G_\lambda^+(x,y)-\log\lvert y \rvert)$ must be a constant, say $\gamma_\lambda$ which also satisfies
$$
\log (1-\epsilon)/(d-1) \leq \gamma_\lambda \leq
\log (1+\epsilon)/(d-1).
$$
As $\epsilon > 0$ is arbitrary, it follows that
\begin{equation}
G_\lambda^+(x,y)=\log\lvert y \rvert + u_\lambda(x,y) \label{15}
\end{equation}
on $V_R^+$ where $u_\lambda$ is a bounded pluriharmonic function satisfying $u_\lambda(x,y) \ra 0$ as $\vert y \vert \ra \infty$.

\medskip

Now fix $\lambda\in M$ and $n\geq 1$. For any $r > n$
\begin{eqnarray*}
G_{r,\lambda}^+(x,y)&=& d^{-r}{\log}^+\lvert H_\lambda^{+r}(x,y)\rvert \\
&=& d^{-n}G_{(r-n),\sigma^n(\lambda)}^+\circ H_\lambda^{+n}(x,y).
\end{eqnarray*}

As $r\rightarrow \infty$, $G_{r,\lambda}^+$ converges uniformly on $(H_{\la}^{+n})^{-1}(V_R^+)$ to the pluriharmonic function  $d^{-n} G_{\sigma^n(\lambda)}^+ \circ H_\lambda^{+n}$.
Hence
$$
d^n G_\lambda^+(x,y)=G_{\sigma^n(\lambda)}^+\circ H_\lambda^{+n}(x,y)
$$
for $(x,y)\in (H_\lambda^{+n})^{-1}(V_R^+)$. By (\ref{14}), for $(x,y)\in (H_\lambda^{+n})^{-1}(V_R^+)$
$$
G_{r,\lambda}^+(x,y)=d^{-n}G_{(r-n),\sigma^n(\lambda)}^+\circ H_\lambda^{+n}(x,y)> d^{-n}(\log R + K_1)> 0,
$$
for each $r>n$ which shows that
\[
G_\lambda^+(x,y)\geq d^{-n}(\log R + K_1)>0
\]
for $(x,y)\in (H_\lambda^{+n})^{-1}(V_R^+)$. This is true for each $n\geq 1$. Hence $G_{r,\lambda}^+$ converges uniformly to the pluriharmonic function $G_\lambda^+$ on every compact
set
of
\[
\bigcup_{n=0}^\infty (H_\lambda^{+n})^{-1}(V_R^+)=\mathbb{C}^2\setminus K_\lambda^+.
\]
Moreover $G_\lambda^+ >0$ on $\mathbb{C}^2\setminus K_\lambda^+$.

\medskip

Note that for each $\lambda\in M$, $G_\lambda^+=0$ on $K_\lambda^+$.
By Remark \ref{re2}, there exists a large enough $R>1$ so that $K_\lambda^+\subseteq V_R \cup V_R^-$ for all $\lambda\in M$. Now choose any $A>R>1$. We will show that
$\{G_{n,\lambda}^+\}$ converges uniformly to $G_\lambda^+$ on the bidisc
\[
\Delta_A=\{(x,y):\lvert x \rvert\leq A,\lvert y \rvert\leq A\}
\]
as $n\ra\infty$. Consider the sets
\[
N=\{(x,y)\in \mathbb{C}^2: \lvert x \rvert \leq A\}, \;N_\lambda=N\cap K_\lambda^+
\]
for each $\lambda\in M$. Start with any point $z=(x_0,x_1)\in \mathbb{C}^2$ and define
$(x_i^\lambda,x_{i+1}^\lambda)$ for $\lambda\in M$ and $i\geq 1$ in the following way:
\[
(x_0^\lambda,x_1^\lambda) \xrightarrow{H_\lambda^{(1)}} (x_1^\lambda,x_2^\lambda) \xrightarrow{H_\lambda^{(2)}} \cdots \xrightarrow{H_\lambda^{(m)}}
(x_m^\lambda,x_{m+1}^\lambda)\xrightarrow{H_\lambda^{(1)}} (x_{m+1}^\lambda,x_{m+2}^\lambda)\ra \cdots ,
\]
where $(x_0^\lambda, x_1^\lambda)=(x_0,x_1)$ and we apply $H_\lambda^{(1)}, \cdots ,H_\lambda^{(m)}$ periodically for all $\lambda\in M$. Inductively one can show that if
$(x_i^\lambda,x_{i+1}^\lambda)\in N_\lambda$ for $0\leq i \leq j-1$, then $\lvert x_i^\lambda\rvert \leq A $ for $0\leq i \leq j$.


\medskip

This implies that there exists $n_0>0$ independent of $\lambda$ so that
\begin{equation}
G_{n,\lambda}^+(x,y)< \epsilon \label{16}
\end{equation}
for all $n\geq n_0$ and for all $(x,y)\in N_\lambda$.
Consider a line segment
\[
L_a=\{(a,w):\lvert w \rvert\leq A\} \subset \mbb C^2
\]
with $\lvert a \rvert \leq A$. Then $G_{n,\lambda}^+-G_\lambda^+$ is harmonic on
$L_a^\lambda=\{(a,w):\lvert w \rvert < A\}\setminus K_\lambda^+$ viewed as a subset of $\mathbb{C}$ and the boundary of $L_a^\lambda$ lies in $\{(a,w):\lvert w \rvert=A\}\cup
(K_\lambda^+\cap L_a)$. By Remark \ref{re1}, there exists $n_1>0$ so that
$$
-\epsilon< G_{n,\lambda}^+(a,w)-G_{\lambda}^+(a,w)<\epsilon
$$
for all $n\geq n_1$ and for all $(a,w)\in\{\lvert a \rvert\leq A,\lvert w \rvert=A\}$. The maximum principle shows that
$$
-\epsilon <G_{n,\lambda}^+(x,y)-G_\lambda^+(x,y) < \epsilon
$$
for all  $n\geq \max\{n_0,n_1\}$ and for all  $(x,y)\in L_a^\lambda$. This shows that for any given $\epsilon>0$, there exists $n_2>0$ such that
\begin{equation}
-\epsilon< G_{n,\lambda}^+(z)-G_\lambda^+(z)<\epsilon \label{17}
\end{equation}
for all $n\geq n_2$ and for all $(\lambda,z)\in M\times \Delta_A$.

\medskip

Hence $G_{n,\lambda}^+$ converges uniformly to $G_\lambda^+$ on any compact subset of $\mathbb{C}^2$ and this is also uniform with respect to $\lambda\in M$. In particular, this implies that for each $\lambda\in M$, $G_\lambda^+$ is continuous  on $\mathbb{C}^2$ and pluriharmonic on $\mathbb{C}^2\setminus K_\lambda^+$. Moreover $G_\lambda^+$ vanishes on $K_\lambda^+$. In particular, for each $\lambda\in M$, $G_\lambda^+$ satisfies the submean value property on $\mathbb{C}^2$. Hence $G_\lambda^+$ is plurisubharmonic on $\mathbb{C}^2$. Analogous results are true for $G_\lambda^-$.

\medskip

Next, to show that the correspondence $\lambda \mapsto G_\lambda^{\pm}$ is continuous, take a compact set $S\subset \mathbb{C}^2$ and $\lambda_0\in M$. Then
\begin{multline*}
\vert G_{\la}^+(x,y)-G_{\la_0}^+(x,y) \vert \le \vert G_{n,\la}^+(x,y)-G_{\la}^+(x,y)\vert + \vert G_{n,\lambda}^+(x,y)-G_{n,\lambda_0}^+(x,y)\vert \\
                                              + \vert G_{n,\lambda_0}^+(x,y)-G_{\lambda_0}^+(x,y)\vert
\end{multline*}
for $(x,y)\in S$. It follows from (\ref{17}) that for given $\epsilon>0$, one can choose a large $n_0>0$ such that the first and third  terms above are less that $\ep/3$. By
choosing $\lambda$ close enough to $\lambda_0$, it follows that $G_{n_0,\lambda}^+(x,y)$ and $G_{n_0,\lambda_0}^+(x,y)$ do not differ by more than $\ep/3$. Hence the correspondence
$\lambda\mapsto G_\lambda^{+}$ is continuous. Similarly, the correspondence $\lambda\mapsto G_\lambda^-$ is also continuous.

\medskip

To prove that $G_\lambda^+$ is H\"{o}lder continuous for each $\lambda\in M$, fix a compact $S \subset \mbb C^2$ and let $R > 1$ be such that $S$ is compactly contained in $V_R$. Using
the continuity of $G_{\la}^+$ in $\la$, there exists a $\de > 0$ such that $G_{\la}^+(x, y) > (d + 1)\de$ for each $\la \in M$ and $(x, y) \in V_R^+$.
Now note that the correspondence $\la \mapsto K_{\la}^+ \cap V_R$ is upper semi-continuous. Indeed, if this is not the case, then there exists a $\la_0 \in M$, an $\ep > 0$ and a sequence
$\la_n \in M$ converging to $\la_0$ such that for each $n \ge 1$, there exists a point $a_n \in K_{\la_n}^+ \cap V_R$ satisfying $\vert a_n - z \vert \ge \ep$ for all $z \in
K_{\la_0}^+$. Let $a$ be a limit point of the $a_n$'s. Then by the continuity of $\la \mapsto G_{\la}^+$ it follows that
\[
0 = G_{\la_n}^+(a_n) \ra G_{\la}^+(a)
\]
which implies that $a \in K_{\la_0}^+$. This is a contradiction. For each $\lambda\in M$, define
\[
\Omega_\delta^\lambda= \big\{ (x,y)\in V_R : \delta < G_\lambda^+(x,y) \leq d \delta \big\}
\]
and
\[
C_\lambda=\sup\big\{ \lvert {\partial G_\lambda^+}/{\partial x}\rvert,\lvert {\partial G_\lambda^+}/{\partial y}\rvert :(x,y)\in \Omega_\delta^\lambda \big\}.
\]
The first observation is that the $C_{\la}$'s are uniformly bounded above as $\la$ varies in $M$. To see this, fix $\la_0 \in M$ and $\tau > 0$ and let $W \subset M$ be a neighbourhood
of $\la_0$ such that the sets
\[
\Om_W = \ov{\bigcup_{\la \in W} \Om_{\de}^{\la}} \;\; \text{and} \;\; K_W = \ov{ \bigcup_{\la \in W} (K_{\la}^+ \cap V_R)}
\]
are separated by a distance of at least $\tau$. This is possible since $K_{\la}^+ \cap V_R$ is upper semi-continuous in $\la$. For each $\la \in W$, $G_{\la}^+$ is pluriharmonic on a
fixed slightly larger open set containing $\Om_W$. Cover the closure of this slightly larger open set by finitely many open balls and on each ball, the mean value property shows that
the derivatives of $G_{\la}^+$ are dominated by a universal constant times the sup norm of $G_{\la}^+$ on it -- and this in turn is dominated by the number of open balls (which is the
same for all $\la \in W$) times the sup norm of $G_{\la}^+$ on $V_R$ upto a universal constant. Since $G_{\la}^+$ varies continuously in $\la$, it follows that the $C_{\la}$'s are
uniformly bounded for $\la \in W$ and the compactness of $M$ gives a global bound, say $C > 0$ independent of $\la$.

\medskip

Fix $\la_0 \in M$ and pick $(x, y) \in S \setminus K_{\la_0}^+$. Let $N > 0$ be such that
\[
d^{-N} \de < G_{\la_0}^+(x, y) \le d^{-N + 1} \de
\]
so that
\[
\de < d^N G_{\la_0}^+(x, y) \le d \de.
\]
The assumption that $N > 0$ means that $(x, y)$ is very close to $K_{\la_0}^+$. But
\[
d^N G_{\la_0}^+(x, y) = G_{\si^N(\la_0)}^+ \circ H_{\la_0}^{+N}(x, y)
\]
which implies that $H_{\la_0}^{+N}(x, y) \in \Om_{\de}^{\si^N(\la_0)}$ where $G_{\si^N(\la_0)}^+$ is pluriharmonic. Note that
\[
H_{\la_0}(V_R \cup V_R^+) \subset V_R \cup V_R^+, \; H_{\la_0}(V_R^+) \subset V_R^+
\]
which shows that $H_{\la_0}^{+k} \in V_R$ for all $k \le N$ since all the $G_{\la}^+$'s are at least $(d+1)\de$ on $V_R^+$.
Differentiation of the above identity leads to
\[
d^N \frac{\pa G_{\la_0}^+}{\pa x}(x, y) = \frac{\pa G_{\si^N(\la_0)}^+}{\pa x} (H_{\la_0}^{+N}) \frac{ \pa (\pi_1 \circ H_{\la_0}^{+N}) }{\pa x}(x, y) + \frac{\pa
G_{\si^N(\la_0)}^+}{\pa y}(H_{\la_0}^{+N})  \frac{ \pa (\pi_2 \circ H_{\la_0}^{+N}) }{\pa x}(x, y).
\]
Let the derivatives of $H_{\la}$ be bounded above on $V_R$ by $A_{\la}$ and let $A = \sup A_{\la} < \infty$. It follows that the derivatives of $H_{\la_0}^{+N}$ are bounded above by
$2^{N-1}A^N$ on $V_R$. Hence
\[
\vert d^N \pa G_{\la_0}^+ / \pa x (x, y) \vert \le C (2A)^N.
\]
Let $\ga = \log 2A/ \log d$ so that $C (2A)^N = C d^{N \ga}$. Therefore
\[
\vert \pa G_{\la_0}^+ / \pa x \vert \le C d^{N(\ga - 1)} \le C (d \de/G_{\la_0}^+)^{\ga - 1}
\]
which implies that
\[
\vert \pa (G_{\la_0}^+)^{\ga}/ \pa x \vert \le C \ga(d \de)^{\ga - 1}.
\]
A similar argument can be used to bound the partial derivative of $(G_{\la_0}^+)^{\ga}$ with respect to $y$. Thus the gradient of $(G_{\la_0}^+)^{\ga}$ is bounded uniformly at all
points that are close to $K_{\la_0}^+$.

\medskip

Now suppose that $(x, y) \in S \setminus K_{\la_0}^+$ is such that
\[
d^{N} \de < G_{\la_0}^+(x, y) \le d^{N + 1} \de
\]
for some $N > 0$. This means that $(x, y)$ is far away from $K_{\la_0}^+$ and the above equation can be written as
\[
\de < d^{-N} G_{\la_0}^+(x, y) \le d \de.
\]
By the surjectivity of $\si$, there exists a $\mu_0 \in M$ such that $\si^N(\mu_0) = \la_0$. With this the invariance property of the Green's functions now reads
\[
G_{\mu_0}^+ \circ (H_{\mu_0}^{+N})^{-1}(x, y) = d^{-N} G_{\la_0}^+(x, y).
\]
The compactness of $S$ shows that there is a fixed integer $m < 0$ such that if $(x, y)$ is far away from $S \setminus K_{\la_0}^+$, then it  can be brought into the strip
\[
\big\{ (x,y) : \delta < G_{\lambda_0}^+(x,y) \leq d \delta \big\}
\]
by $(H_{\la}^{+ \vert k \vert})^{-1}$ for some $m \le k < 0$ and for all $\la \in M$. By enlarging $R$, we may assume that the image of $S$ under all the maps $(H_{\la}^{+ \vert k
\vert})^{-1}$, $m \le k < 0$, is contained in $V_R$. By increasing $A$, we may also assume that all the derivatives of $H_{\la}$ and $H_{\la}^{-1}$ are bounded by $A$ on $V_R$.
Now repeating the same argument as above, it follows that the gradient of $(G_{\la_0}^+)^{\ga}$ is bounded uniformly at all points that are far away from $K_{\la_0}^+$ -- the nuance
about choosing $\ga$ as before is also valid. The choice of $\mu_0$ such that $\si^{N}(\mu_0) = \la_0$ is irrelevant since the derivatives involved are with respect to $x, y$ only.
The only remaining case is when $(x, y) \in \Om_{\la_0}^{\de}$ which precisely means that $N = 0$. But in this case, $(G_{\la_0}^+)^{\ga - 1}$ is uniformly bounded
on $V_R$ and so are the derivatives of $G_{\la_0}^+$ on $\Om_{\la_0}^{\de}$ by the reasoning given earlier. Therefore there is a uniform bound on the gradient of $(G_{\la_0}^+)^{\ga}$
everywhere on $S$. This shows that $(G_{\la_0}^+)^{\ga}$ is Lipschitz on $S$ which implies that $G_{\la_0}^+$ is H\"{o}lder continuous on $S$ with exponent $1/\ga = \log d/ \log 2A$.
A set of similar arguments can be applied to deduce analogous results for $G_{\la}^{-}$.
\end{proof}


\subsection{Proof of Proposition \ref{pr2}}
\begin{proof}
We have
\[
(H_{\la}^{\pm 1})^{\ast}(\mu_{\sigma(\lambda)}^\pm) = (H_{\la}^{\pm 1})^{\ast}(dd^cG_{\sigma(\lambda)}^\pm) = dd^c(G_{\sigma(\lambda)}^\pm \circ H_\lambda^{\pm 1}) = dd^c(d G_\lambda^\pm) = d\mu_\lambda^\pm
\]
where the third equality follows from Proposition \ref{pr1}. A similar exercise shows that
\[
(H_{\la}^{\pm 1})_{\ast} \mu_{\la}^{\pm} = d^{-1} \mu_{\si(\la)}^{\pm}.
\]
If $\si$ is the identity on $M$, then
\[
G_\lambda^\pm \circ H_\lambda^{\pm 1}=d G_\lambda^\pm,
\]
which in turn imply that
\[
(H_{\la}^{\pm 1})^{\ast} \mu_{\la} = (H_{\la}^{\pm 1})^{\ast} (\mu_{\la}^+ \wedge \mu_{\la}^-) = (H_{\la}^{\pm 1})^{\ast} \mu_{\la}^+ \wedge (H_{\la}^{\pm 1})^{\ast} \mu_{\la}^- =
d^{\pm 1} \mu_{\la}^+ \wedge d^{\mp 1} \mu_{\la}^- = \mu_{\la}.
\]

\medskip

By Proposition \ref{pr1}, the support of $\mu_\lambda^+$ is contained in $J_\lambda^+$. To prove the converse, let $z_0\in J_\lambda^+$ and suppose that $\mu_\lambda^+ =0$ on a
neighbourhood $U_{z_0}$ of $z_0$. This means that $G_\lambda^+$ is pluriharmonic on $U_{z_0}$ and $G_\lambda^+$ attains its minimum value zero at $z_0$. This implies
that $G_\lambda^+ \equiv 0$ on $U_{z_0}$ which contradicts the fact that $G_\lambda^+>0$ on $\mathbb{C}^2\setminus K_\lambda^+$. Similar arguments can be applied to prove that
supp$(\mu_\lambda^-)=J_\lambda^-$.

 \medskip

Finally, to show that $\la \mapsto J_{\la}^+$ is lower semi-continuous, fix $\la_0 \in M$ and $\ep > 0$. Let $x_0\in
J_{\lambda_0}^+= {\rm supp}(\mu_{\lambda_0}^+)$. Then $\mu_{\lambda_0}^+(B(x_0, {\epsilon}/{2}))\neq 0$. Since the correspondence $\lambda \mapsto \mu_\lambda^+$ is continuous, there
exists a $\delta>0$ such that
\[
d(\lambda,\lambda_0)<\delta \text{ implies } \mu_\lambda^+(B(x_0;{\epsilon}/{2}))\neq 0.
\]
Therefore $x_0\in {(J_\lambda^+)}^\epsilon=\bigcup_{a\in J_\lambda^+}B(a,\epsilon)$ for all $\lambda \in M$ satisfying $d(\lambda,\lambda_0)< \delta$.
Hence the correspondence $\lambda\mapsto J_\lambda^{\pm}$ is lower semi-continuous.

\medskip

Let $\mathcal L$ be the class of plurisubharmonic functions on $\mbb C^2$ of logarithmic growth, i.e.,
$$
\mathcal{L}=\{ u\in \mathcal{PSH}(\mathbb{C}^2): u(x,y)\leq \log^+\lVert (x,y) \rVert +L \}
$$
for some $L>0$ and let
$$
\tilde{\mathcal{L}}=\{ u\in \mathcal{PSH}(\mathbb{C}^2):\log^+\lVert (x,y) \rVert -L \leq u(x,y)\leq \log^+\lVert (x,y) \rVert +L\}
$$
for some $L>0$. Note that there exists $L>0$ such that
$$
G_\lambda^+(z)\leq \log^+ \lVert z \rVert +L
$$
for all $z\in \mathbb{C}^2$ and for all $\lambda\in M$. Thus $G_\lambda^+ \in \mathcal{L}$ for all $\lambda\in M$. For $E\subseteq \mathbb{C}^2$, the pluricomplex Green function
of $E$ is
$$
L_E(z)=\sup\{u(z):u\in\mathcal{L},u\leq 0 \text{ on } E\}
$$
and let $L_E^{\ast}(z)$ be its upper semi-continuous regularization.

\medskip

It turns out that the pluricomplex Green function of $K_\lambda^{\pm}$ is $G_\lambda^{\pm}$ for all $\lambda\in M$. The arguments are similar to those employed for a single H\'{e}non
map and we merely point out the salient features. Fix $\lambda\in M$. Then $G_{\lambda}^+=0$ on $K_\lambda^+$ and  $G_\lambda^+ \in \mathcal{L}$. So $G_\lambda^+ \leq L_{K_{\lambda}^+}$.
To show equality, let $u\in \mathcal{L}$
be such that $u\leq 0=G_\lambda^+$ on $K_\lambda^+$. By Proposition \ref{pr1}, there exists $M>0$ such that
\[
\log\lvert y \rvert-M<G_\lambda^+(x,y)<\log\lvert y \rvert+M
\]
for $(x,y)\in V_R^+$. Since $u\in \mathcal{L}$,
\[
u(x,y)-G_\lambda^+(x,y)\leq M_1
\]
for some $M_1 > 0$ and $(x,y)\in V_R^+.$

Fix $x_0 \in\mbb C$ and note that $u(x_0,y)-G_\lambda^+(x_0,y)$ is a bounded subharmonic function on the vertical line $T_{x_0}=\mathbb{C}\setminus (K_\lambda^+ \cap \{x=x_0\})$ and
hence it can be extended across the point $y=\infty$ as a subharmonic function. Note also that
$$
u(x_0,y)-G_\lambda(x_0,y)\leq 0
$$
on $\partial T \subseteq K_\lambda^+ \cap \{x=x_0\}$. By the maximum principle, it follows that $u(x_0,y)-G_\lambda(x_0,y)\leq 0$ on $T_{x_0}$. This implies that $u\leq G_\lambda^+
\text{ in } \mathbb{C}^2\setminus K_\lambda^+$ which in turn shows that
$L_{K_{\lambda}^{+}}=G_{\lambda}^{+}$. Since $G_\lambda^+$ is continuous on $\mathbb{C}^2$, we have
\[
L_{K_{\lambda}^{+}}=L^{\ast}_{K_{\lambda}^{+}}=G_\lambda^+.
\]
Similar arguments show that
\[
L_{K_{\lambda}^{-}}=L^{\ast}_{K_{\lambda}^{-}}=G_{\lambda}^{-}.
\]
Let $u_\lambda=\max \{G_\lambda^+,G_\lambda^-\}$. Again by Proposition \ref{pr1}, it follows that $u_\lambda\in \tilde{\mathcal{L}}$.
For $\epsilon>0$, set $G_{\lambda,\epsilon}^{\pm}=\max \{G_\lambda^{\pm},\epsilon\}$ and $u_{\lambda,\epsilon}=\max \{G_{\lambda,\epsilon}^+, G_{\lambda,\epsilon}^{-}\}$.
By Bedford--Taylor,
\[
{(dd^c u_{\lambda,\epsilon})}^2=dd^c G_{\lambda,\epsilon}^+ \wedge dd^c G_{\lambda,\epsilon}^{-}.
\]
Now for a $z\in \mathbb{C}^2 \setminus K_\lambda^{\pm}$ ,
there exists a small neighborhood $\Omega_{z}\subset \mathbb{C}^2\setminus K_\lambda^{\pm}$ of $z$ such that
${(dd^c u_{\lambda,\epsilon})}^2=0$ on $\Omega_z$ for sufficiently small $\epsilon$. It follows that supp${((dd^c u_\lambda))}^2 \subset K_\lambda$.


\medskip

Since $G_\lambda^{\pm}=L^{\ast}_{{K_\lambda}^{\pm}} \leq L^{\ast}_{K_\lambda}$, we have $u_\lambda\leq L^{\ast}_{K_\lambda}$. Further, note that $L^{\ast}_{K_\lambda} \leq
L_{K_\lambda}\leq 0=u_\lambda$ almost every where on $K_\lambda$ with respect to the measure ${(dd^c u_\lambda )}^2$. This is because the set
$\{L_{K_\lambda}^* > L_{K_\lambda}\}$ is pluripolar and consequently has measure zero with respect to ${(dd^c u_\lambda)}^2$. Therefore $L^{\ast}_{K_\lambda}\leq u_\lambda$ in
$\mathbb{C}^2$. Finally, $L_{K_\lambda}$ is continuous and thus $L^{\ast}_{K_\lambda}=L_{K_\lambda}=\max \{G_\lambda^+, G_\lambda^-\}$.

For a non-pluripolar bounded set $E$ in $\mathbb{C}^2$, the complex equilibrium measure is $\mu_E={(dd^c L^{\ast}_E)^2}$. Again by
Bedford--Taylor, $\mu_{K_\lambda}=
\lim_{\epsilon \ra 0}{(dd^c \max\{G_\lambda^+, G_\lambda^-,\epsilon\})}^2$ which when combined with
$$
\mu_\lambda=\mu_\lambda^+ \wedge \mu_\lambda^-= \lim_{\epsilon\ra 0}dd^c G_{\lambda,\epsilon}^+ \wedge dd^c G_{\lambda,\epsilon}^-
$$
and
$$
{(dd^c \max \{G_\lambda^+, G_\lambda^-,\epsilon\})}^2=dd^c G_{\lambda,\epsilon}^+ \wedge dd^c G_{\lambda,\epsilon}^-
$$
shows that $\mu_\lambda$ is the equilibrium measure of $K_\lambda$. Since supp$(\mu_\lambda^{\pm})=J_\lambda^{\pm}$, we have supp$(\mu_\lambda) \subset J_\lambda$.
\end{proof}

\section{Convergence of currents}
In this section, we prove Theorem \ref{thm1} and Proposition \ref{pr2}.  
\subsection{Proof of Theorem \ref{thm1}}
\begin{proof}
Let $\mathcal L_y$ be the subclass of $\mathcal L$ consisting of all those functions $v$ for which there exists $R > 0$ such that
\[
v(x, y) - \log \vert y \vert
\]
is a bounded pluriharmonic function on $V_R^+$.

\medskip

Fix $\lambda\in M$ and let $\omega= 1/4 \;dd^c  \log (1 + \Vert z \Vert^2)$. For a $(1, 1)$ test form $\varphi$ on $\mbb C^2$, it follows that  there exists a $C >0$ such that
\[
-C \Vert
\varphi \Vert \omega \leq \varphi \le C \Vert \varphi \Vert \omega
\]
by the positivity of $\om$.

\medskip
\no
{\it Step 1:} $S_{\la}$ is nonempty.\\

Note that
\begin{eqnarray}
\frac{1}{d^n}\left |
\int_{\mathbb{C}^2}(H_\lambda^{+n})^{\ast}(\psi T)\wedge \varphi
\right|
&\lesssim &\frac{\Vert \varphi \Vert}{d^n}\int_{\mathbb{C}^2}(H_\lambda^{+n})^{\ast}(\psi T)\wedge dd^c \log (1 + \Vert z \Vert^2) \nonumber \\
& \lesssim & \frac{\Vert \varphi \Vert}{d^n}\int_{\mathbb{C}^2}dd^c(\psi T)\wedge \log (1 + \Vert (H_{\la}^{+n})^{-1}(z) \Vert ).\label{18}
\end{eqnarray}

\medskip

A direct calculation shows that
\[
\frac{1}{d^n}\log^+ \| (H_\lambda^{+n})^{-1}(z) \| \leq \log^+|z|+C \label{19}
\]
for some $C>0$, for all $n\geq 1$, $\lambda\in M$ and
\begin{equation}
\log (1 + \Vert z \Vert^2) \leq 2 \log^+|z|+2\log 2.\label{20}
\end{equation}
It follows that
\begin{equation*}
0 \le \frac{1}{d^n} \log \left( 1 + \Vert (H_{\la}^{+n})^{-1} \Vert \right) \le 2 \log^+ \vert z \vert + C
\end{equation*}
for some $C>0$, for all $n>0$ and $\lambda\in M$. Hence
\begin{equation}
\frac{1}{d^n}\left | \int_{\mathbb{C}^2} (H_\lambda^{+n})^{\ast}(\psi T)\wedge \varphi \right| \lesssim \Vert \varphi \Vert. \label{21}
\end{equation}

\medskip

The Banach-Alaoglu theorem gives that there is a subsequence $\frac{1}{d^{n_j^\lambda}}(H_{\la}^{+n_j^{\la}})^{\ast}(\psi T)$ that converges in the sense of currents to a positive $(1,1)$-current, say
$\gamma_\lambda$. This shows that $S_\lambda$ is nonempty. It also follows from the above discussion that $\int_{\mbb C^2} \gamma_{\la} \wedge \om < + \infty$.

\medskip

\no {\it Step 2:} Each $\gamma_{\la} \in S_{\la}$ is closed. Further, the support of $\ga_{\la}$ is contained in $K_{\la}^+$.\\

Let $\chi$ be a smooth real $1$-form with compact support in $\mbb C^2$ and  $\psi_1 \ge 0$ be such that $\psi_1 = 1$ in a neighborhood of ${\rm supp}(\psi)$. Then
\[
\int_{\mathbb{C}^2}d \chi \wedge (H_\lambda^{+n_j^{\lambda}})^{\ast}(\psi T) =
\int_{\mathbb{C}^2}\chi \circ (H_\lambda^{+n_j^{\la}})^{-1} \wedge d\psi \wedge \psi_1 T.
\]
To obtain which the assumption that ${\rm supp}(\psi) \cap {\rm supp}(dT) = \phi$ is used. By the Cauchy-Schwarz inequality it follows that the term on the right above is dominated by
the square root of
\[
\left(\int_{\mathbb{C}^2} \big( (J \chi \wedge \chi)\circ (H_\lambda^{+n_j^{\la}})^{-1} \big) \wedge \psi_1 T\right)
\left( \int_{\mathbb{C}^2} d\psi \wedge d^c \psi \wedge \psi_1 T \right)
\]
whose absolute value in turn is bounded above by a harmless constant times $d^{n_j^\lambda}$. Here $J$ is the standard $\mbb R$-linear map on $1$-forms satisfying $J(d z_j) = i d \ov
z_j$ for $j = 1, 2$.
Therefore
\[
\left| \frac{1}{d^{n_j^{\la}}}
\int_{\mathbb{C}^2}(\chi \circ (H_{\la}^{+n_j^{\la}})^{-1} \wedge d\psi \wedge \psi_1 T
\right| \lesssim d^{- n_j^{\la} / 2}.
\]
Evidently, the right hand side tends to zero as $j \ra \infty$. This shows that $\gamma_\lambda$ is closed.

\medskip

Let $R>0$ be large enough so that supp$(\psi T)\cap V_R^+=\phi$. Let $z\notin K_\lambda^+$ and $B_z$ a small open ball around it such that
$\overline{B_z} \cap K_\lambda^+=\phi$. By Lemma \ref{le1}, there exists an $N>0$ such that $H_\lambda^{+n}(B_z)\subset V_R^+$ for all $n>N$. Therefore
$B_z \cap \text{supp}(H_\lambda^{+n})^{\ast}(\psi T)= B_z \cap (H_{\la}^{+n})^{-1}(\text{supp}(\psi T))=\phi$ for all $n>N$. Since supp$(\gamma_\lambda)\subset
\overline{\bigcup_{n=N}^\infty \text{supp}(H_\lambda^{+n})^{\ast}(\psi T)})$, we have $z\notin \text{supp}(\gamma_\lambda)$. This implies $\text{supp}(\gamma_\lambda)\subset
K_\lambda^+$. Since $K_\lambda^+\cap V_R^+=\phi$ for all $\lambda\in M$, it also follows that $\text{supp}(\gamma_\lambda)$ does not intersect $\ov{V_R^+}$.

\medskip

\no {\it Step 3:} Each $\ga_{\la}$ is a multiple of $\mu_{\la}^+$.

\medskip

It follows from Proposition 8.3.6 in \cite{MNTU} that $\ga_{\la} = c_{\ga, \la} dd^c U_{\ga, \la}$ for some $c_{\ga, \la} > 0$ and $U_{\ga, \la} \in \mathcal L_y$. In this
representation, $c_{\ga, \la}$ is unique while $U_{\ga, \la}$ is unique upto additive constants. We impose the following condition on $U_{\gamma,\lambda}$:
\[
\lim_{|y|\rightarrow \infty} (U_{\gamma,\lambda}-\log|y|)=0 \label{22}
\]
and this uniquely determines $U_{\ga, \la}$. It will suffice to show that $U_{\ga, \la} = G_{\la}^+$.

\medskip

Let $\gamma_{\lambda,x}$ denote the restriction of $\gamma_\lambda$ to the plane $\{(x,y):y\in \mathbb{C}\}$. Since $U_{\ga, \la} \in \mathcal L_y$, it follows that

\begin{equation}
\int_{\mathbb{C}}\gamma_{\lambda,x}=2\pi c_{\gamma,\lambda}, \;\; U_{\gamma,\lambda}(x,y)=\frac{1}{2\pi c_{\gamma,\lambda}} \int_{\mathbb{C}}\log |y-\zeta|\gamma_{\lambda,x}(\zeta).
\label{18}
\end{equation}

\medskip

Consider a uniform filtration $V^{\pm}_R, V_R$ for all the maps $H_{\la}$ where $R^d > 2R$ and $\vert p_{j, \la}(y) \vert \ge \vert y \vert^d / 2$ for $\vert y \vert \ge R$. Let $0
\not= a = \sup \vert a_j(\la) \vert < \infty$ (where the supremum is taken over all $1 \le j \le m$ and $\la \in M$) and choose $R_1 > R^d /2$. Define
\[
A = \big \{ (x, y) \in \mbb C^2 : \vert y \vert^d \ge 2(1 + a) \vert x \vert + 2 R_1 \big \}.
\]
Evidently $A \subset \{ \vert y \vert > R \}$. Lemma \ref{le1} shows that for all $\la \in M$, $H_{\la}(x, y) \subset V_R^+$ when $(x, y) \in A \cap V_R^+$. Furthermore, for $(x, y) \in
A \cap (\mbb C^2 \setminus V_R^+)$, it follows that
\[
\vert p_{j, \la}(y) - a_j(\la)x \vert \ge \vert y \vert^d / 2 - a \vert x \vert \ge \vert y \vert + R.
\]
This shows that $H_{\la}(A) \subset V_R^+$. By Lemma \ref{le1}, again it can be seen that $H_{\la}^{+n}(A) \subset V_R^+$ for all $n \ge 1$ which shows that $A \cap K_{\la}^+ = \phi$ for
all $\la \in M$. Let $C>0$ be such that
\[
C^d \geq \max\{2(1+\lvert a \rvert), 2R_1\}.
\]
If $|y|\geq C(\lvert x \rvert^{1/d}+1)$, then
\begin{equation*}
{|y|}^{d} \geq C^{d}(\lvert x\rvert+1) \geq 2(1+\lvert a \rvert)\lvert x \rvert + 2R_1
\end{equation*}
which implies that
\[
B= \big\{ (x,y)\in \mathbb{C}^2: |y|\geq C(\lvert x \rvert^{1/d}+1) \big\}\subset A
\]
and hence $K_\lambda^+ \cap B=\phi $. Since $V_R^+ \subset B $ for sufficiently large $R$, by applying Lemma \ref{le1} once again it follows that
\begin{equation}
K_\lambda^+\cap B=\phi \text{ and } \bigcup_{n=0}^\infty (H_{\la}^{+n})^{-1}(B)=\mathbb{C}^2\setminus K_\lambda^+\label{23}
\end{equation}
for all $\lambda\in M$.

\medskip

Set $r=C(|x|^{1/d}+1)$. Since supp$(\gamma_\lambda)\subset K_\lambda^+$, it follows that
\[
\text{supp}(\gamma_{\lambda,x})\subset \{\lvert y \rvert \leq r\}
\]
for all $\lambda\in M$. Since
\[
\lvert y \rvert-r\leq \lvert y-\zeta\rvert\leq \lvert y \rvert+r
\]
for $\lvert y \rvert>r$ and $\lvert \zeta\rvert\leq r$, (\ref{18}) yields
\[
\log(\lvert y \rvert-r) \leq U_{\gamma,\lambda}(x,y) \leq \log(\lvert y \rvert+r)
\]
which implies that
\[
-(r/{\lvert y \rvert})/(1- r/{\lvert y \rvert})\leq U_{\gamma,\lambda}(x,y)- \log \lvert y \rvert \leq  r/{\lvert y \rvert}.
\]
Hence for $\lvert y \rvert > 2r$, we get
\begin{equation}
-2r/{\lvert y \rvert} \leq U_{\gamma,\lambda}(x,y)- \log \lvert y \rvert \leq r/{\lvert y \rvert} \label{24}
\end{equation}
for all $\lambda\in M$.

\medskip

For each $N \ge 1$, let $\gamma_\lambda(N) = d^{N}(H_{\la}^{+N})_{\ast}(\gamma_\lambda)$. Then
\[
\gamma_\lambda(N)=\lim_{j \ra \infty} d^{-n_j + N}\big( H_{\si^N(\la)}^{+(n_j - N)} \big)^{\ast}(\psi T) \in S_{\sigma^N(\lambda)}(\psi T).
\]
Therefore
\[
\ga_{\si^N(\la)} = c_{\gamma,\sigma^N(\lambda)}dd^c U_{\gamma,\sigma^N(\lambda)}
\]
for some  $c_{\gamma,\sigma^N(\lambda)}>0$ and $U_{\gamma,\sigma^N(\lambda)}\in \mathcal{L}_y$ and moreover,
\[
c_{\gamma,\lambda}dd^c U_{\gamma,\lambda} = \ga_{\la} = d^{-N} \big( H_{\la}^{+N}
\big)^{\ast} \ga_{\si^N(\la)} = c_{\gamma,\sigma^N(\lambda)}dd^c \big(d^{-N} \big(H_{\la}^{+N} \big)^{\ast} U_{\ga, \si^N(\la)} \big).
\]
Note that both $d^{-N} \big(H_{\la}^{+N} \big)^{\ast} U_{\ga, \si^N(\la)}$ and $U_{\gamma,\sigma^N(\lambda)}$ belong to $\mathcal{L}_y$. It follows that $c_{\ga,
\la} = c_{\ga, \si^N(\la)}$ and $d^{-N} \big(H_{\la}^{+N} \big)^{\ast} U_{\ga, \si^N(\la)}$ and $U_{\gamma,\lambda}$ coincide up to an additive constant which can be shown to be
zero as follows.

\medskip

By the definition of the class $\mathcal L_y$, there exists a pluriharmonic function $u_{\la, N}$ on some $V_R^+$ such that
\[
U_{\gamma,\sigma^N(\lambda)}(x,y)- \log \lvert y \rvert = u_{\lambda,N} \text{ and } \lim_{\lvert y \rvert \rightarrow \infty}u_{\lambda,N}(x,y)= u_0 \in \mathbb{C}.
\]
Therefore if $(x,y)\in (H_{\la}^{+N})^{-1}(V_R^+)$ and $(x_N^{\la}, y_N^{\la}) = H_{\la}^{+N}(x, y)$, then
\[
d^{-N} \big(H_{\la}^{+N} \big)^{\ast} U_{\ga, \si^N(\la)} (x, y) - d^{-N}\log \lvert y_N^\lambda \rvert = d^{-N}u_{\lambda, N}(x_N^\lambda,y_N^\lambda).
\]
By (\ref{15}), we have that
\[
d^{-N}\log\lvert y_N^\lambda\rvert - \log\lvert y \rvert \rightarrow 0
\]
as $\lvert y \rvert \ra \infty$ which shows that
\[
d^{-N} \big(H_{\la}^{+N} \big)^{\ast} U_{\ga, \si^N(\la)}(x, y) - \log\lvert y \rvert \rightarrow 0
\]
as $\vert y \vert \ra \infty$. But by definition,
\[
U_{\gamma,\lambda}(x,y) - \log\lvert y \rvert \rightarrow 0
\]
as $\lvert y \rvert\rightarrow \infty$ and this shows that $ d^{-N} \big(H_{\la}^{+N} \big)^{\ast} U_{\ga, \si^N(\la)} = U_{\gamma,\lambda}$.

\medskip

Let $(x,y)\in \mathbb{C}^2\setminus K_\lambda^+$ and $\epsilon>0$. For a sufficiently large $n$, $(x_n^\lambda, y_n^\lambda)=H_\lambda^{+n}(x,y)$ satisfies $\lvert x_n^\lambda\rvert
\leq \lvert y_n^\lambda\rvert$ and $(x_n^\lambda,y_n^\lambda)\in B$ as defined above. Hence by (\ref{24}), we get
\[
\left| d^{-n} \big(H_{\la}^{+n} \big)^{\ast} U_{\ga, \si^n(\la)} - d^{-n}\log \lvert y_n^\lambda \rvert \right| \leq \frac{2C}{d^n\lvert y _n^\lambda\rvert}({\lvert x_n^\lambda
\rvert}^{1/d}+1)<\epsilon.
\]
On the other hand, by using (\ref{15}), it follows that
\[
\left| G_\lambda^+(x,y)- d^{-n}\log\lvert y_n^\lambda\rvert\right|<\epsilon
\]
for large $n$. Combining these two inequalities and the fact that $ d^{-n} \big(H_{\la}^{+n} \big)^{\ast} U_{\ga, \si^n(\la)}=U_{\gamma,\lambda}$ for all $n\geq 1$, we get
\[
\left| G_\lambda^+(z)-U_{\gamma,\lambda}(z)\right|<2\epsilon.
\]
Hence $U_{\gamma,\lambda}=G_\lambda^+$ in $\mathbb{C}^2\setminus K_\lambda^+$.

\medskip

The next step is to show that $U_{\gamma,\lambda}=0$ in the interior of $K_\lambda^+$. Since $U_{\gamma,\lambda}=G_\lambda^+$ in $\mathbb{C}^2\setminus K_\lambda^+$, the maximum
principle applied to $U_{\gamma,\lambda}(x,.)$ with $x$ being fixed, gives $U_{\gamma,\lambda}\leq 0$ on $K_\lambda^+$. Suppose that there exists a nonempty
$\Omega\subset\subset K_\lambda^+$ satisfying
$U_{\gamma,\lambda}\leq -t$ in $\Omega$ with $t>0$. Let $R>0$ be so large that $\bigcup_{n=0}^{\infty}H_\lambda^{+n}(\Omega)\subset V_R$ -- this follows from Lemma \ref{le1}. Since
$d^{-n} \big(H_{\la}^{+n} \big)^{\ast} U_{\ga, \si^n(\la)}=U_{\gamma,\lambda}$ for each $n \ge 1$, it follows that
\[
H_\lambda^{+n}(\Omega)\subset \big\{U_{\gamma,\sigma^n(\lambda)}\leq -d^n t\big\}\cap V_R
\]
for each $n \ge 1$. The measure of the last set with $x$ fixed and $\lvert x \rvert\leq R$ can be estimated in this way -- let
\[
Y_x=\big\{ y\in \mathbb{C}:U_{\gamma,\sigma^n(\lambda)}\leq -d^n t\big\}\cap \big\{\lvert y \rvert <R\big\}.
\]
By the definition of capacity
\[
\text{cap}(Y_x)\leq \exp (-d^n t)
\]
and since the Lebesgue measure of $Y_x$, say $m(Y_x)$ is at most $\pi e {\text{cap}(Y_x)}^2$ (by the compactness of $Y_x \subset \mbb C$) we get
\[
m(Y_x)\leq \pi \exp(1-2d^n t).
\]
Now for each $\la\in M$, the Jacobian determinant of $H_\la$ is a constant given by $a_\la= a_1(\la) a_2(\la) \cdots a_m(\la)\neq 0$ and since the correspondence $\la \mapsto a_\la$ is
continuous, an application of Fubini's theorem yields
\[
a^n m(\Omega)\leq \lvert a_{\sigma^{n-1}(\la)}\cdots a_\la\rvert m(\Omega)=m(H_\la^{+n}(\Omega))\leq \int_{\lvert x \rvert\leq R}m(Y_x)dv_x \leq \pi^2 R^2 \exp (1-2d^n t)
\]
where $a=\inf_{\la\in M} \lvert a_\la \rvert $. This is evidently a contradiction for large $n$ if $m(\Omega)>0$.

\medskip

So far, it has been shown that $U_{\gamma,\lambda}=G_\lambda^+$ in $\mathbb{C}^2\setminus J_\lambda^+$. By using the continuity of $G_\lambda^+$ and the upper semi-continuity of
$U_{\gamma,\lambda}$, we have that $U_{\gamma,\lambda}\geq G_\lambda^+$ in $\mathbb{C}^2$. Let $\epsilon>0$ and consider the slice $D_\lambda=\{y:G_\lambda^+<\epsilon\}$ in the
$y$-plane for some fixed $x$. Note that $U_{\gamma,\lambda}(x,.)=G_\lambda^+(x,.)=\epsilon$ on the boundary $\partial D$. Hence by the maximum principle, $U_{\gamma,\lambda}(x,.)\leq
\epsilon$ in $D_\lambda$. Since $x$ and $\epsilon$ are arbitrary, it follows that $U_{\gamma,\lambda}=G_\lambda^+$ in $\mathbb{C}^2$. This implies that
\[
\gamma_\lambda=c_{\gamma,\lambda}\mu_\lambda^+
\]
for any $\gamma_\lambda\in S_\lambda(\psi T)$.

\medskip

$(i)$ follows as the step $2$ and $(ii)$ follows from the fact that the mass of ${\big( H_\lambda^{+n}\big)}^*\big( T\wedge dd^c\psi \big)$ is bounded uniformly in $n$.

\medskip

This completes the proof of Theorem \ref{thm1}.
\end{proof}

\subsection{Proof of Proposition \ref{pr3}}

\begin{proof}
Let $\si : M \ra M$ be an arbitrary continuous map and pick a $\ga_{\la} \in S(\psi, T)$. Let $\theta = 1/2 \;dd^c \log (1 + \vert x \vert^2)$ in $\mbb C^2$ (with coordinates $x, y$)
which is a positive closed $(1, 1)$-current depending only on $x$. Then for any test function $\varphi$ on $\mbb C^2$,
\[
\int_{\mathbb{C}^2}\varphi \gamma_\lambda \wedge \theta = c_{\gamma,\lambda}\int_{\mathbb{C}^2}U_{\gamma,\lambda}dd^c \varphi \wedge \theta
= c_{\gamma,\lambda}\int_{\mathbb{C}}\theta \int_{\mathbb{C}}U_{\gamma,\lambda} \Delta_y \varphi
= c_{\gamma,\lambda} \int_{\mathbb{C}}\theta \int_{\mathbb{C}}\varphi \Delta_y U_{\gamma,\lambda}.
\]
Since $y \mapsto U_{\gamma,\lambda}(x,y)$ has logarithmic growth near infinity and $\varphi$ is arbitrary, it follows that
\begin{equation}
\int_{\mathbb{C}^2}\gamma_\lambda \wedge \theta = 2\pi c_{\gamma,\lambda}\int_{\mathbb{C}^2}\theta
={(2\pi)}^2c_{\gamma,\lambda}.\label{25}
\end{equation}
Let $R > 0$ be large enough so that $\text{supp}(\psi T) \cap V_R^+ = \phi$, which implies that $\text{supp}(\psi T)$ is contained in the closure of $V_R \cup V_R^-$. Then
\begin{eqnarray*}
\int_{\mathbb{C}^2}\frac{1}{d^{n_j^\lambda}} (H_\lambda^{{+ n_j^\lambda}})^{\ast}(\psi T) \wedge \theta
&=& \frac{1}{d^{n_j^\lambda}}\int_{\mathbb{C}^2}\psi T \wedge \frac{1}{2} (H_\lambda^{{+ n_j^\lambda}})_{\ast}dd^c\log (1+|x|^2)\\
&=& \frac{1}{d^{n_j^\lambda}}\int_{\mathbb{C}^2} (\psi T)\wedge dd^c\left(\frac{1}{2}\log (1+|\pi_1\circ (H_\lambda^{{+ n_j^\lambda}})^{-1}|^2)\right)\\
&=& \frac{1}{d^{n_j^\lambda}} \int_{\overline{V_R\cup V_R^-}}\psi T \wedge dd^c\left(\frac{1}{2}\log (1+|\pi_1\circ (H_\lambda^{+ n_j^\lambda})^{-1}|^2)\right).
\end{eqnarray*}
It is therefore sufficient to study the behavior of $\log (1+|\pi_1\circ (H_\lambda^{+ n_j^\lambda})^{-1}|^2)$. But
\[
\log^+ \vert x \vert \le \log^+ \vert (x, y) \vert \le \log^+ \vert x \vert + R
\]
for $(x, y) \in V_R \cup V_R^-$ and by combining this with
\[
2 \log^+ \vert x \vert \le \log (1 + \vert x \vert^2) \le 2 \log^+ \vert x \vert + \log 2,
\]
it follows that the behavior of $(1/2) d^{-n_j^{\la}} \log (1+|\pi_1\circ (H_\lambda^{+ n_j^\lambda})^{-1}|^2)$ as $j \ra \infty$ is similar to that of
$d^{-n_j^{\la}} \log^+ \vert (H_\lambda^{+ n_j^\lambda})^{-1} \vert$.

\medskip

Now suppose that $\si$ is the identity on $M$. In this case, $(H_\lambda^{+ n_j^\lambda})^{-1}$ is just the usual  $n_j^{\la}$--fold iterate of the map $H_{\la}$ and by Proposition \ref{pr1},
it follows that
\[
\lim_{j \ra \infty} d^{-n_j^{\la}} \log \Vert (H_\lambda^{+n_j^\lambda})^{-1} \Vert = G_{\la}^-
\]
and hence
\[
4 \pi^2 c_{\ga, \la} = \int_{\mbb C^2} \ga_{\la} \wedge \theta = \int_{\mathbb{C}^2} \lim_{j \ra \infty} \frac{1}{d^{n_j^\lambda}} (H_\lambda^{{+ n_j^\lambda}})^{\ast}(\psi T) \wedge
\theta = \int_{\mbb C^2} \psi T \wedge \mu_{\la}^-.
\]
The right hand side of the above equation is independent of the subsequence used in the construction of $\ga_{\la}$ and hence $S_\lambda(\psi, T)$ contains a unique element.

\medskip

The other case to consider is when there exists a $\la_0 \in M$ such that $\si^n(\la) \ra \la_0$ for all $\la$. For each $n \ge 1$, let
\[
\ti G_{n, \la}^- = \frac{1}{d^n} \log^+ \Vert (H_{\la}^{+n})^{-1} \Vert.
\]
Note that $\ti G_{n, \la}^-  \not= G_{n, \la}^-!$ It will suffice to show that $\ti G_{n, \la}^-$ converges uniformly on compact subsets of $\mbb C^2$ to a plurisubharmonic function,
say
$\ti G_{\la}^-$. Let
\[
\ti K_{\la}^- = \big\{ z \in \mbb C^2 : \;\text{the sequence} \;\{ (H_{\la}^{+n})^{-1}(z) \}  \;\text{is bounded} \;\big\}
\]
and let $A \subset \mbb C^2$ be a relatively compact set such that $A \cap \ti K_{\la}^- = \phi$ for all $\la \in M$. The arguments used in Lemma \ref{le1} show that
\[
\mbb C^2 \setminus \ti K_{\la}^- = \bigcup_{n=0}^{\infty} H_{\la}^{+n}(V_R^-)
\]
for a sufficiently large $R > 0$. As Proposition \ref{pr1}, it can be shown that $\ti G_{n, \la}^-$ converges to a pluriharmonic function $\ti G_{\la}^-$ on $V_R^-$. Hence for large $m, n$,
\begin{equation} \label{con}
\vert \ti G_{m, \la}^-(p) - \ti G_{n, \la}^-(q) \vert < \ep 
\end{equation}
for $p, q \in V_R^-$ that are close enough. Let $n_0$ be such that $(H_{\la_0}^{+n_0})^{-1}(A) \subset V_R^-$ and pick a relatively compact set $S \subset V_R^-$
such that $(H_{\la_0}^{+n_0})^{-1}(A) \subset S$. Pick any $\la$. Since $\si^n(\la) \ra \la_0$ and the maps $H_{\la}^{\pm 1}$ depend continuously on $\la$, it follows that
$H_{\si^n(\la)}^{+n_0}(A) \subset S$. By choosing $m, n$ large enough, it is possible to ensure that for all $(x, y) \in A$, $(H_{\si^{m - n_0}(\la)}^{+n_0})^{-1}(x, y)$ and
$(H_{\si^{n - n_0}(\la)}^{+n_0})^{-1}(x, y)$ are as close to each other as desired. By writing
\[
\ti G_{n, \la}^-(x, y) = \frac{1}{d^{n_0}} \frac{1}{d^{n - n_0}} \log^+ \Vert H_{\la}^{-1} \circ \cdots \circ H_{\si^{n - n_0 + 1}(\la)}^{-1} \circ
(H_{\si^{n - n_0}(\la)}^{+n_0})^{-1}(x, y) \Vert
\]
and using (\ref{con}) it follows that $\ti G_{n, \la}^-$ converges uniformly to a pluriharmonic function on $A$. To conclude that this convergence is actually uniform on compact sets of
$\mbb C^2$,
it suffices to appeal to the arguments used in Proposition \ref{pr1}.
\end{proof}

\thispagestyle{empty}
\chapter{Random iterations of H\'{e}non maps}
In this chapter, we study dynamics of random iterations of H\'{e}non maps. Several results obtained here generalize the results in Chapter $3$. 
\section{Random Green functions and their averages}
In this section, we prove Theorem \ref{R thm1} and Proposition \ref{R pr1}.

\medskip
\no
From Chapter $3$ we have learned that there exists a uniform filtration for the family of H\'{e}non maps ${\{H_\lambda\}}_{\lambda\in M}$ given by:
\begin{align*}
V_R^+ &= \big\{ (x, y) \in \mbb C^2 : \vert y \vert > \vert x \vert, \vert y \vert > R \big\},\\
V_R^- &= \big\{ (x, y) \in \mbb C^2 : \vert y \vert < \vert x \vert, \vert x \vert > R \big\}\text{ and }\\
V_R   &= \big\{ (x, y) \in \mbb C^2 : \vert x \vert, \vert y \vert \le R\}
\end{align*}
where $R>0$. 
Further, for a sufficiently large $R>0$ and for all $\lambda\in M$, we have:
$$
H_\lambda(V_R^+)\subset V_R^+, \ \ H_\lambda(V_R^+\cup V_R)\subset V_R^+\cup V_R
$$
and
$$
H_\lambda^{-1}(V_R^-)\subset V_R^-, \ \ H_\lambda^{-1}(V_R^-\cup V_R)\subset V_R^-\cup V_R.
$$
As in Lemma \ref{le1}, it can be shown that for sufficiently large $R>0$,  
$$
I_\Lambda^{\pm}=\mathbb{C}^2\setminus K_\Lambda^{\pm}=\bigcup_{n=0}^\infty {(H_{n,\Lambda}^\pm)}^{-1}(V_R^{\pm})
$$
for each $\Lambda\in X$.

\subsection{Proof of Theorem \ref{R thm1}}
\begin{proof}
Since $I^+$ is an attracting fixed point for each $H_\lambda$ and the correspondence $\lambda\mapsto H_\lambda$ is continuous, one can choose a suitable neighborhood of $I^+$ in $\mathbb{P}^2$, say $U$, such that $H_\lambda(Y)\subset \subset Y$ for all $\lambda\in M$ where $Y=\mathbb{P}^2\setminus \overline{U}$. Now start with a compact set $K$ in $\mathbb{C}^2$. One can choose $R$ sufficiently large so that $K\subset V_R\cup V_R^+$ and $H_\lambda(V_R\cup V_R^+)\subset V_R \cup V_R^+$ for all $\lambda\in M$. Now define 
$$
v_\lambda(z)=\frac{1}{d}\log^+ \lVert H_\lambda(z)\rVert-\log^+\lVert z \rVert
$$
for $z\in V_R\cup V_R^+$. Proof of Proposition $\ref{pr1}$ provides a constant $K>0$ such that $\lvert v_\lambda(z) \rvert < K$ for all $z\in V_R^+ \cup V_R $ and for all $\lambda\in M$. Consider $(Y,d_{FS})$ where $d_{FS}$ is the Fubini-Study metric on $Y$. Thus $Y$ is a metric space of finite diameter. Note that  each $v_\lambda$ and $H_\lambda$ are Lipschitz on $Y$ with  uniform Lipschitz coefficients $A$ and $L$ respectively. In particular, one can choose $L$ to be $\sup_{\lambda\in M} L_\lambda$ where each $L_\lambda=\sup_{\mathbb{P}^2\setminus U}\lVert DH_\lambda \rVert$.   \\
\indent
Now 
\begin{equation}
G_{n+1,\Lambda}^+(z)-G_{n,\Lambda}^+(z) = \frac{1}{d^{n+1}} \log^+ \big \lVert H_{n+1,\Lambda}^+(z)\big \rVert-\frac{1}{d^{n}} \log^+ \big\lVert H_{n,\Lambda}^+(z)\big \rVert 
= \frac{1}{d^n} \big ( v_{\lambda_{n+1}}\circ H_{n,\Lambda}^+(z)\big ).
\end{equation}
Thus 
\begin{equation}
\big \lvert G_{n+1,\Lambda}^+(z)-G_{n,\Lambda}^+(z)\big \rvert=d^{-n}\big\lvert  v_{\lambda_{n+1}}\circ H_{n,\Lambda}^+(z) \big\rvert \leq Kd^{-n} \label{R2}
\end{equation}
on $V_R\cup V_R^+$ for all $n\geq 1$ and for all $\Lambda\in X$. Therefore $G_{n,\Lambda}^+$ converges uniformly to the continuous function $G_\Lambda^+$ on any compact subset of $\mathbb{C}^2$. Also note that 
\[
d G_{n+1,\Lambda'}^+=G_{n,\Lambda}^+\circ H_\lambda
\]
for all $n\geq 1$, which yields $d G_{\Lambda'}^+=G_\Lambda^+ \circ H_\lambda$. That $G_\Lambda^{+}$ is strictly positive pluriharmonic function on $\mathbb{C}^2\setminus K_\Lambda^{+}$ and vanishes precisely on $K_\Lambda^{+}$ follow by using  similar kind of arguments as Proposition $\ref{pr1}$.\\
\indent
Note that 
\[
G_\Lambda^+(z)=\log^+\lVert z \rVert+\sum_{n\geq 0} d^{-n}\big(v_{\lambda_{n+1}} \circ H_{n,\Lambda}^+(z)\big) 
\]
for all $\Lambda\in X$. Since $(Y,d_{FS})$ has finite diameter, it is sufficient to work with $a,b \in V_R\cup V_R^+$ with $d_{FS}(a,b)<<1$. Now
\begin{eqnarray}
&&\big\lvert \sum_{n\geq 0}d^{-n}v_{\lambda_{n+1}}\circ H_{n,\Lambda}^+(a)- \sum_{n\geq 0}d^{-n}v_{\lambda_{n+1}}\circ H_{n,\Lambda}^+(b) \big\rvert \nonumber \\
&\leq& \sum_{0\leq n \leq N-1} d^{-n}\big\lvert v_{\lambda_{n+1}}\circ H_{n,\Lambda}^+(a)- v_{\lambda_{n+1}}\circ H_{n,\Lambda}^+(b)\big\rvert 
+  \sum_{n \geq N} d^{-n}\big\lvert v_{\lambda_{n+1}}\circ H_{n,\Lambda}^+(a)- v_{\lambda_{n+1}}\circ H_{n,\Lambda}^+(b)\big\rvert \nonumber\\
&\leq&  A \sum_{0\leq n \leq N-1} d^{-n} \big\lvert H_{n,\Lambda}^+(a)-H_{n,\Lambda}^+(b)\big\rvert + 2K \sum_{n\geq N} d^{-n} 
\leq A d_{FS}(a,b) \sum_{0\leq n \leq N-1} d^{-n}L^n + K'd^{-N}\nonumber \\
\label{R3}
\end{eqnarray} 
where $K'={2Kd}/{(d-1)}$. If $L<d$, the last sum in (\ref{R3}) is of order at most equal to $ d_{FS}(a,b)+ d^{-N}$. So for a given $0<\beta<1$, if we choose $N> {-\beta \log d_{FS}(a,b)}/{\log d}$, then the sum in (\ref{R3}) is $\lesssim {d_{FS}(a,b)}^\beta$. Thus in this case, $G_\Lambda^+$ is locally H\"{o}lder continuous in $\mathbb{C}^2$ for any $0<\beta <1$. For $L\geq d$, the sum in (\ref{R3}) is of order $d_{FS}(a,b){(Ld^{-1})}^N+d^{-N}$. Now if we choose 
${-\log d_{FS}(a,b)}/{\log L}\leq N <({-\log d_{FS}(a,b)}/{\log L})+1$, the sum in (\ref{R3}) is $ \lesssim {d_{FS}(a,b)}^{{\log d}/{\log L}}$ since $d^{-N}< {d_{FS}(a,b)}^{{\log d}/{\log L}}$. Therefore $G_\Lambda^+$ is locally H\"{o}lder continuous in $\mathbb{C}^2$ for any $0<\beta <{\log d}/{\log L}$. So for each $\Lambda\in X$, $G_\Lambda^+$ is $\beta$-H\"{o}lder continuous for any $\beta$ such that $0<\beta< \min\{1,{\log d}/{\log L}\}$. A similar argument can be implemented to show that $G_\Lambda^-$ is H\"{o}lder continuous. 

\medskip

Now to prove that the correspondence $\Lambda \mapsto G_\Lambda^{+}$ is continuous, take a compact set $S\subset \mathbb{C}^2$ and $\Lambda_0\in M$. Then
\begin{equation*}
\vert G_{\Lambda}^+(z)-G_{\Lambda_0}^+(z) \vert \le \vert G_{n,\Lambda}^+(z)-G_{\Lambda}^+(z)\vert + \vert G_{n,\Lambda}^+(z)-G_{n,\Lambda_0}^+(z)\vert  + \vert G_{n,\Lambda_0}^+(z)-G_{\Lambda_0}^+(z)\vert
\end{equation*}
for all $z \in S$. For a given $\epsilon>0$, using (\ref{R2}) we can choose $n$ large enough so that the first and third terms in the right hand side of the above equation are less than $\epsilon/3$. The middle term can also be managed to be less than $\epsilon/3$ by taking $\Lambda$ very close to $\Lambda_0$. Hence we prove that the correspondence $\Lambda\mapsto G_\Lambda^+$ is continuous. The corresponding results related to $G_\Lambda^-$ can be proved by applying a similar set of arguments. 
\end{proof}

\subsection{Proof of Proposition \ref{R pr1}} 
\begin{proof}
Pick any $z\in \mathbb{C}^2$ and let $\Delta$ be a small linear disk centered at $z$. Then we have
\begin{multline*}
EG^+(z)=\int_X G_\Lambda^+ (z)
\leq \int_X \frac{1}{\lvert \partial \Delta \rvert} \int_{\partial \Delta} G_\Lambda^+ (\xi) 
= \frac {1}{\lvert\partial \Delta \rvert} \int_{\partial \Delta} \int_X G_\Lambda^+(\xi)
= \frac{1}{\lvert\partial \Delta \rvert} \int_{\partial \Delta} EG^+(\xi).
\end{multline*}
The first inequality follows due to the fact that $G_\Lambda^+$ is subharmonic in $\mathbb{C}^2$ for each $\Lambda\in X$ and the next equality is obtained by applying Fubini's theorem to the continuous function $G^+: X \times \mathbb{C}^2 \ra \mathbb{R}$ defined as follows: $(\Lambda, z) \mapsto G_\Lambda^+(z)$. This shows that $EG^+$ is plurisubharmonic. That for a fixed compact set  in $\mathbb{C}^2$, a fixed H\"{o}lder coefficient and a fixed H\"{o}lder exponent work for each $G_\Lambda^+$ implies the local H\"{o}lder continuity and in particular, the continuity of the average Green function $EG^+$ in $\mathbb{C}^2$. Similar statement is valid for $EG^{-}$.  

\medskip

For a $z\in \bigcap_{\Lambda\in X} K_\Lambda^+$, clearly $EG^+(z)=0$. To see the converse, let 
\[
EG^+(z)=\int_{\Lambda\in X} G_\Lambda^+(z)=0
\] 
for a $z\in \mathbb{C}^2$. Since the correspondence $\Lambda\mapsto G_\Lambda^+(z)$ is continuous for a fixed $z\in \mathbb{C}^2$ and $G_\Lambda^+$ is positive function on $\mathbb{C}^2$ for all $\Lambda\in X$, we have $G_\Lambda^+(z)=0$ for all $\Lambda\in X$. Hence $z\in \bigcap_{\Lambda} K_\Lambda^+$. Further, it is easy to observe that $EG^{\pm}$ are pluriharmonic outside $\overline{\bigcup_{\Lambda}K_\Lambda^\pm}$.
\end{proof}

\section{Random Green currents and their averages}
In this section, we prove Proposition \ref{R pr2}.
\subsection{Proof of Proposition \ref{R pr2}} 
\begin{proof}
For a test form $\varphi$ in $\mathbb{C}^2$, we note the followings:
\begin{equation*}
\langle \mu^\pm,\varphi\rangle=\langle dd^c{EG}^{\pm},  \varphi\rangle
= \int_X\langle G_\Lambda^{\pm}, dd^c \varphi\rangle 
=\int_X \langle dd^c G_\Lambda^{\pm},\varphi\rangle
= \int_X \langle \mu_\Lambda^{\pm},\varphi\rangle
= \langle \int_X \mu_\Lambda^{\pm},\varphi \rangle.
\end{equation*}

Thus we get $\mu^{\pm}=\int_X \mu_\Lambda^{\pm}$.

\medskip

By Theorem \ref{R thm1}, it follows that  $G_\Lambda^{\pm}\circ H_\lambda^{\pm 1}=d G_{\lambda\Lambda}^{\pm}$
which in turn gives $ {(H_\lambda^{\pm 1})}^*(dd^c G_\Lambda^\pm)= d (dd^c G_{\lambda\Lambda}^\pm)$. Hence ${(H_\lambda^{\pm 1})}^* \mu_\Lambda^{\pm}=d \mu_{\lambda\Lambda}^\pm$ for all $\lambda\in M$ and for all $\Lambda\in X$.

\medskip

That for a $z\in \mathbb{C}^2$, $\int_{M} {EG}^{\pm}\circ H_\lambda^{\pm}(z)d\nu(\lambda)= d {(EG)}^{\pm}(z)$ follows from the fact that $G_\Lambda^{\pm}\circ H_\lambda^{\pm 1}=d G_{\lambda \Lambda}^{\pm}$ for all $\lambda\in M$ and for all $\Lambda\in X$ .

\medskip

To prove the next assertion observe the followings:
\begin{multline*}
\langle \int_{M} {(H_\lambda^{\pm 1})}^* \mu^{\pm},\varphi \rangle = \int_{M} \langle dd^c ({EG}^{\pm}\circ H_\lambda^{\pm 1 }),\varphi \rangle
= \int_{M}\langle {EG}^{\pm}\circ H_\lambda^{\pm 1 }, dd^c \varphi \rangle\\
= \langle \int_{M} {EG}^{\pm}\circ  H_\lambda^{\pm 1 },dd^c \varphi\rangle
= \langle d {(EG)}^{\pm},dd^c \varphi\rangle
= d \langle dd^c ({EG}^{\pm}),\varphi \rangle
= d \langle \mu^{\pm},\varphi\rangle
\end{multline*}
Therefore $\int_{M} {(H_\lambda^{\pm 1})}^* \mu^{\pm}=d \mu^{\pm}$ holds.

\medskip

That the support of $\mu_\Lambda^\pm$ is $J_\Lambda^\pm$ and the correspondence $\Lambda\mapsto J_\Lambda^\pm$ is lower semi-continuous can be shown by implementing the same techniques as in Proposition $\ref{pr2}$.

\medskip

Since $\mu^\pm=\int_{X} \mu_\Lambda^\pm$, we have ${ \rm supp} (\mu^{\pm})\subset \overline{\bigcup_{\Lambda\in X} J_\Lambda^\pm }$. Now to prove the other inclusion, first note that $\mu_\Lambda^{\pm}$ vary continuously in $\Lambda$. Therefore if some $\mu_{\Lambda_0}^{\pm}$ has mass in some open set $\Omega\subset\mathbb{C}^2$, then there exists a neighborhood $U_{\Lambda_0}\subset X$ of $\Lambda_0$ such that $\mu_\Lambda^\pm$ has nonzero mass in $\Omega$ for all $\Lambda\in U_{\Lambda_0}$ and thus $\mu^{\pm}$ also has nonzero mass in $\Omega$. This completes the proof. 
\end{proof}

\section{Convergence theorems}
In this section, we prove Theorem \ref{R thm2} and Theorem $\ref{R thm3}$.

\subsection{Proof of Theorem \ref{R thm2}}
\begin{proof}
Since the correspondence $\lambda\mapsto H_\lambda$ is continuous, we can assume $U$ to be a suitable domain satisfying $H_\lambda(U)\subset \subset U$ for all $\lambda\in M$. Let $\varphi$ be a $(1,1)$-form in $\mathbb{P}^2$. Since $\varphi$ is of class $C^2$, $dd^c \varphi$ is a continuous form of maximal degree. Thus we can assume that the signed measure given by $dd^c \varphi$ has no mass on set of volume zero and in particular, on the hyperplane at infinity. Hence multiplying $\varphi$ with a suitable constant, we can assume 
${\lVert dd^c \varphi \rVert}_\infty \leq 1$. Thus $\gamma=dd^c \varphi$ is a complex measure in $\mathbb{C}^2$ with mass less than or equal to $1$.   
 Define $\gamma_{n}={(H_{\lambda_n}\circ \cdots\circ H_{\lambda_1})}_* (\gamma)$ for all $n\geq 1$. Note that $\gamma_n$ has the same mass as $\gamma$ for all $n\geq 1$ since $H_\lambda$ is automorphism on $\mathbb{C}^2$ for all $\lambda\in M$. Let ${\gamma_n}'$ and ${\gamma_n}''$ are the restriction of $\gamma_n$ to $\mathbb{P}^2 \setminus U$ and $U$ respectively. Clearly $\lVert {\gamma_n}' \rVert \leq 1$ on $\mathbb{P}^2\setminus U$ and $ \lVert{\gamma_n}''\rVert \leq 1$ on $U$. Observe that due to our choice of $U$, $H_\lambda^{-1}$ defines a map from $\mathbb{P}^2\setminus U$ to $\mathbb{P}^2\setminus U$ with $C^1$ norm bounded by some $L>0$ for all $\lambda\in M$. So the $C^1$ norm of 
 $(H_{\lambda_1}^{-1}\circ\cdots\circ H_{\lambda_n}^{-1})$ is bounded by $L^n$ on $\mathbb{P}^2\setminus U$ for all $n\geq 1$ and consequently we have  ${\lVert\gamma_n' \rVert}_{\infty}\leq L^{4n}$ for all $n\geq 1$.\\
\indent
For each $n\geq 1$, we define:
\begin{equation*}
 q_{\Lambda_n}^+ := \frac{1}{2\pi}G_{\Lambda_n}^+ - \frac{1}{2\pi} \log {(1+ {\lVert z \rVert}^2)}^{\frac{1}{2}} .
\end{equation*}
 Note that $ q_{\Lambda_n}^+$ is a quasi potential of $ \mu_{\Lambda_n}^+$ for all $n\geq 1$. Define $v_{\Lambda_n} := u_n -  q_{\Lambda_{n+1}}^+$.
 Using a similar argument as (\ref{15}), one can show that $q_{\Lambda_n}^+$ is bounded on $U$ by a fixed constant for all $n\geq 1$. This implies that there exists $A_1>0$ such that $\lvert v_{\Lambda_n}\rvert \leq A_1$ on $U$ for all $n\geq 1$. Now observe the followings:
 \begin{eqnarray}
 && \big\langle  d^{-n} {(H_{\lambda_n}\circ \cdots\circ H_{\lambda_1} )}^* ({S_n})-\mu_\Lambda^+, \varphi \big\rangle \nonumber\\
  &=& \big\langle d^{-n} {(H_{\lambda_n}\circ \cdots\circ H_{\lambda_1} )}^* ({S_n})- d^{-n} {(H_{\lambda_n}\circ \cdots\circ H_{\lambda_1} )}^* (\mu_{\Lambda_{n+1}}^+),\varphi \big\rangle \nonumber \\
 &=& d^{-n} \big\langle {(H_{\lambda_n}\circ \cdots\circ H_{\lambda_1} )}^* (dd^c (v_{\Lambda_n})), \varphi \big\rangle \nonumber\\
 &=&d^{-n} \big\langle v_{\Lambda_n}, {(H_{\lambda_n}\circ \cdots\circ H_{\lambda_1} )}_*(dd^c \varphi) \big\rangle \nonumber\\
 &=& d^{-n} \big\langle v_{\Lambda_n}, \gamma_n  \big\rangle 
 = d^{-n} \big\langle \gamma_n',  v_{\Lambda_n} \big\rangle + d^{-n} \big\langle \gamma_n'', v_{\Lambda_n} \big\rangle  .\label{R4}
\end{eqnarray}
\indent
Consider the family $\{v_{\Lambda_n}\}_{n\geq 1}$. Note that $dd^c (v_{\Lambda_n})= {S_n}- \mu_{\Lambda_{n+1}}^+$. This implies that 
${\lVert dd^c (v_{\Lambda_n})\rVert}_*$ is uniformly bounded by a fixed constant for all $n\geq 1$ since ${S_n}$ and $\mu_{\Lambda_n}^+$ both have mass $1$ for all $n\geq 1$. Hence by Lemma \ref{pre 4}, it follows that $\{v_{\Lambda_n}\}_{n\geq 1}$ is a bounded subset in DSH$(\mathbb{P}^2)$. Since $ \lVert\gamma_n'\rVert \leq 1$ and ${\lVert\gamma_n'\rVert}_\infty \leq L^{4n}$, Corollary \ref{pre 6} gives the following:
\begin{equation}
\big \lvert d^{-n} \big\langle \gamma_n', v_{\Lambda_n}\big \rangle \big \rvert \leq d^{-n} c (1+\log^+  {L^{4n}}) \lesssim nd^{-n}.
\end{equation} 
Now the second term in (\ref{R4}) is of order $O(d^{-n})$ since $\lVert \gamma_n'' \rVert \leq 1$ and $\lvert v_{\Lambda_n}\rvert \leq A_1$ on $U$. This estimate along with (\ref{R4}) gives 
\begin{equation}
\big\lvert \big\langle  d^{-n} {(H_{\lambda_n}\circ \cdots\circ H_{\lambda_1} )}^* ({S_n})-\mu_\Lambda^+, \phi \big\rangle \big\rvert 
\leq And^{-n} {\lVert \phi \rVert}_{C^1}.\label{uni}
\end{equation}
\end{proof}

\begin{rem}
It is clear from (\ref{uni}) and Lemma \ref{pre 4} that the constant $A$ does not depend on the $\Lambda$ we are starting with. In particular it shows that for given currents ${\{S_n\}}_{n\geq 1}$ as prescribed before $d^{-n} {(H_{\lambda_n}\circ \cdots\circ H_{\lambda_1} )}^*(S_n)$ converges uniformly to $\mu_\Lambda^+$  for all $\Lambda\in X$  in the weak sense of currents.
\end{rem}

\subsection{Proof of Theorem \ref{R thm3}}
\begin{proof}
First note that 
\begin{equation*}
\Theta^n(S_n)=\int_{M^n} {\frac{{(H_{n,\Lambda}^+)}^*(S_n)}{d^n}}.
\end{equation*}
Now for a test form $\varphi$ on $\mathbb{P}^2$, we have
\begin{eqnarray}
\big\langle \Theta^n(S_n),\varphi \big\rangle &=& \int_{M^n} \Big\langle \frac{{(H_{n,\Lambda}^+)}^*(S_n)}{d^n},\varphi \Big\rangle \nonumber 
= \int_{X} \Big\langle \frac{{(H_{n,\Lambda})}^*(S_n)}{d^n},\varphi \Big\rangle.
\end{eqnarray}
For each $n\geq 1$, we define $F_n: X \ra \mathbb{C}$ as follows:
$$
F_n(\Lambda):=  \Big\langle \frac{{(H_{n,\Lambda})}^*(S_n)}{d^n},\varphi \Big\rangle.
$$ 
By Theorem \ref{R thm2}, 
$
F_n(\Lambda) \ra F(\Lambda)=\langle \mu_\Lambda^+,\varphi\rangle
$
for all $\Lambda\in X$. Note that each $F_n$ and $F$ are integrable. Also we have $\lvert F_n(\Lambda)\rvert \lesssim {\lVert \varphi \rVert}_\infty$ for all $n\geq 1$ and for all $\Lambda\in X$. Now we use dominated convergence theorem to get the followings:
\begin{eqnarray*}
\big\langle\Theta^n(S_n),\varphi \big\rangle =\int_X F_n &\ra& \int_X F 
= \int_X \big\langle \mu_\Lambda^+,\varphi \big\rangle
= \langle \mu^+,\varphi \rangle.
\end{eqnarray*}
This completes the proof.
\end{proof}

\section{Rigidity of random Julia sets}
In this section, we prove Theorem \ref{R thm4}.
\subsection{Proof of Theorem \ref{R thm4}}
\begin{proof}
Fix $\Lambda\in X$ and let $S$ be a positive closed $(1,1)$-current of mass $1$ with support in $\overline{K_\Lambda^+}$. We show that $S=\mu_\Lambda^+$.  For each $n\geq 1$, define 
$$
S_{n,\Lambda}=d^n (H_{\lambda_n}\circ \cdots\circ H_{\lambda_1})_* S
$$
on $\mathbb{C}^2$. Note that each $S_{n,\Lambda}$ is a positive closed $(1,1)$-current on $\mathbb{C}^2$ with support in $K_{\Lambda_n}^+$. Therefore it can be extended through the hyperplane by $0$ as a positive closed $(1,1)$-current on $\mathbb{P}^2$. Now since 
$$
{(H_{\lambda_n}\circ \cdots\circ H_{\lambda_1})}^*{(H_{\lambda_n}\circ \cdots\circ H_{\lambda_1})}_* (S_{\Lambda,n})=S
$$  
on $\mathbb{C}^2$, we have
\begin{equation*}
d^{-n}{(H_{\lambda_n}\circ \cdots\circ H_{\lambda_1})}^* (S_{n,\Lambda})=S
\end{equation*}
on $\mathbb{C}^2$ and thus
\begin{equation}
d^{-n}{(H_{\lambda_n}\circ \cdots\circ H_{\lambda_1})}^* (S_{n,\Lambda})=S \label{R5}
\end{equation}
on $\mathbb{P}^2$. As a current on $\mathbb{P}^2$ each $S_{n,\Lambda}$ vanishes in a neighborhood of $I^-$  and is of mass $1$. Since for each $n\geq 1$, as a current on $\mathbb{C}^2$, $S_{\Lambda,n}$ has support in $K_{\Lambda_n}^+$, we have ${\rm{supp}}(S_{n,\Lambda})\cap \overline{V_R^+}=\phi$ for sufficiently large $R>0$. By Proposition $8.3.6$ in \cite{MNTU}, for each $n\geq 1$ there exists $u_{n,\Lambda}$ such that 
$$
S_{n,\Lambda}=c_{n,\Lambda} dd^c (u_{n,\Lambda})
$$ 
where $u_{n,\Lambda}(x,y)-\log\lvert y \rvert$ is a bounded pluriharmonic function  on $\overline{V}_R^+$ and $c_{n,\Lambda}>0$. Hence ${(S_{n,\Lambda})}_{n\geq 1}$ satisfies the required hypothesis of Theorem \ref{R thm2} and thus we get  
$$
d^{-n}{(H_{\lambda_n}\circ \cdots\circ H_{\lambda_1})}^* (S_{n,\Lambda})\ra \mu_\Lambda^+
$$
in the sense of current as $n\ra \infty$. Therefore by (\ref{R5}), we have $S=\mu_\Lambda^+$.
\end{proof}

\begin{cor}
For each $\Lambda\in X$, the random Julia set $J_\Lambda^+=\partial K_\Lambda^+$ is rigid.
\end{cor}

\thispagestyle{empty}
\chapter{ Some remarks on global dynamics}
In this chapter, we discuss some important dynamical features of the skew map $H$. As mentioned before, in this chapter $\sigma$ is considered to be a homeomorphism on $M$.
\section{Existence of global Green functions}
In this section, we prove Proposition \ref{G pr1}. 
\subsection{Proof of Proposition \ref{G pr1}}

\begin{proof}
Since $M$ is compact subset of $\mathbb{C}^2$, the existence of fibered Green functions $\mathcal{G}_\lambda^\pm$ implies that the global Green functions 
\[
\mathcal{G}^{\pm}= \lim_{n\ra \infty} \mathcal{G}^{\pm}_n
\]
exist in $ M\times \mathbb{C}^2$. Further, note that $\mathcal{G}^\pm(\lambda,z)=\mathcal{G}_\lambda^\pm(z)$ for all $\lambda\in M$ and for all $z\in \mathbb{C}^2$.\\
\indent
Now the sequences $\mathcal{G}_{n,\lambda}^{\pm}$ converge uniformly to the continuous functions $\mathcal{G}_\lambda^{\pm}$ as $n\ra \infty$ on compact subset of $\mathbb{C}^2$ and this convergence is uniform in $\lambda$. Thus it follows that $\mathcal{G}_n^{\pm}$ converge uniformly to $\mathcal{G}^{\pm}$ as $n\ra \infty$ on compact subset of $M\times \mathbb{C}^2$. Further, (\ref{G1}) gives the following: 
\[
\mathcal{G}^{\pm}\circ H=d^{\pm 1}\mathcal{G}^\pm 
\]
on $M\times \mathbb{C}^2$. Being the uniform limit of  sequence of plurisubharmonic functions, $\mathcal{G}^\pm$ is also plurisubharmonic on $M\times \mathbb{C}^2$. Further, $\mathcal{G}_\lambda^\pm$ vanishes precisely on $\mathcal{K}_\lambda^{\pm}$ and positive on $\mathcal{I}_\lambda^\pm$. Thus by (\ref{Esc}), it follows that $\mathcal{G}^\pm$ vanishes precisely on $\mathcal{K}^\pm$ and positive on $\mathcal{I}^\pm$.
\end{proof}

\section{Mixing}
In this section, we prove Theorem \ref{G thm2} and Proposition \ref{G pr2}.

\medskip

The following lemma is crucial to prove the main result in this section. 
\subsection{A convergence lemma}
\begin{lem}\label{G le1}
For $\lambda\in M$ and for a function $\psi \in C_0(M\times \mathbb{C}^2)$, we have
\begin{equation*}
\frac{1}{d^n}{(H_\lambda^n)}^*(\psi_{\sigma^n(\lambda)} \vartheta_{\sigma^n(\lambda)}^+ )-\langle
\vartheta_{\sigma^n(\lambda)},\psi_{\sigma^n(\lambda)}\rangle \vartheta_\lambda^+ \ra 0
\end{equation*}
as $n\ra \infty$ on $\mathbb{C}^2$ where $\psi_\lambda$ is the restriction of $\psi$ on the fiber $\{\lambda\}\times \mathbb{C}^2$.
\end{lem}
\begin{proof}
Without loss of generality, we consider $\psi \geq 0$. Now observe that the difference between 
\begin{eqnarray}\label{diff}
&&\Big \lVert \frac{1}{d^n} {(H_\lambda^n)}^*(\psi_{\sigma^n(\lambda)}\vartheta_{\sigma^n(\lambda)}^+)\Big \rVert-\langle \vartheta_{\sigma^n(\lambda)},\psi_{\sigma^n(\lambda)}\rangle \nonumber\\ 
&=& \frac{1}{2\pi{d^n}} \int_{\mathbb{C}^2} {(H_\lambda^n)}^* (\psi_{\sigma^n(\lambda)} \vartheta_{\sigma^n(\lambda)}^+ )\wedge dd^c \log {(1+{\lVert z\rVert}^2)}^{\frac{1}{2}}-\langle \vartheta_{\sigma^n(\lambda)},\psi_{\sigma^n(\lambda)}\rangle\nonumber\\
&=& \frac{1}{2\pi}\int_{\mathbb{C}^2} (\psi_{\sigma^n(\lambda)} \vartheta_{\sigma^n(\lambda)}^+ ) \wedge (dd^c \mathcal{G}_{n,\sigma^n(\lambda)}^-)-\langle \vartheta_{\sigma^n(\lambda)},\psi_{\sigma^n(\lambda)}\rangle.
\end{eqnarray}
The last equality holds since  $\log {(1+{\lVert z\rVert}^2)}^{\frac{1}{2}}\sim \log^+\lVert z \rVert$ in $\mathbb{C}^2$. A similar result as  Proposition \ref{pr1} shows that $\mathcal{G}_{n,\lambda}^-$ converges uniformly to $\mathcal{G}_\lambda^-$ as $n\ra \infty$ on compact subset of $\mathbb{C}^2$ and the convergence is uniform in $\lambda$. Thus the quantity in (\ref{diff}) which is equal to
\begin{eqnarray}\label{C1}
\frac{1}{2\pi}\int_{\mathbb{C}^2} (\psi_{\sigma^n(\lambda)} \vartheta_{\sigma^n(\lambda)}^+ ) \wedge (dd^c \mathcal{G}_{n,\sigma^n(\lambda)}^-) -\frac{1}{2\pi}\int_{\mathbb{C}^2} (\psi_{\sigma^n(\lambda)} \vartheta_{\sigma^n(\lambda)}^+ ) \wedge (dd^c \mathcal{G}_{\sigma^n(\lambda)}^-)
\end{eqnarray}
 tends to zero as $n\ra \infty$. \\ 
\indent
Consider the sequence of currents $\{T_n\}$ where 
\begin{equation}\label{proof}
T_n= \frac{1}{d^n}{(H_\lambda^n)}^*(\psi_{\sigma^n(\lambda)} \vartheta_{\sigma^n(\lambda)}^+ )-\langle
\vartheta_{\sigma^n(\lambda)},\psi_{\sigma^n(\lambda)}\rangle \vartheta_\lambda^+
\end{equation}
for all $n\geq 1$. We show that $T_n \ra 0$ as $n\ra \infty$. Let $A=\big\{n\geq 0: \langle
\vartheta_{\sigma^n(\lambda)},\psi_{\sigma^n(\lambda)}\rangle=0 \big\}$. Then it follows by (\ref{C1}) that any subsequence of $\{T_n\}$, which corresponds to the set $A$, tends to $0$ as $n\ra \infty$. So to prove (\ref{proof}), we assume $\langle\vartheta_{\sigma^n(\lambda)},\psi_{\sigma^n(\lambda)}\rangle \neq 0$ for all $n\geq 0$. Now consider the sequence of currents 
$$
C_n=\frac{1}{d^n}{\langle\vartheta_{\sigma^n(\lambda)},\psi_{\sigma^n(\lambda)}\rangle}^{-1}{(H_\lambda^n)}^*(\psi_{\sigma^n(\lambda)} \vartheta_{\sigma^n(\lambda)}^+ )
$$
for $n\geq 0$. Let $\gamma$ be a limit point of the sequence $\{C_n\}$. Then by (\ref{C1}), it follows that $\lVert \gamma\rVert=1$. Further, as step $2$ of Theorem \ref{thm1}, it can be shown that $\gamma$ is a closed positive $(1,1)$-current having support in $K_\lambda^+$. So by Theorem \ref{R thm4}, $\gamma=\vartheta_\lambda^+$. This completes the proof.
\end{proof}

\begin{rem}\label{des}
We proved in Theorem \ref{thm1} that each member of $S_\lambda(\psi, T)$ is a positive multiple of $\mu_\lambda^+=2\pi \vartheta_\lambda^+$ for each $\lambda\in M$. Using a similar set of arguments as Lemma \ref{G le1}, it can be shown that for each $\lambda\in M$,
\[
\frac{1}{d^n}{(H_\lambda^n)}^*(\psi T)-\langle \psi T, \vartheta_{\sigma^n(\lambda)}^-\rangle \vartheta_\lambda^+\ra 0
\]
as $n\ra \infty$. Therefore when $\sigma$ is a homeomorphism, $S_\lambda(\psi, T)$ contains a unique point if and only if the sequence $\big\{\langle \psi T,\vartheta_{\sigma^n(\lambda)}^-\rangle\big\}$ has a unique limit point. 
\end{rem}

\subsection{Proof of Theorem \ref{G thm2}}
\begin{proof}
To show that $H$ is mixing, it is sufficient to prove that
$$
\big \langle \vartheta, \varphi(\psi\circ H^n) \big \rangle \ra \big \langle \vartheta, \varphi \big \rangle \big \langle \vartheta,\psi \big \rangle
$$
as $n\ra \infty$ where $\varphi, \psi \in C^0(M\times \mathbb{C}^2)$. Let $\varphi_\lambda$ be the restriction of $\varphi$ to $\{\lambda\}\times \mathbb{C}^2$ and define $\tilde{\varphi}(\lambda)=\langle \vartheta_\lambda, \varphi_\lambda\rangle$. Similarly, we define $\psi_\lambda$ and $\tilde{\psi}$. Now
\begin{eqnarray}
\big\langle \vartheta, \varphi(\psi\circ H^n) \big\rangle 
&=& \int_M \big\langle \vartheta_\lambda, \varphi_\lambda(\psi_{\sigma^n(\lambda)}\circ H_\lambda^n)\big\rangle \vartheta'(\lambda) \nonumber \\
&=& \int_M \big\langle {(H_\lambda^n)}^* \vartheta_{\sigma^n(\lambda)}, \varphi_\lambda(\psi_{\sigma^n(\lambda)}\circ H_\lambda^n)\big\rangle \vartheta'(\lambda)\nonumber  \\
&=& \int_M \big\langle \vartheta_{\sigma^n(\lambda)}, \psi_{\sigma^n(\lambda)} {(H_\lambda^n)}_* \varphi_\lambda \big\rangle \vartheta'(\lambda)\nonumber \\
&=&\int_M \big\langle \vartheta_{\sigma^n(\lambda)}, \psi_{\sigma^n(\lambda)} \tilde{\varphi}(\lambda)\big\rangle \vartheta'(\lambda) + \int_M \big\langle \vartheta_{\sigma^n(\lambda)},\psi_{\sigma^n(\lambda)}\big({(H_\lambda^n)}_*\varphi_\lambda-\tilde{\varphi}(\lambda) \big)\big\rangle \vartheta'(\lambda)\nonumber\\
&=&\big\langle \vartheta', \tilde{\varphi}(\tilde{\psi}\circ \sigma^n)\big\rangle + \int_M \big\langle \vartheta_{\sigma^n(\lambda)},\psi_{\sigma^n(\lambda)}\big({(H_\lambda^n)}_*\varphi_\lambda-\tilde{\varphi}(\lambda) \big)\big\rangle \vartheta'(\lambda).\label{G3}
\end{eqnarray}
Since $\sigma$ is mixing for $\vartheta'$, the first term in (\ref{G3}) tends to $\big\langle \vartheta',\tilde{\varphi}\big\rangle \big\langle \vartheta',\tilde{\psi} \big\rangle= \big\langle \vartheta,\varphi\big\rangle \big\langle \vartheta,\psi \big\rangle$ as $n\ra \infty$. Now we show that for each $\lambda\in M$,
\begin{equation}\label{G4}
\big\langle \vartheta_{\sigma^n(\lambda)}, \psi_{\sigma^n(\lambda)}\big( {(H_\lambda^n)}_*\varphi_\lambda\big)\big\rangle-\big\langle \vartheta_\lambda,\varphi_\lambda\big\rangle
\big\langle \vartheta_{\sigma^n(\lambda)},\psi_{\sigma^n(\lambda)}\big\rangle \ra 0
\end{equation} 
as $n\ra \infty$. To prove (\ref{G4}), observe that for each $\lambda\in M$ and $n\geq 1$, the term \\ $\big\langle \vartheta_{\sigma^n(\lambda)}, \psi_{\sigma^n(\lambda)}\big( {(H_\lambda^n)}_*\varphi_\lambda\big)\big\rangle$ is equal to the following:
\begin{eqnarray}
&&\int_{\mathbb{C}^2} \varphi_\lambda {(H_\lambda^n)}^*\big (\psi_{\sigma^n(\lambda)} \vartheta_{\sigma^n(\lambda)}^+\big) \wedge {(H_\lambda^n)}^*\big(\vartheta_{\sigma^n(\lambda)}^-\big)\nonumber \\
&=&\frac{1}{2\pi}\int_{\mathbb{C}^2} \varphi_\lambda (d^{-n}){(H_\lambda^n)}^*\big (\psi_{\sigma^n(\lambda)}\vartheta_{\sigma^n(\lambda)}^+\big) \wedge dd^c \mathcal{G}_\lambda^- \nonumber \\
&=& \frac{1}{2\pi}\int_{\mathbb{C}^2} (dd^c \varphi_\lambda) (d^{-n}){(H_\lambda^n)}^*(\psi_{\sigma^n(\lambda)} \vartheta_{\sigma^n(\lambda)}^+ )\wedge \mathcal{G}_\lambda^- 
+\frac{1}{2\pi}\int_{\mathbb{C}^2} (d^c \varphi_\lambda)(d^{-n}) {(H_\lambda^n)}^*(d\psi_{\sigma^n(\lambda)})\wedge \vartheta_{\sigma^n(\lambda)}^+\wedge \mathcal{G}_\lambda^- \nonumber\\
&+&\frac{1}{2\pi}\int_{\mathbb{C}^2} (d \varphi_\lambda)(d^{-n}) {(H_\lambda^n)}^*(d^c\psi_{\sigma^n(\lambda)})\wedge \vartheta_{\sigma^n(\lambda)}^+\wedge \mathcal{G}_\lambda^- 
+\frac{1}{2\pi}\int_{\mathbb{C}^2}  \varphi_\lambda(d^{-n}) {(H_\lambda^n)}^*dd^c (\psi_{\sigma^n(\lambda)})\wedge \vartheta_{\sigma^n(\lambda)}^+\wedge \mathcal{G}_\lambda^- \nonumber \\ \label{G5}
\end{eqnarray}
By using a similar result as the second part of Theorem \ref{thm1}, it follows that the last three terms of (\ref{G5}) tend to zero as $n\ra \infty$. Hence using Lemma \ref{G le1}, for each $\lambda\in M$, we get 
$$
\big\langle \vartheta_{\sigma^n(\lambda)}, \psi_{\sigma^n(\lambda)}\big( {(H_\lambda^n)}_*\varphi_\lambda\big) \big\rangle-\frac{1}{2\pi}\big\langle \vartheta_{\sigma^n(\lambda)}, \psi_{\sigma^n(\lambda)}\big\rangle  \int_{\mathbb{C}^2} (dd^c \varphi_\lambda)(\vartheta_\lambda^+)\wedge \mathcal{G}_\lambda^-\ra 0
 $$
 as $n\ra \infty$. Thus for each $\lambda\in M$,
 \begin{equation}
 \big\langle \vartheta_{\sigma^n(\lambda)}, \psi_{\sigma^n(\lambda)}\big( {(H_\lambda^n)}_*\varphi_\lambda\big)\big\rangle-\big\langle \vartheta_{\sigma^n(\lambda)}, \psi_{\sigma^n(\lambda)}\big\rangle \big\langle \vartheta_\lambda,\varphi_\lambda \big\rangle\ra 0\label{G6}
 \end{equation}
 as $n\ra \infty$. Now for each $n\geq 1$, define $R_n: M \ra \mathbb{C}$ as follows:
 $$
 R_n(\lambda):= \big\langle \vartheta_{\sigma^n(\lambda)}, \psi_{\sigma^n(\lambda)}\big( {(H_\lambda^n)}_*\varphi_\lambda\big) \big\rangle-\big\langle \vartheta_{\sigma^n(\lambda)}, \psi_{\sigma^n(\lambda)}\big\rangle \big\langle \vartheta_\lambda,\varphi_\lambda \big\rangle.
 $$
Then applying dominated convergence theorem, it follows by (\ref{G6}) that 
 $$
 \int_M R_n(\lambda)\vartheta'(\lambda)\ra 0
 $$
 as $n\ra \infty$. This completes the proof.
\end{proof}

\subsection{Proof of Proposition \ref{G pr2}}
\begin{proof}
For any two compactly supported continuous functions $\varphi$, $\psi$ in $\mathbb{C}^2$  and for a fixed $\lambda\in M$, a similar calculation as (\ref{G6}) gives 
\[
 \big\langle \vartheta_{\sigma^n(\lambda)}, \psi\big( {(H_\lambda^n)}_*\varphi\big)\big\rangle-\big\langle \vartheta_{\sigma^n(\lambda)}, \psi\big\rangle \big\langle \vartheta_\lambda,\varphi \big\rangle\ra 0
\]
as $n\ra \infty$. Now since for each $\lambda\in M$, ${(H_\lambda)}^*(\vartheta_{\sigma(\lambda)})=\vartheta_\lambda$, it follows that
\[
\big\langle \vartheta_\lambda, \big( {(H_\lambda^n)}^*\psi\big)\varphi \big\rangle-\big\langle \vartheta_{\sigma^n(\lambda)}, \psi\big\rangle \big\langle \vartheta_\lambda,\varphi \big\rangle\ra 0
\]
as $n\ra \infty$. This completes the proof.
\end{proof}

\section{A lower bound on the entropy of $\vartheta$}
This section  obtains a lower bound for the measure theoretic entropy of $H$ for $\vartheta$  and thus to get a lower bound for the topological entropy of $H$. First we recall the definitions of measure theoretic entropy and topological entropy.

\medskip

Suppose $X$ is a measure space with an invariant probability measure $m$ and $T: X\ra X$ is a measure preserving transformation. Let $\mathcal{A}=\{\mathcal{A}_1,\ldots,\mathcal{A}_k\}$ be a partition of $X$ and $\bigvee _{i=0}^{n-1}T^{-i}\mathcal{A}$ be the partition generated by $\mathcal{A},T^{-1}\mathcal{A},\ldots, T^{-n+1}\mathcal{A}$ . Then 
\[
h(T,\mathcal{A})=\lim_{n\ra \infty} \frac{1}{n} H\Big(\bigvee_{i=0}^{n-1} T^{-i} \mathcal{A} \Big)
\]
always exists and is called the entropy of $T$ with respect to $\mathcal{A}$ where 
$$
H(\mathcal{A})=- \sum_{i=1}^k m(\mathcal{A}_i)\log m(\mathcal{A}_i).
$$
The {\it{measure theoretic entropy}} of $T$ is given by $h_m(T)=\sup h(T,\mathcal{A})$ where the supremum is taken over all finite partitions $\mathcal{A}$ of $X$.

\medskip

Now let $X$ be a metric space endowed with a metric $d$ and $T$ be a continuous self map on it. For $n\geq 0$, we define metric $d_n$ on $X$ by 
$$d_n(x,y)=\max_{0\leq i \leq {n-1}} d(T^i(x), T^i(y)).
$$
For a given $n\geq 0$ and $\epsilon>0$, a set $E\subset X$ is said to be $(n,\epsilon)$-separated with respect to $T$ if for all $x,y\in E$ with $x\neq y$, we have $d_n(x,y)>\epsilon$. For each $n\geq 1$, let $s_n(\epsilon)$ denotes the cardinality of largest $(n,\epsilon)$-separated set. One can show that 
\[
s(\epsilon)=\limsup_{n\ra \infty} \frac{1}{n}\log s_n (\epsilon)
\]
exists for any $\epsilon>0$. The topological entropy of $T$ is given by
\[
h_{\text{top}}(T)=\lim_{\epsilon\ra 0}s(\epsilon).
\] 
Misiurewicz's variational principle states that 
\[
h_{\text{top}}(T)=\sup h_m (T)
\] 
where the supremum is taken over all invariant probability measures.

\subsection{Proof of Theorem \ref{thm2}}
\begin{proof}
Let $L$ be a $C^2$-smooth subharmonic function of $\lvert y \rvert$ such that $L(y)=\log \lvert y \rvert$ for $\lvert y \rvert >R$ and define $\Theta=\frac{1}{2\pi}dd^c L$ and $\Theta_{n,\lambda}={(H_\lambda^n)}^* \Theta$ for each $\lambda\in M$ and for each $n\geq 1$. Now consider the disc $\mathcal{D}= \{x=0,\vert y \vert < R\} \subset \mbb C^2$ 
and let for each $\lambda\in M$ and for each $n\geq 1$,
\[
\alpha_{n,\la}= [\mathcal{D}] \wedge \Theta_{n,\lambda}.
\]
Note that $\int_{\mathcal{D}} \alpha_{n,\lambda}=d^n$ and thus $\rho_{n,\lambda}=d^{-n} \alpha_{n,\lambda}$ is a probability measure. 

\medskip

\no
{\it Step 1:} For each $\lambda\in M$ and $n\geq 1$, define
\[
\vartheta_{n,\lambda} =\frac{1}{n} \sum_{j=0}^{n-1} \frac{d^{-n}}{2\pi} {\big(H_{\sigma^{-j}(\lambda)}^{-1} \circ \cdots \circ H_{\sigma^{-1}(\lambda)}^{-1}\big)}^*(\alpha_{n,\sigma^{-j}(\lambda)}).
\]
 {\it Claim:}  For each $\lambda\in M$, 
$
\vartheta_{n,\lambda}\ra \vartheta_\lambda
$
as $n\ra \infty$.\\
\indent
Let $\{l_n\}_{n\geq 1}$ be a sequence of natural numbers such that $l_n \ra \infty$ and $\frac{l_n}{n}\ra 0$ as $n\ra \infty$. Then for each $n\geq 1$ and for each $\lambda\in M$, $\vartheta_{n,\lambda}$ is equal to the following:
\begin{multline} \label{eq a}
\frac{d^{-n}}{2n\pi} \Big( \sum_{j=0}^{l_n} {\big(H_{\sigma^{-j}(\lambda)}^{-1} \circ \cdots \circ H_{\sigma^{-1}(\lambda)}^{-1}\big)}^*\big(\left[\mathcal{D}\right]\wedge \Theta_{n,\sigma^{-j}(\lambda)}\big)\\
 + \sum_{j=l_n+1}^{n-l_n}{\big(H_{\sigma^{-j}(\lambda)}^{-1} \circ \cdots \circ H_{\sigma^{-1}(\lambda)}^{-1}\big)}^*\big(\left[ \mathcal{D}\right]\wedge \Theta_{n,\sigma^{-j}(\lambda)}\big)\\ 
 + \sum_{j=n-l_n+1}^{n-1}{\big(H_{\sigma^{-j}(\lambda)}^{-1} \circ \cdots \circ H_{\sigma^{-1}(\lambda)}^{-1}\big)}^*\big(\left[\mathcal{D}\right]\wedge \Theta_{n,\sigma^{-j}(\lambda)}\big)\Big). 
\end{multline}
We show that the first and the third terms of the above equation tend to zero as $n\ra \infty$. Let $\chi$ be a test function, then consider
\begin{multline}
{d^{-n}}\Big\langle  \sum_{j=0}^{l_n} {\big(H_{\sigma^{-j}(\lambda)}^{-1} \circ \cdots \circ H_{\sigma^{-1}(\lambda)}^{-1}\big)}^* \big(\left[\mathcal{D}\right]\wedge\Theta_{n,\sigma^{-j}(\lambda)}\big)\\
+\sum_{j=n-l_n+1}^{n-1}{\big(H_{\sigma^{-j}(\lambda)}^{-1} \circ \cdots \circ H_{\sigma^{-1}(\lambda)}^{-1}\big)}^*\big(\left[\mathcal{D}\right]\wedge\Theta_{n,\sigma^{-j}(\lambda)}\big),\chi\Big\rangle \nonumber 
\end{multline}
which is equal to 
\begin{multline}
\sum_{j=0}^{l_n}\Big\langle {d^{-n}}{\big(H_{\sigma^{-j}(\lambda)}^{-1} \circ \cdots \circ H_{\sigma^{-1}(\lambda)}^{-1}\big)}^* \big(\left[\mathcal{D}\right]\wedge\Theta_{n,\sigma^{-j}(\lambda)}\big), \chi\Big\rangle\\
 + \sum_{j=n-l_n+1}^{n-1} \Big\langle{d^{-n}} {\big(H_{\sigma^{-j}(\lambda)}^{-1} \circ \cdots \circ H_{\sigma^{-1}(\lambda)}^{-1}\big)}^*\big(\left[\mathcal{D}\right]\wedge\Theta_{n,\sigma^{-j}(\lambda)}\big) 
,\chi\Big\rangle. \label{eq b}
\end{multline}
\no
Note that each term of (\ref{eq b}) is uniformly bounded by some fixed constant independent of $n$ and the total number of term in (\ref{eq b}) are $2 l_n$. This shows that the first and the third terms in (\ref{eq a}) tend to $0$ as $n\ra \infty$ since ${l_n}/{n}\ra 0$ as $n\ra \infty$. \\
\indent
Furthermore,
\begin{eqnarray*}
 &&d^{-n} {\big(H_{\sigma^{-j}(\lambda)}^{-1} \circ \cdots \circ H_{\sigma^{-1}(\lambda)}^{-1}\big)}^* \big(\left[\mathcal{D}\right]\wedge \Theta_{n,\sigma^{-j}(\lambda)}\big)\\
 &&= d^{-j}{\big(H_{\sigma^{-j}(\lambda)}^{-1} \circ \cdots \circ H_{\sigma^{-1}(\lambda)}^{-1}\big)}^* \left[\mathcal{D}\right] \wedge (d^{-n+j}){\big(H_{\sigma^{n-j-1}(\lambda)}\circ \cdots \circ H_\lambda\big)}^*(dd^c L)\\
 &&= \vartheta_{j,\lambda}^- \wedge dd^c \mathcal{G}_{n-j,\lambda}^+
 \end{eqnarray*}
 where $\vartheta_{j,\lambda}^-=d^{-j}{\big(H_{\sigma^{-j}(\lambda)}^{-1} \circ \cdots \circ H_{\sigma^{-1}(\lambda)}^{-1}\big)}^* \left[\mathcal{D}\right]$.\\
\indent 
 Now note that
 \begin{eqnarray}
 \frac{1}{2n\pi }\Big(\sum_{j=l_n+1}^{n-l_n} \mathcal{G}_{n-j,\lambda}^+ \vartheta_{j,\lambda}^-\Big) 
&=& \mathcal{G}_\lambda^+ \frac{1}{2n\pi } \sum_{j=l_n+1}^{n-l_n}\vartheta_{j,\lambda}^- + \frac{1}{2n\pi} \sum_{j=l_n+1}^{n-l_n}\Big( \mathcal{G}_{n-j,\lambda}^+ -\mathcal{G}_\lambda^+\Big)\vartheta_{j,\lambda}^- \nonumber\\
\label{eq c}
 \end{eqnarray}
 \no
 and since on any compact subset of $\mathbb{C}^2$, the sequence of functions $\mathcal{G}_{n-j,\lambda}^-$ converge to $\mathcal{G}_\lambda^-$ as $n-j\ra \infty$ and $\vartheta_{j,\lambda}^-$ has uniformly bounded mass, the last term of (\ref{eq c}) tends to zero as $n\ra \infty$. Since $\left[ \mathcal{D}\right]$ is a closed positive $(1,1)$-current of mass $1$ on $\mathbb{C}^2$ vanishing outside $\mathcal{D}$, an analogous result to Theorem \ref{R thm2} implies that $\vartheta_{j,\lambda}^-$ converges to $\vartheta_\lambda^-$ as $j\ra \infty$. Therefore
 \[
\lim_{n\ra \infty} \frac{d^{-n}}{2n\pi}\sum_{j=l_n+1}^{n-l_n}{\big(H_{\sigma^{-j}(\lambda)}^{-1} \circ \cdots \circ H_{\sigma^{-1}(\lambda)}^{-1}\big)}^* \big(\left[\mathcal{D}\right]\wedge \Theta_{n,\sigma^{-j}(\lambda)}\big)=\vartheta_\lambda
 \]
 for each $\lambda\in M$ and thus
\[
\lim_{n\ra \infty}\vartheta_{n,\lambda}=\vartheta_\lambda
\] 
for each $\lambda\in M$.

\medskip

\no
{\it Step 2:} For an arbitrary compactly supported probability measure $\vartheta'$ on $M$ and for each $n \ge 0$ let $\vartheta_n$ and $\rho_n$ be defined by the recipe in (\ref{mu}), i.e., for a test function
$\varphi$,
\begin{equation}
\langle \vartheta_n, \varphi \rangle = \int_M \left ( \int_{\{ \la \} \times \mbb C^2} \varphi \; \vartheta_{n, \la}  \right) \vartheta'(\la) \;\; \text{and} \;\; \langle \rho_n, \varphi \rangle  = \int_M \left (
\int_{\{ \la \} \times \mbb C^2} \varphi \; \rho_{n, \la} \right) \vartheta'(\la). 
\end{equation}

\medskip

\no
{\it Claim:} 
$$
\lim_{n\ra \infty} \vartheta_n=\vartheta \; \text{and} \; \vartheta_n=\frac{1}{2n\pi}\sum_{j=0}^{n-1}H_*^j\rho_n
$$
where $H$ is as in (\ref{homeo}).  For the first claim, note that for all test functions $\varphi$
\begin{eqnarray}
 \lim_{n\ra \infty}\langle \vartheta_n,\varphi\rangle &=& \lim_{n\ra \infty}\int_M \langle \vartheta_{n,\la},\varphi\rangle \vartheta'(\la) = \int_M \lim_{n\ra \infty}\langle \vartheta_{n,\la},\varphi\rangle \label{28}
\vartheta'(\la)\\ \notag
 &=& \int_M \langle \vartheta_\la,\varphi\rangle \vartheta'(\la) = \langle\vartheta,\varphi\rangle  
\end{eqnarray}
where the second equality follows by the dominated convergence theorem. For the second claim, note that
\begin{eqnarray}
\left \langle \frac{1}{2n\pi}\sum_{j-0}^{n-1}H^j_*\rho_n,\varphi \right \rangle 
&=& \int_M \left \langle \frac{1}{2n\pi} \sum_{j=0}^{n-1} {\big(H_{\sigma^{-j}(\lambda)}^{-1} \circ \cdots \circ H_{\sigma^{-1}(\lambda)}^{-1}\big)}^*(\rho_{n,\sigma^{-j}(\lambda)}),\varphi \right \rangle \vartheta'(\la)\nonumber\\
 &=& \int_M \langle\vartheta_{n,\la},\varphi \rangle \vartheta'(\la) = \langle \vartheta_n,\varphi\rangle.
\end{eqnarray}
Hence we get
\[
\lim_{n\ra\infty}\frac{1}{2n\pi}\sum_{j-0}^{n-1}H^j_*\rho_n=\vartheta.
\]

\medskip

\no
{\it Step 3:} Note that the support of $\vartheta$ is contained in $\text{supp}(\vartheta') \times V_R$. Let $\mathcal{P}$ be a partition of $M\times V_R$ so that the $\vartheta$-measure of the boundary of  each
element of $\mathcal{P}$ is zero and each of its elements has diameter
less than $\epsilon$. This choice is possible by Lemma 8.5 in \cite{W}. For each $n\geq 0$, define the $d_n$ metric on $M\times V_R$ by
$$
d_n(p,q)=\max_{0\leq i \leq {n-1}}d(H^i(p),H^i(q))
$$
where $d$ is the product metric on $M\times V_R$. Note that each element $\mathcal{B}$ of
$\bigvee_{j=0}^{n-1}H^{-j}\mathcal{P}$ is inside an $\epsilon$-ball in the $d_n$ metric and if $\mathcal B_{\la} = (\mathcal B \times \{ \la \}) \cap V_R$, then the $\rho_n$ measure of $\mathcal{B}$ is given by
\begin{equation*}
 \rho_n(\mathcal{B}) = \int_M {\rho_{n,\la}(\mathcal{B}_\la)}\vartheta'(\la)
 = \int_M \left ( d^{-n}\int_{{\mathcal{B}_\la}\cap \mathcal{D}} {H_\la^{n}}^{\ast} \Theta \right) \vartheta'(\la)
 = \int_M \left ( d^{-n}\int_{H_\la^n({\mathcal{B}_\la}\cap \mathcal{D})} \Theta \right ) \vartheta'(\la).
\end{equation*}
Therefore, since $\Theta$ is bounded above on $\mathbb{C}^2$, there exists $C>0$ such that
\begin{equation}
 \rho_n'(\mathcal{B})=\frac{1}{2\pi} \rho_n(\mathcal{B})\leq  C \; d^{-n}\int_M \text{Area}( H_\la^n(\mathcal{B}_\la\cap \mathcal{D})) \vartheta'(\la) \label{29}
 = C \; d^{-n} \text{Area} \left( H^n ( \mathcal{B}\cap (\mathcal{D}\times M)) \right).
 \end{equation}

\medskip

 Now we work with $X = \text{supp}(\vartheta) \subset M \times V_R$ and view $H$ as a self map of $X$. If $v^0(H,n,\epsilon)$ denotes the supremum of the areas of $H^n(\mathcal{B}\cap (\mathcal{D}\times M))$ over all $\epsilon$-balls $\mathcal{B}$, then
\begin{equation*}
\mathcal H_{ \rho_n'}\left( \bigvee_{j=0}^{n-1} H^{-j}\mathcal{P} \right) \geq -{\log \big(\frac{C}{2\pi}\big)}+n \log d -\log v^0(H,n,\epsilon)
\end{equation*}
by (\ref{29}). By appealing to Misiurewicz's variational principle as explained in \cite{BS3}, we get a lower bound for the measure theoretic entropy $h_\vartheta$ of $H$ with respect to the
measure $\vartheta$ as follows:
\begin{equation*}
h_\vartheta \geq \limsup_{n\ra \infty} \frac{1}{n}\Big(-\log\big(\frac{C}{2\pi}\big)+n \log d -\log v^0(H,n,\epsilon)\Big) \geq \log d -\limsup_{n\ra \infty} v^0(H,n,\epsilon).
\end{equation*}
By Yomdin's result (\cite{Y}), it follows that $\lim_{\epsilon\ra 0} v^0(H,n,\epsilon)=0$. Thus $h_\vartheta\geq \log d$. To conclude, note that $\text{supp}(\vartheta) \subset \mathcal J \subset
M\times V_R$ and therefore by the variational principle the topological entropy of $H$ on $\mathcal J$ is also at least $\log d$.

\end{proof}

\thispagestyle{empty}
\chapter{Dynamics of fibered maps on $\mathbb{P}^k$}
In this chapter, we study skew product of holomorphic endomorphisms of $\mathbb{P}^k$ that are fibered over compact base. We show the existence of fibered Green functions and fibered basins of attraction which are pseudoconvex and Kobayashi hyperbolic. This work can be found in \cite{KP}.
\subsection{Proof of Proposition \ref{pr4}}
\begin{proof}
By (\ref{ineq}), there exists a $C>1$ such that
\[
C^{-1} \Vert F_{\sigma^{n-1}(\la)} \circ \cdots \circ F_\la(x) \Vert^d \leq \Vert F_{\sigma^{n}(\la)} \circ \cdots \circ F_\la(x)\Vert \leq C \Vert F_{\sigma^{n-1}(\la)}\circ \cdots
\circ F_\la(x)\Vert^d
\]
for all $\la\in M$, $x\in \mathbb{C}^{k+1}$ and for all $n\geq 1$. Consequently,
\begin{equation}
\vert G_{n+1,\la}(x)-G_{n,\la}(x) \vert \leq \log C/d^{n+1}. \label{30}
\end{equation}
Hence for each  $\la\in M$,  as $n\ra\infty$, $G_{n,\la}$ converges uniformly to a continuous plurisubharmonic function $G_\la$ on $\mathbb{C}^{k+1}$. If $G_n(\la, x) = G_{n, \la}(x)$,
then (\ref{30}) shows that $G_n \ra G$ uniformly on $M \times (\mbb C^{k+1} \setminus \{0\})$.

\medskip

Furthermore, for $\la\in M$ and $c\in \mathbb{C}^*$
\begin{eqnarray}
G_\la(cx)&=&\lim_{n\ra \infty}\frac{1}{d^n}\log \Vert F_{\sigma^{n-1}(\la)}\circ \cdots \circ F_\la(cx)\Vert \nonumber\\
&=&\lim_{n\ra\infty}\left( \frac{1}{d^n}\log {|c|}^{d^n}+\frac{1}{d^n}\log
\Vert F_{\sigma^{n-1}(\la)}\circ \cdots \circ F_\la(z) \Vert \right) = \log \vert c \vert + G_\la(x).\nonumber\\
\label{31}
\end{eqnarray}
We also note that
\[
G_{\sigma(\la)}\circ F_\la(x) = d \lim_{n\ra \infty }\frac{1}{d^{n+1}}\log \Vert F_{\sigma^{n}(\la)}\circ \cdots \circ F_\la(x) \Vert = d G_\la(x) 
\]
for each $\la\in M$.

\medskip

Finally, pick $x_0 \in \mathcal A_{\la_0}$ which by definition means that $\Vert F_{\si^{n-1}(\la_0)} \circ \cdots \circ F_{\si(\la_0)} \circ F_{\la_0}(x_0) \Vert \le \ep$ for all large $n$.
Therefore $G_{n, \la_0}(x_0) \le d^{-n} \log \ep$ and hence $G_{\la_0}(x_0) \le 0$. Suppose on the contrary that $G_{\la_0}(x_0) = 0$. Note that there exists a uniform $r >
0$ such that
\[
\Vert F_{\la}(x) \Vert \le (1/2) \Vert x \Vert
\]
for all $\la \in M$ and $\Vert x \Vert \le r$. This shows that the ball $B_r$ around the origin is contained in all the basins $\cal A_{\la}$. Now $G_{\la}(0) = -\infty$ for all $\la
\in M$ and since $G_{\la_n} \ra G_{\la}$ locally uniformly on $\mbb C^{k+1} \setminus \{0\}$ as $\la_n\ra \la$ in $M$, it follows that there exists a large $C > 0$ such that
\[
\sup_{(\la, x) \in M \times \pa B_r} G_{\la}(x) \le - C. 
\]
By the maximum principle, it follows that for all $\la \in M$
\begin{equation}
G_{\la}(x) \le -C \label{32}
\end{equation}
on $B_r$. On the other hand, the invariance property $G_{\si(\la)} \circ F_{\la} = d G_{\la}$ implies that
\[
d^n G_{\la} = G_{\si^n(\la)} \circ F_{\si^{n-1}(\la)} \circ \cdots \circ F_{\la}
\]
for all $n \ge 1$. Since we are assuming that $G_{\la_0}(x_0) = 0$, it follows that
\[
G_{\si^n(\la_0)} \circ F_{\si^{n-1}(\la_0)} \circ \cdots \circ F_{\la_0}(x_0) = 0
\]
for all $n \ge 1$ as well. But $F_{\si^{n-1}(\la_0)} \circ \cdots \circ F_{\si(\la_0)} \circ F_{\la_0}(x_0)$ is eventually contained in $B_r$ for large $n$ and this means that
\[
0 = G_{\si^n(\la_0)} \circ F_{\si^{n-1}(\la_0)} \circ \cdots \circ F_{\la_0}(x_0) \le -C
\]
by (\ref{32}). This is a contradiction. Thus $\cal A_{\la} \subset \{G_{\la} < 0\}$ for all $\la \in M$.

\medskip

For the other inclusion, let $x \in \mathbb{C}^{k+1}$ be such that $G_\la(x)=-a$ for some $a>0$. This implies that for a given $\epsilon>0$ there exist $j_0$ such that
\[
-(a+\epsilon)< \frac{1}{d^j}\log \Vert F_{\sigma^{j-1}(\la)}\circ \cdots \circ F_\la(x)\Vert < -a+\epsilon
\]
for all $j\geq j_0$. This shows that $ F_{\sigma^{j-1}(\la)}\circ \cdots \circ F_\la(x) \ra 0$ as $j\ra \infty$. Hence $x\in \mathcal{A}_\la$.

\end{proof}

\subsection{Proof of Proposition \ref{pr5}}
\begin{proof}
 Recall that $\Om_{\la} = \pi(\mathcal Q_{\la})$ where $\mathcal Q_{\la} \subset \mathbb{C}^{k+1}$ is the collection of those points in a neighborhood of
which $G_{\la}$ is pluriharmonic and $\Om'_{\la} \subset \mbb P^k$ consists of those points $z \in \mbb P^k$ in a neighborhood of which the sequence
\[
\{ f_{\si^{n-1}(\la)} \circ \cdots \circ f_{\si(\la)} \circ f_{\la} \}_{n \ge 1}
\]
is normal, i.e., $\Om'_{\la}$ is the Fatou set. Once it is known that the basin $\mathcal A_{\la} = \{ G_{\la} < 0 \}$, showing
that $\Om_{\la} = \Om'_{\la}$ and that each $\Om_{\la}$ is in fact pseudoconvex and Kobayashi hyperbolic follows in much the same way as in \cite{U}. Here are the main points in the proof:

\medskip

\no {\it Step 1:} For each $\la \in M$, a point $p \in \Om_{\la}$ if and only if there exists a neighborhood $U_{\la, p}$ of $p$ and a holomorphic section $s_{\la} : U_{\la, p} \ra \mbb
C^{k+1}$ such that $s_{\la}(U_{\la, p}) \subset \pa \cal A_{\la}$. The choice of such a section $s_{\la}$ is unique upto a constant with modulus $1$.

\medskip

Suppose that $p\in \Omega_\lambda$. Let $U_{\lambda,p}$ be an open ball with center at $p$ that lies in a single coordinate chart with respect to
the standard coordinate system of $\mathbb{P}^k$. Then $\pi^{-1}(U_{\lambda,p})$ can be identified with $\mathbb{C}^{\ast} \times U_{\lambda,p}$ in canonical way and each point of
$\pi^{-1}(U_{\lambda,p})$ can be written as $(c,z)$. On $\pi^{-1}(U_{\lambda,p})$, the function $G_{\la}$ has the form
\begin{equation}
G_\la(c,z)=\log|c|+\gamma_\la(z) \label{33}
\end{equation}
by (\ref{31}). Assume that there is a section $s_\lambda$ such that
$s_\la(U_{\la,p})\subset \partial \mathcal{A}_\lambda$. Note that $s_\lambda(z)=(\sigma_\la(z),z)$ in $U_{\la,p}$ where $\sigma_\la$ is a non--vanishing holomorphic function on
$U_{\la,p}$. By Proposition \ref{pr4}, $G_\lambda\circ s_\la=0$ on $U_{\la,p}$. Thus
\[
0=G_\lambda\circ s_\lambda(z)=\log|\sigma_\lambda(z)|+\gamma_\lambda(z).
\]
Thus $\gamma_\lambda(z)=-\log \vert \sigma_\lambda(z)\vert$ is pluriharmonic on $U_{\lambda,p}$ and consequently $G_\lambda$ is pluriharmonic on
$\pi^{-1}(U_{\lambda,p})$ by (\ref{33}). On the other hand,
suppose that $\gamma_\la$ is pluriharmonic. Then there exists a conjugate function $\gamma_\la^{\ast}$ on $U_{\la,p}$
such that $\gamma_\la+i\gamma_\la^{\ast}$ is holomorphic. Define $\sigma_\la(z)=\exp (-\gamma_\la(z)-i\gamma_\la^{\ast}(z))$ and $s_\la(z)=(\sigma_\la(z),z)$.
Then $G_\la(s_\la(z))=\log |\sigma_\la(z)|+\gamma_\la(z)=0$ which shows that $s_\la(U_{\la,p})\subset \partial \mathcal{A}_\la$.

\medskip

\no {\it Step 2:} $\Om_{\la} = \Om'_{\la}$ for each $\la \in M$.

\medskip

Let $p\in \Omega_\la'$ and suppose that $U_{\la,p}$ is a neighborhood of $p$ on which there is a subsequence of
\[
\{f_{\sigma^{j-1}(\la)}\circ \cdots \circ f_\la\}_{j\geq 1}
\]
which is  uniformly convergent. Without loss of generality we may assume that
\[
g_\la = \lim_{j\ra\infty} f_{\sigma^{j-1}(\la)}\circ \cdots \circ f_\la
\]
on $U_{\la, p}$. By rotating the homogeneous coordinates $[x_0:x_1: \cdots : x_k]$ on $\mathbb{P}^k$, we may assume that
$g_\la(p)$ avoids the hyperplane at infinity $H = \big\{x_0=0\big\}$ and that $g_\lambda(p)$ is of the form $[1:g_1: \cdots : g_k]$.
Now choose an $\epsilon$ neighborhood
\[
N_\epsilon=\big\{\vert x_0 \vert < \epsilon {\big({\vert x_0 \vert}^2+ \cdots +{\vert x_k \vert}^2\big)}^{1/2} \big\}
\]
of $\pi^{-1}(H)$ in $\mathbb{C}^{k+1}\setminus \big\{0\big\}$ so that
\[
1>\epsilon {\big(1+{\vert g_1 \vert}^2+ \cdots +{ \vert g_k \vert}^2\big)}^{1/2}.
\]
Clearly $g_\lambda(p)\notin \pi(N_\epsilon)$. Shrink $U_{\la,p}$ if needed so that
\[
f_{\sigma^{j - 1}(\la)}\circ \cdots \circ f_\la (U_{\la,p})
\]
is uniformly separated from $\pi(N_\epsilon)$ for sufficiently large $l$. Define
\[
s_\la(z)=
\begin{cases}
\log  \Vert z \Vert & ;\text{ if } z\in N_\epsilon, \\
\log(\vert z_0 \vert / \vert \epsilon \vert ) & ;\text{ if } z\in \mathbb{C}^{k+1}\setminus (N_\epsilon \cup \{0\}).
\end{cases}
\]

\no Note that $0\leq s(z)-\log \Vert z \Vert \leq \log(1/\epsilon)$ which implies that
\[
d^{-{j}}s_\lambda (f_{\sigma^{j-1}(\la)}\circ \cdots \circ f_\la(z))
\]
converges uniformly to the Green function $G_\la$ as $j\ra \infty$ on $\mathbb{C}^{k+1}$. Further, if $z\in \pi^{-1}(U_{\la,p})$, then
\[
F_{\sigma^{j-1}(\la)}\circ \cdots \circ F_\la (z)\in \mathbb{C}^{k+1}\setminus (N_\epsilon \cup \{0\}).
\]
This shows that $d^{-{j}}s_\lambda(f_{\sigma^{j-1}(\la)}\circ \cdots \circ f_\la(z))$ is pluriharmonic in  $\pi^{-1}(U_{\la,p})$ and as a consequence the limit function $G_\la$ is also
pluriharmonic in $ \pi^{-1}(U_{\la,p})$. Thus $p\in \Omega_\la$.

\medskip

Now pick a point $p\in \Omega_\lambda$. Choose a neighborhood $U_{\lambda,p}$ of $p$ and a section $s_\lambda: U_{\lambda,p}\ra \mathbb{C}^{k+1}$ as in Step 1.
Since $F_{\la} : \cal A_{\la} \ra \cal A_{\si(\la)}$ is a proper map for each $\la$, it follows that
\[
(F_{\sigma^{j-1}(\lambda)}\circ \cdots \circ F_{\sigma(\lambda)}\circ
F_\lambda)(s_\lambda(U_{\lambda, p}))\subset \partial \mathcal{A}_{\sigma^j(\lambda)}.
\]
It was noted earlier that there exists a $R > 0$ such that $\Vert F_{\la}(x) \Vert \ge 2 \Vert x \Vert$ for all $\la$ and $\Vert x \Vert \ge R$. This shows that
$\mathcal{A}_\lambda\subset {B}_R$ for all $\lambda\in M$, which in turn implies that the sequence
\[
\big\{(F_{\sigma^{j-1}(\lambda)}\circ \cdots \circ F_{\sigma(\lambda)}\circ F_\lambda)\circ s_\lambda\big\}_{j\geq 0}
\]
is uniformly bounded on $U_{\la, p}$. We may assume that it converges and let $g_\lambda:U_{\lambda,p} \ra \mathbb{C}^{k+1}$ be its limit function.
Then $g_\lambda(U_{\lambda,p})\subset \mathbb{C}^{k+1}\setminus \{0\}$ since all the boundaries $\pa \cal A_{\la}$ are at a uniform distance away from the origin; indeed,
recall that there exists a uniform $r > 0$ such that the ball ${B}_r \subset \mathcal{A}_\lambda$ for all $\lambda\in M$. Thus $\pi \circ g_\lambda$ is
well-defined and the sequence
$\big\{f_{\sigma^{j-1}(\lambda)}\circ \cdots \circ f_{\sigma(\lambda)}\circ f_\lambda\big\}_{j\geq 0}$ converges to $\pi\circ g_\lambda$ uniformly on compact sets.
Thus $\big\{f_{\sigma^{j-1}(\lambda)}\circ \cdots \circ f_{\sigma(\lambda)}\circ f_\lambda \big\}_{j\geq 0}$ is a normal family in $U_{\lambda,p}$. Hence $p\in \Omega_{\lambda}'$.

\medskip

\no {\it Step 3:} Each $\Om_{\la}$ is pseudoconvex and Kobayashi hyperbolic.

\medskip

In Lemma 2.4 of \cite{U} Ueda proved that
if $h$ is a plurisubharmonic function on $\mathbb{C}^k$ and
\[
\mathcal{Q}=\big\{x\in \mathbb{C}^k: h \text{ is pluriharmonic in a neighborhood of x}\big\}
\]
is nonempty, then $\mathcal{Q}$ is pseudoconvex in $\mathbb{C}^k$. This implies that $\Omega_\lambda$ is pseudoconvex for each $\lambda\in M$.\\
\indent
Next we want to prove that $\Omega_\lambda$ is Kobayashi hyperbolic for each $\lambda\in M$. A complex manifold $M$ is called Kobayashi hyperbolic if the Kobayashi pseudo distance on it is a distance.  We recall some relevant facts about Kobayashi hyperbolicity which we use in the course of the following proof -- $(1)$ Any bounded domain in $\mathbb{C}^k$ turns out to be Kobayashi hyperbolic, $(2)$ If $M$ is Kobayashi hyperbolic and  $g:N\ra M$ is injective holomorphic map, then $N$ is Kobayashi hyperbolic, $(3)$ A  manifold $M$ is hyperbolic if its covering manifold is hyperbolic. \\
\indent
  To show that $\Omega_\lambda$ is Kobayashi hyperbolic, it suffices to prove that each component $U$ of
$\Omega_\lambda$ is Kobayashi hyperbolic. For a point $p$ in $U$ choose $U_{\lambda,p}$ and $s_\lambda$ as in Step $1$. Then $s_\lambda$ can be analytically continued to $U$. This
analytic continuation of $s_\lambda$ gives a holomorphic map $\tilde{s}_{\lambda}: \widetilde{U}\ra \mathbb{C}^{k+1}$ satisfying $\pi\circ \tilde{s}_{\lambda}=p$ where $\widetilde{U}$
is
a covering of $U$ and $p: \widetilde{U}\ra U$ is the corresponding covering map. Note that there exists a uniform $R>0$ such that $\lVert F_\lambda(z)\rVert \geq 2 \lVert z \rVert$ for
all $\lambda\in M$ and for all $z\in \mathbb{C}^{k+1}$  with $\lVert z \rVert \geq R$. Thus $\mathcal{A}_\lambda \subset B(0,R)$ and  $\tilde{s}_{\lambda}(\widetilde{U})\subset
B(0,2R)$.
Since $\tilde{s}_{\lambda}$ is injective and $B(0,2R)$ is Kobayashi hyperbolic in $\mathbb{C}^{k+1}$, it follows that $\widetilde{U}$ is Kobayashi hyperbolic. Hence $U$ is Kobayashi
hyperbolic.
\end{proof}

\thispagestyle{empty}

\bibliographystyle{amsplain}

\bibliography{ref}

\end{document}